\documentclass[11pt,a4paper]{amsart}
\pdfoutput=1
\usepackage{times}
\usepackage{amssymb}

\usepackage{ifthen}
\usepackage{cite}
\usepackage{graphicx}

\makeatletter
\newcommand{\ga}{\alpha}
\newcommand{\gb}{\beta}

\newcommand{\gep}{\epsilon}

\renewcommand{\gg}{\gamma}

\newcommand{\gl}{\lambda}
\newcommand{\go}{\omega}
\newcommand{\gs}{\sigma}

\newcommand{\vp}{\varphi}

\newcommand{\gD}{\Delta}

\newcommand{\gL}{\Lambda}
\newcommand{\gO}{\Omega}

\newcommand{\cA}{\mathcal{A}}
\newcommand{\cB}{\mathcal{B}}
\newcommand{\cC}{\mathcal{C}}
\newcommand{\cD}{\mathcal{D}}
\newcommand{\cE}{\mathcal{E}}
\newcommand{\cF}{\mathcal{F}}
\newcommand{\cG}{\mathcal{G}}

\newcommand{\cL}{\mathcal{L}}
\newcommand{\cM}{\mathcal{M}}

\newcommand{\cP}{\mathcal{P}}

\newcommand{\cX}{\mathcal{X}}
\newcommand{\cY}{\mathcal{Y}}

\newcommand{\N}{\mathbb{N}}
\newcommand{\Q}{\mathbb{Q}}
\newcommand{\R}{\mathbb{R}}

\newcommand{\Z}{\mathbb{Z}}

\newtheorem{theorem}{Theorem}[section]

\newtheorem{lemma}[theorem]{Lemma}
\newtheorem{corollary}[theorem]{Corollary}

\newtheorem{definition}[theorem]{Definition}

\theoremstyle{remark}

\newtheorem{examples}[theorem]{Examples}
\newtheorem{remark}[theorem]{Remark}

\newcommand{\erf}{\mathop{\operator@font erf}\nolimits}
\newcommand{\erfc}{\mathop{\operator@font erfc}\nolimits}
\newcommand{\sign}{\mathop{\operator@font sign}\nolimits}

\newenvironment{enum_a}
    {\begin{enumerate}}
    {\end{enumerate}}
\newenvironment{enum_i}
    {\begin{enumerate}}
    {\end{enumerate}}

\newif\if@golden  \@goldentrue
\newcommand{\f@ctor}{1}
\newlength{\aiv@width}  
\newlength{\aiv@height} 
\newlength{\tmp@width}  
\newlength{\tmp@height} 
\if@twocolumn\if@golden
  \else\fi
\else\if@golden\ifcase\@ptsize\relax\or
\or\fi
  \else\ifcase\@ptsize\relax\or
\or\fi\fi
\if@twocolumn\if@golden\ifcase\@ptsize\relax\renewcommand{\f@ctor}{53}
  \or\renewcommand{\f@ctor}{46}\or\renewcommand{\f@ctor}{43}\fi
  \else\ifcase\@ptsize\relax\renewcommand{\f@ctor}{51}\or
  \renewcommand{\f@ctor}{45}\or\renewcommand{\f@ctor}{42}\fi\fi
\else\if@golden\ifcase\@ptsize\relax \renewcommand{\f@ctor}{46}
  \or\renewcommand{\f@ctor}{43}\or\renewcommand{\f@ctor}{43}\fi
  \else\ifcase\@ptsize\relax\renewcommand{\f@ctor}{43}
  \or\renewcommand{\f@ctor}{40}\or\renewcommand{\f@ctor}{40}\fi\fi\fi
\multiply\textheight by \f@ctor
\let\comp\circ

\newcommand{\sgl}{\ensuremath\sqrt{2\gl}}

\newcommand{\da}{\ensuremath\downarrow}
\newcommand{\ua}{\ensuremath\uparrow}

\newcommand{\cond}{\ensuremath\,\big|\,}

\newcommand{\eval}{\mathop{\big|}\nolimits}

\newcommand{\FX}{\ensuremath \cF^X}
\newcommand{\CD}{\ensuremath C_{\gD}}
\newcommand{\CDR}{\ensuremath \CD(\R_+)}

\newcommand{\Ieqref}[1]{\textup{\tagform@{I.\ref{I_#1}}}}

\newcommand{\IIeqref}[1]{\textup{\tagform@{II.\ref{II_#1}}}}

\newcommand{\Rbp}{\ensuremath\overline{\R}_+}
\newcommand{\hgO}{\ensuremath{\hat\gO}}
\newcommand{\hgo}{\ensuremath{\hat\go}}
\newcommand{\hcM}{\ensuremath{\hat\cM}}

\newcommand{\hP}{\ensuremath{\hat P}}
\newcommand{\hE}{\ensuremath{\hat E}}

\newcommand{\hcA}{\ensuremath{\hat \cA}}
\newcommand{\hcF}{\ensuremath{\hat \cF}}
\newcommand{\hPx}{\ensuremath{\hat P}_x}
\newcommand{\hEx}{\ensuremath{\hat E}_x}
\newcommand{\hth}{\ensuremath{\hat\theta}}
\newcommand{\hX}{\ensuremath{\hat X}}
\newcommand{\hT}{\ensuremath{\hat T}}
\newcommand{\hgL}{\ensuremath{\hat\gL}}

\newcommand{\zt}[1][t]{\ensuremath\{\zeta>#1\}}
\newcommand{\LtR}[1][\gL_t]{\ensuremath #1\times\Rbp}
\newcommand{\LR}{\ensuremath(\LtR[\gL])}
\newcommand{\hMT}{\ensuremath\hat\cM_{\hat T}}
\newcommand{\Xil}[1][n-1]{\ensuremath\Xi^{\le #1}}
\newcommand{\cCl}[1][n-1]{\ensuremath\cC^{\le #1}}

\newcommand{\Qxl}{\ensuremath Q_x^{\le n-1}}
\newcommand{\Xiu}[1][n]{\ensuremath\Xi^{\ge #1}}
\newcommand{\cCu}[1][n]{\ensuremath\cC^{\ge #1}}
\newcommand{\Qu}[1][n]{\ensuremath Q^{\ge #1}}

\newlength{\BCs@ze}
\newlength{\BCsh@ft}
\ifcase\@ptsize\relax
\or
\or
\fi
\DeclareFixedFont\MT{OMS}{cmsy}{m}{n}{\BCs@ze}    % standard 'Computer Modern' symbols
\newcommand{\BigCart}{\ensuremath\mathop{\raisebox{\BCsh@ft}{{\MT\char"02}}}}
\makeatother

\numberwithin{equation}{section}
\allowdisplaybreaks[2]

\date{December 6, 2010}

\title[Brownian Motions on Intervals Revisited]
{Brownian Motions on Metric Graphs:\\
{\small Feller Brownian Motions on Intervals Revisited}}

\dedicatory{Dedicated to the memory of Pierre Duclos}

\author[V.~Kostrykin]{Vadim Kostrykin}
\address{Vadim Kostrykin\newline
Institut f\"ur Mathematik\newline
Johannes Gutenberg--Universit\"at\newline
D--55099 Mainz, Germany}
\email{kostrykin@mathematik.uni-mainz.de}

\author[J.~Potthoff]{J\"urgen Potthoff}
\address{J\"urgen Potthoff\newline
Institut f\"ur Mathematik\newline
Universit\"at Mannheim\newline
D--68131 Mann\-heim, Germany}
\email{potthoff@math.uni-mannheim.de}

\author[R.~Schrader]{Robert Schrader}
\address{Robert Schrader\newline
Institut f\"{u}r Theoretische Physik\newline
Freie Universit\"{a}t Berlin, Arnimallee~14\newline
D--14195 Berlin, Germany}
\email{schrader@physik.fu-berlin.de}

\copyrightinfo{2010}{V.~Kostrykin, J.~Potthoff, R.~Schrader}

\begin{document}
\begin{abstract}
The construction of the paths of all possible Brownian motions
(in the sense of~\cite{Kn81}) on a half line or a finite interval
is reviewed.
\end{abstract}

\maketitle
\tableofcontents
\thispagestyle{empty}

\section{Introduction} \label{sect_intro}
In recent years there has been a growing interest in metric graphs as the underlying
structure for models in many branches of science, for example in physics, biology,
chemistry, engineering and computer science, to name just a few. The interested
reader is referred to the review~\cite{Ku04} and to the articles in the
volume~\cite{BrEx09}. Metric graphs are piecewise linear varieties with
singularities at a finite number of points, namely at the vertices, and they can be
thought of as a finite collection of finite intervals or half lines, which are glued
together at some of their endpoints.

There exists an extensive amount of literature on Laplace operators and their
semigroups on metric graphs, see e.g., \cite{KoSc99, KoSc00, KoSc06a, KoPo07d,
KoPo09c} and the literature quoted there. In view of this, it is natural to
investigate Brownian motions on metric graphs, and in this context we also want to
mention the articles~\cite{BaCh84, DeJa93, FrWe93, FrSh00, Fr94, Gr99, Kr95} which
deal with various aspects of stochastic processes on (metric) graphs.

The present article is one of four articles (cf.~\cite{BMMG1, BMMG2, BMMG3})
of the authors in which we address the problem of the characterization of all
Brownian motions (the precise definition of this class of stochastic processes is
given in~\cite{BMMG1}), and their pathwise construction on a given metric graph.
Thus, for metric graphs we consider --- with certain  restrictions (see~\cite{BMMG1}
and below) --- the analogue of the problem raised by Feller in his pioneering
articles~\cite{Fe52, Fe54, Fe57a, Fe58}, which have to be considered together with
the work of It\^o and McKean~\cite{ItMc63, ItMc74}, and the article~\cite{We56} of
Wentzell. For later accounts of this subject we also refer to~\cite{DyJu69}
and~\cite{Kn81}.

Of course, according to our above description, the simplest metric graphs are given
by a half line as $\R_+$ or a compact interval, say, $[0,1]$, that is, the cases
considered in the aforementioned classical literature. One of the intentions of the
present article is to bring the material presented there into a form which is suitable
for a generalization to metric graphs. The other intention is to give a pedagogical
(and rather detailed) account of the construction of Feller Brownian motions on
intervals. A very readable treatment can be found in the book~\cite{Kn81} by Knight.
However, there some arguments are only hinted at, others are not so easy to follow
(at least for the present authors). We have tried to make the article self-contained,
which means that several of our arguments are rather well-known. On the other hand,
we shall give a number of arguments and
computations which to the best of our knowledge are new. For example, we shall
compute all transition and resolvent kernels explicitly, and  in two cases we
provide rather simple calculations of the generators based on Dynkin's
formula~\cite{Dy65a}. These calculations carry through to the case of metric graphs,
and this will make it possible to compare the resolvents obtained with those found
in~\cite{KoSc99, KoSc00}. This in turn allows --- at least in certain cases --- to
rediscover from a stochastic point of view a central entity of quantum theory,
namely the \emph{scattering matrix}.

For the definition of a Brownian motion on an interval we shall follow
Knight~\cite{Kn81} (in~\cite{BMMG1} this definition is extended to metric
graphs). Let $I$ be a finite or semi-finite interval, and let $X$ denote a
stochastic process with values in $I\cup\{\gD\}$, where $\gD$ is a cemetery point.
Moreover, let $B$ be a standard one dimensional Brownian motion. We denote by $\hat
X$, $\hat B$, the processes $X$, $B$ respectively, with absorption (i.e., stopping)
at the endpoint(s) of $I$.

\begin{definition}  \label{def_BM}
A Brownian motion $X$ on a finite or semi-infinite interval is a normal strong
Markov process on $I\cup\{\gD\}$, a.s.\ with right continuous paths, continuous
paths up to its lifetime, and such that the stochastic process~$\hat X$ is
equivalent to~$\hat B$.
\end{definition}

We remark that this definition implies that we consider a considerably smaller class
of stochastic processes than has been treated in the quoted work by Feller and
It\^o--McKean, in that there the paths are not required to be continuous up to the
lifetime. In particular, there jumps from an endpoint back into the interval are allowed.
On the other hand, this restriction will allow us in the present series of articles
to work within the class of Feller processes, which has advantages concerning the control
of the strong Markov property. Some of the consequences the restriction we impose
will also be discussed below. Within the more general framework of metric graphs the
assumption that the paths are continuous up to their lifetime will be removed in
a forthcoming work.

For the sake of definiteness, from now on we shall only consider intervals $I$ of
the form $[0,+\infty)$ or $[0,1]$. $C_0([0,+\infty))$ denotes the space of real
valued continuous functions on $[0,+\infty)$ vanishing at infinity, while $C([0,1])$
is the space of real valued continuous functions on $[0,1]$. By $C_0(I)$ we mean
either of the two spaces, and we make the convention that every $f\in C_0(I)$ is
extended to $I\cup\{\gD\}$ by setting $f(\gD)=0$. $C_0^2(I)$ denotes the subspace of
$C_0(I)$ consisting of those functions $f$ in $C_0(I)$ which are twice continuously
differentiable in the interior of $I$, and are such that the second derivative $f''$
extends to a
continuous function on $I$. It is not difficult to check that for $f\in C^2_0(I)$
the derivative $f'$ has a (finite) limit $f(0+)$ from the right at $0$, and --- if
$I=[0,1]$ --- also a (finite) limit $f(1-)$ from the left at $1$. Then for Brownian
motions on $I$ as in definition~\ref{def_BM} Feller's theorem is the following
statement.

\begin{theorem} \label{Feller-bc}
Assume that $X$ is a Brownian motion on $I$. Then the generator $A$ of its semigroup
on $C_0(I)$ is given by $A f = 1/2 f''$ with a domain $\cD(A)$ being a subset of
$C_0^2(I)$. Moreover the following holds true:
\begin{enum_a}
    \item   Suppose that $I=[0,+\infty)$. Then there exist $a_0$, $b_0$,
            $c_0\in [0,1]$ with $a_0+b_0+c_0 = 1$, $a_0\ne 0$, such that every
            $f\in\cD(A)$ satisfies the Wentzell boundary condition at the
		   origin	
            \begin{equation}    \label{Wbc_0}
                a_0 f(0) - b_0 f'(0+) + \frac{c_0}{2}\,f''(0+) = 0,
            \end{equation}
            and the domain $\cD(A)\subset C_0^2([0,+\infty))$ is uniquely determined
            by this boundary condition.
    \item   Suppose that $I=[0,1]$. Then there exist $a_v$, $b_v$,
            $c_v\in [0,1]$, $v=0$, $1$, with $a_v+b_v+c_v = 1$, $a_v\ne 0$,
            such that every $f\in\cD(A)$ satisfies the boundary condition~\eqref{Wbc_0},
            supplemented by the boundary condition
            \begin{equation}    \label{Wbc_1}
                a_1 f(1) + b_1 f'(1-) + \frac{c_1}{2}\,f''(1-) = 0.
            \end{equation}
            The domain $\cD(A)\subset C^2([0,1])$ is uniquely determined
            by the boundary conditions~\eqref{Wbc_0}, \eqref{Wbc_1}.
\end{enum_a}
\end{theorem}

Furthermore the following statement holds true:

\begin{theorem} \label{Feller_prop}
Every Browian motion on $I$ in the sense of definition~\ref{def_BM} is a Feller
process.
\end{theorem}

Proofs of these theorems can be found in~\cite{Kn81}. For the case of metric graphs
proofs of similar statements are provided in~\cite{BMMG1}.

We remark that the boundary conditions determined by a Brownian motion in the sense
of definition~\ref{def_BM} are \emph{local}, because each of the
conditions~\eqref{Wbc_0}, \eqref{Wbc_1} only involves the values of the function $f$
and its derivatives at one or the other endpoints of the interval. In fact, this is
a consequence of our requirement that the paths of the Brownian motion be continuous
up to their lifetime. More general boundary conditions arise if one allows for jumps
from the endpoints back into the interval. They have been discussed in the above
quoted work by Feller and by It\^o--McKean, see also \cite{DyJu69, Kn81}.

In order to explain the ideas of how to construct such Brownian motions pathwise from
a standard Brownian motion $B$, for the choice $I=\R_+$ let us first discuss the
``pure'' cases, where only one of the terms in equation~\eqref{Wbc_0} is present.

First consider the choice $a_0=1$, $b_0=c_0=0$, i.e., the Dirichlet boundary
condition $f(0)=0$. (This condition is actually not permitted according to the statement of
the theorem, but temporarily we shall consider it nevertheless, see also below.)
Then this boundary condition can be implemented by killing the standard Brownian motion
when it reaches the origin, since by our convention $f(\gD)=0$. The reason that this
case has to be excluded in theorem~\ref{Feller-bc} is that due our definition, a
Brownian motion $X$ on $\R_+$ is such that $\hat X$ is equivalent to $\hat B$, and
therefore has to be able to reach $0$, which is not possible for a right continuous
process which is killed upon reaching $0$. Moreover, a process which is killed upon
reaching $0$, obviously cannot be a Feller process.

The choice $a_0=0$, $b_0=1$, $c_0=0$ is the Neumann boundary
condition $f'(0+)=0$, and it is well-known (e.g., \cite[p.~40]{ItMc74}, cf.\ also
section~\ref{sect_el}) that this can be obtained from a Brownian motion with reflection
at $0$.

Finally, it we choose $a_0=b_0=0$, $c_0=1$, then we find the Wentzell boundary condition
$f''(0+)=0$. Obviously, this means that the generator annihilates the function at $0$,
i.e., the semigroup acts trivially there. In terms of the underlying stochastic process
this means that it stops to move there, and so one can implement this boundary condition
with the Brownian motion $\hat B$ with absorption in $0$.

For the general case~\eqref{Wbc_0} one has to construct a stochastic process which
combines all these effects. According to~\cite[Section~2]{ItMc63}, it was Feller who
suggested how to do this (at least for the \emph{elastic case} $a_0\ne 0$, $b_0\ne
0$, $c_0=0$): For a reflecting Brownian motion one builds in the effects of killing
and absorption \emph{on the scale of the local time at $0$}. More precisely, for the
Brownian motion on $\R_+$ with all three coefficients in~\eqref{Wbc_0} non-zero, one
uses the local time of the reflecting Brownian motion to define a new stochastic
time scale which slows down the process when it is at the origin. This makes the
origin ``sticky'' for the process, and replaces the full absorption in the ``pure''
case $c_0=1$. Then, on this new stochastic time scale the process is killed
exponentially to produce a term involving $f(0)$ like in the ``pure'' Dirichlet
case. These ideas have been carried out in~\cite{ItMc63, ItMc74},  cf.\
also~\cite{DyJu69, Kn81}. The case of an elastic Brownian motion has also been
discussed in~\cite{KaSh91, Wi79}.

Before we end this introduction with an overview of the organization of this article
let us quickly settle one trivial case of a Brownian motion on $\R_+$, namely when
$b_0=0$. The subcase $a_0=0$ has already been treated above: the Brownian motion
$\hat B$ with absorption realizes this boundary condition. So consider $a_0\ne 0$,
$c_0\ne 0$, and set $\gb = a_0/c_0$. Define a stochastic process $X$ as follows:
Start with a Brownian motion at $x\ge 0$, and when hitting the origin keep the
process there for an independent exponential holding time of rate $\gb$. After
expiration of this holding time let the process jump to the cemetery point $\gD$. It
is straightforward to check that this process is a Brownian motion on $\R_+$ in the
sense of definition~\ref{def_BM}, that it has $1/2$ times the Laplacean as the
generator, and that its semigroup $U$ acts at the origin as $U_t f(0) = \exp(-\gb t)
f(0)$, $t\ge 0$. Thus this process implements the desired boundary condition $\gb
f(0) + 1/2 f''(0+) =0$.

The organization of the article is as follows. In
sections~\ref{sect_el} -- \ref{sect_gen} we consider the interval $I=\R_+$.
\emph{Elastic Brownian motion} ($a_0\ne 0$, $b_0\ne 0$, $c_0=0$) on $\R_+$ is
constructed and analyzed in section~\ref{sect_el}. In section~\ref{sect_st} the
Brownian motion with \emph{sticky origin} ($a_0=0$, $b_0\ne 0$, $c_0\ne 0$) is
considered, and in section~\ref{sect_gen} we treat the general case where all three
coefficients are non-zero. In section~\ref{sect_fii}, following an idea of
Knight~\cite{Kn81}, we use the stochastic processes constructed before on the
intervals $(-\infty,1]$ and $[0,+\infty)$ to piece together a stochastic process on
$[0,1]$, which generates the boundary conditions~\eqref{Wbc_0}, \eqref{Wbc_1}. The
most difficult point is the verification of the (simple) Markov property of this
process, and some parts of the rather lengthy and technical proof are deferred to an
appendix. There are also appendices about killing a process with a perfect
homogeneous functional (included here, because we use the same method of killing as
in~\cite{Kn81, KaSh91}, for which we need some results which are not easily
available in the standard literature), about some results related to the Brownian
local time, and about general heat kernels and their Laplace transforms.

\vspace{1.5\baselineskip}
\noindent
\textbf{Acknowledgement.}
The authors thank Mrs.~and Mr.~Hulbert for their warm hospitality at the
\textsc{Egertsm\"uhle}, Kiedrich, where part of this work was done.
J.P.\ gratefully acknowledges fruitful discussions with O.~Falkenburg,
A.~Lang and F.~Werner. R.S.~thanks the organizers of the \emph{Chinese--German
Meeting on Stochastic Analysis and Related Fields}, Beijing, May 2010,
where some of the material of this article was presented.

\section{Elastic Brownian Motion} \label{sect_el}
Throughout this article we suppose that $(\gO,\cA)$ is a measurable space endowed
with a family of probability measures $P=(P_x,\,x\in\R)$, a right continuous
filtration $\cF=(\cF_t,\,t\ge 0)$ which is complete with respect to $P$, and a
standard one-dimensional Brownian motion $B=(B_t,\,t\ge 0)$ relative to $\cF$, so
that $P_x(B_0=x)=1$ for every $x\in\R$. Without loss of generality we may assume that
$(\gO,\cA)$ is equipped with a family $\theta = (\theta_t,\,t\ge 0)$ of shift
operators for $B$: $B_s\comp \theta_t = B_{s+t}$ for all $s$, $t\ge 0$. In
particular, our assumptions mean that $B$ is a normal strong Markov process relative
to $\cF$. As usual, we set $\cF_\infty = \gs(\cF_t,\,t\ge 0)$.

The local time of $B$ at the origin will be denoted by $L=(L_t,\,t\ge 0)$, and we
choose its normalization so that Tanaka's formula (e.g., \cite[p.~205]{KaSh91},
\cite[p.~207]{ReYo91}, \cite[p.~68]{Mc69}) holds in the form
\begin{equation*}
    d|B_t| = \sign(B_t)\,dB_t + dL_t,\qquad t\ge 0.
\end{equation*}
It will be convenient to assume that both, $B$ and $L$, have exclusively continuous
paths.

Figure~\ref{fig1} shows a simulated path of a Brownian motion on the real line~(blue)
and its local time at the origin~(red).%
\footnote{\footnotesize
    All simulations of this article have been done with {\scshape SciLab}~5.2.1.
}\space
All other paths in the figures below will be constructed from this path. Of course,
the figures of these paths can only be considered as caricatures of the ``true'' paths.
\begin{figure}
\begin{center}
    \includegraphics{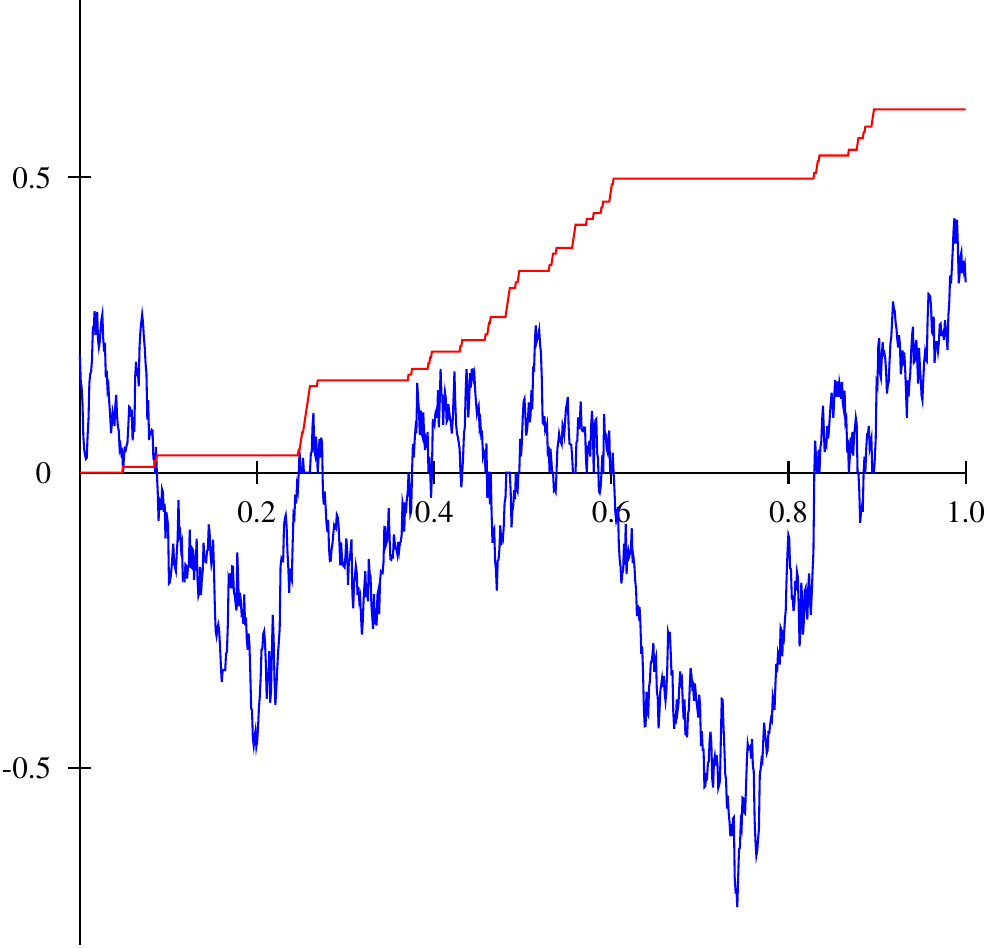}\\
    \caption{Simulation of a path of a Brownian motion (blue) and its local time
    at the origin (red).}
    \label{fig1}
\end{center}
\end{figure}

We shall often have occasion to use the right continuous pseudo-inverse
of $L$ denoted by $K$:
\begin{equation}    \label{el1}
    K_r = \inf\,\{t\ge 0,\,L_t>r\},\qquad r\in\R_+,
\end{equation}
with the usual convention $\inf\emptyset=+\infty$. The law of $K_r$, $r>0$, under
$P_0$ is computed in appendix~\ref{app_LT} (cf.\ lemma~\ref{a_LT_lem1}). Since for fixed
$r\ge 0$, $K_r$ is the moment when the local time $L$ increases above level $r$, and
$L$ only grows when $B$ is at the origin, we find that $B_{K_r}=0$, $r\in\R_+$. It
is not hard to show that the continuity of $L$ entails $L_{K_r}=r$,
$r\in\R_+$, and for all $r$, $t\ge 0$,
\begin{equation}    \label{el2}
    \{K_r\ge t\} = \{L_t\le r\}.
\end{equation}
In particular, the last equation shows that for every $r\ge 0$, $\{K_r<t\}$ belongs
to $\cF_t$, and since $\cF$ is right continuous, $K_r$ is an $\cF$--stopping time
for all $r\ge 0$. For later purposes we also remark here that due to its right
continuity $K$ is a measurable stochastic process.

Let us consider the reflecting Brownian motion $|B_t|$, $t\ge 0$ (cf.\
figure~\ref{fig2}).
\begin{figure}[ht]
\begin{center}
    \includegraphics{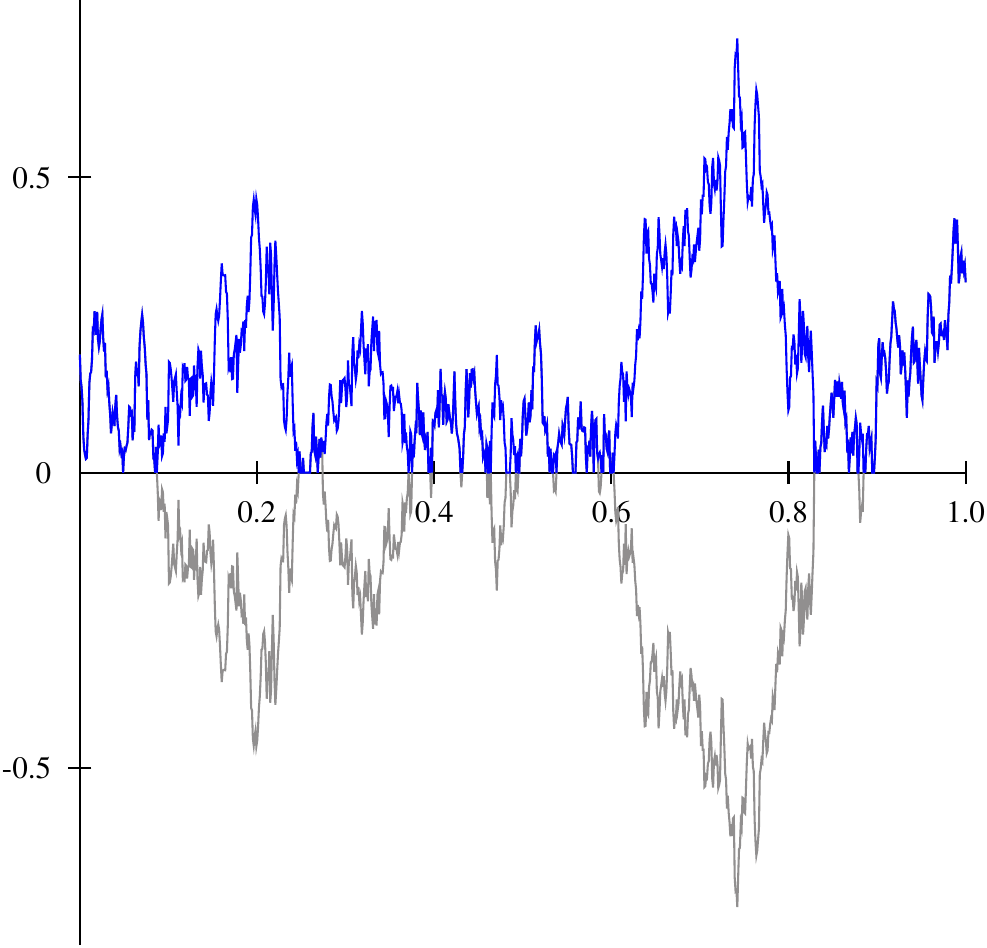}\\[0\baselineskip]
    \caption{A path of a Brownian motion (grey) and the associated reflected path (blue)}
    \label{fig2}
\end{center}
\end{figure}
Clearly, the transition semigroup of the reflecting Brownian motion has the
integral kernel
\begin{equation}    \label{el3}
    p^N(t,x,y) = p(t,x,y) + p(t,x,-y),\qquad t\ge0,\,x,\,y\in\R_+,
\end{equation}
where $p(t,x,y)$ denotes the usual heat kernel on the line, i.e.,
\begin{equation}    \label{el4}
    p(t,x,y) = \frac{1}{\sqrt{2\pi t}}\,e^{-(x-y)^2/2t},\qquad t\ge0,\,x,\,y\in\R_+.
\end{equation}
Hence the resolvent $R^N=(R^N_\gl,\,\gl>0)$ of the reflecting Brownian motion
has the integral kernel
\begin{equation}    \label{el5}
    r^N_\gl(x,y)
        = \frac{1}{\sgl}\,\Bigl(e^{-\sgl|x-y|}+e^{-\sgl(x+y)}\Bigr),
                \qquad \gl>0,\,x,\,y\in\R_+.
\end{equation}
It is now a straightforward calculation to show that for every $f\in C_0(\R_+)$
we have
\begin{equation}    \label{el6}
    (R^N_\gl f)'(0+) = 0,\qquad \gl>0,
\end{equation}
i.e., the generator of reflecting Brownian motion is $1/2$ times the second
derivative with Neumann boundary condition at the origin.

\emph{Elastic Brownian motion} $B^e$ is the stochastic process defined by killing
the reflecting Brownian exponentially on the scale of the local time at the origin,
cf., e.g., \cite[p.~45]{ItMc74}, \cite[p.~158]{Kn81}, \cite[p.~425~f.]{KaSh91}.
Here we shall follow the construction given in \cite{Kn81, KaSh91}, which is slightly
different from the one in \cite{ItMc74}, or in --- in the more general form of
killing on the scale of a \emph{perfect continuous additive homogeneous functional}
--- \cite{BlGe68, Wi79}. For a more detailed account the interested reader is also
referred to appendix~\ref{app_killing}.

Let $\gb>0$, and introduce the auxiliary probability space $(\R_+,\cB(\R_+),P_\gb)$
where $P_\gb$ is the exponential law of rate $\gb$. Denote by $S$ the random
variable $S(s) = s$, $s\in\R_+$. Now form for every $x\in\R$ the product space of
$(\gO,\cA,P_x)$ and $(\R_+,\cB(\R_+),P_\gb)$, and denote the resulting probability
space by $(\hgO,\hcA,\hP_x)$. $S$ and all random variables on $(\gO,\cA,\cP_x)$,
$x\in\R$, are extended in the trivial way to the product space, and we denote the
extended random variables by the same symbols.

On $(\hgO,\hcA)$ define a random time $\zeta_\gb$ by
\begin{equation}    \label{el7}
    \zeta_\gb = \inf\,\{t\ge 0,\, L_t > S\}.
\end{equation}
Observe that we can write
\begin{equation}    \label{el8}
    \zeta_\gb = K_S,
\end{equation}
because $K$ is a measurable process. Now set
\begin{equation}    \label{el9}
    B^e_t = \begin{cases}
                |B_t|,   & t < \zeta_\gb,\\
                \gD,     & t \ge \zeta_\gb,
            \end{cases}
\end{equation}
where $\gD$ is a cemetery state.
\begin{figure}[ht]
\begin{center}
    \includegraphics{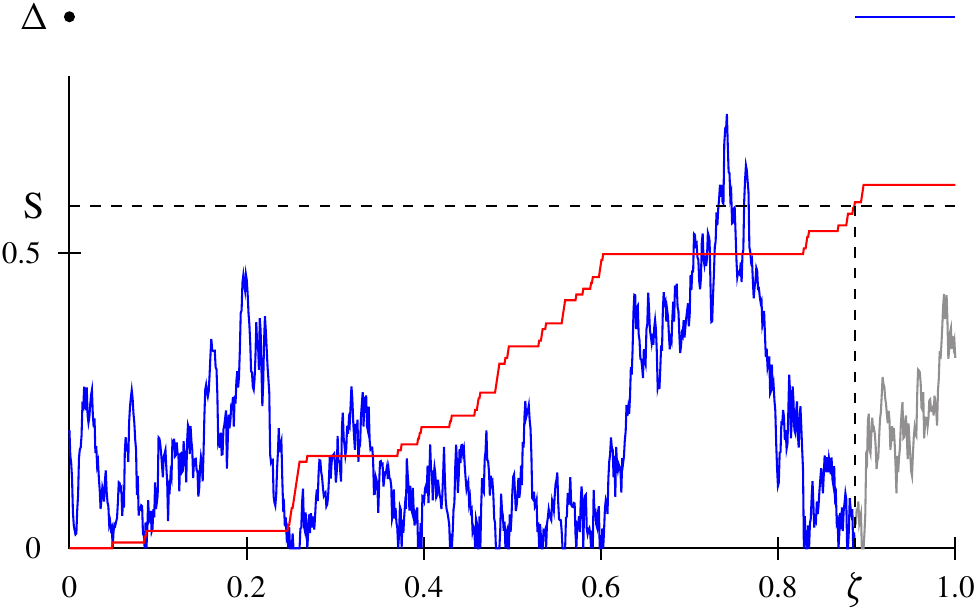}\\[0\baselineskip]
    \caption{Construction of elastic Brownian motion}  \label{fig3}
\end{center}
\end{figure}
In figure~\ref{fig3} the local time of the path of a reflected Brownian motion (in
grey) is drawn in red, the value of $S$ is $0.58$ (horizontal dashed line). The time
of killing the reflected Brownian motion for the corresponding path of elastic
Brownian motion is at $\zeta_\gb = 0.89$ (vertical dashed line). The path of the
elastic Brownian motion is the blue piece of the depicted path, extended to be equal
to $\gD$ after time $\zeta_\gb$.

As is shown in appendix~\ref{app_killing}, $(\hgO,\hcA)$ can be equipped with a
right continuous filtration $\hcF=(\hcF_t,\,t\ge 0)$ (there denoted by
$(\hcM_t,\,t\ge 0)$) which is complete with respect to the family
$(\hP_x,\,x\in\R_+)$, such that $B^e$ is $\hcF$--adapted and a strong Markov process
relative to $\hcF$. In particular, $B^e$ is a strong Markov process relative to its
natural filtration. Moreover, the lifetime $\zeta_\gb$ is a stopping time for
$\hcF$.

In the sequel we make the convention that every real valued function on $\R_+$ is
extended to $\R_+\cup\{\gD\}$ via $f(\gD)=0$.

Now we compute the boundary condition of the generator of the elastic Brownian
motion. Our calculation is based on Dynkin's formula (cf., e.g.,
\cite[p.~133]{Dy65a}, \cite[p.~99]{ItMc74}, \cite[p.~131]{Wi79}) for the generator
$A^e$ of the elastic Brownian motion $B^e$.

\begin{theorem} \label{el_thm1}
The domain $\cD(A^e)$ of the generator $A^e$ of the elastic Brownian motion $B^e$
of parameter $\gb>0$ is equal to the space of functions $f\in C_0^2(\R_+)$ so that
\begin{equation}    \label{el10}
    f'(0+) = \gb f(0)
\end{equation}
holds.
\end{theorem}

\begin{proof}
By construction of $B^e$ the origin is not an absorbing point for $B^e$. Thus it is
permissible to apply Dynkin's formula at the origin: for every $f\in C_0^2(\R_+)$,
\begin{equation}    \label{el11}
    A^ef(0) = \lim_{\gep\downarrow 0}
                \frac{E_0\bigl(f(B^e_{\tau_\gep})\bigr)-f(0)}{E_0(\tau_\gep)},
\end{equation}
where $\tau_\gep$, $\gep>0$, is the hitting time of the complement of the interval
$[0,\gep)$, i.e., of the set $[\gep,+\infty)\cup \{\gD\}$, by $B^e$. Thus $\tau_\gep
= H_\gep\land \zeta_\gb$, where $H_\gep$ is the hitting time of the point $\gep$
on the real axis by the reflecting Brownian motion $|B|$. Write
\begin{align*}
    E_0\bigl(f(B^e_{\tau_\gep})\bigr)
        &= E_0\bigl(f(B^e_{\tau_\gep}); H_\gep<\zeta_\gb\bigr)
                + E_0\bigl(f(B^e_{\tau_\gep}); H_\gep\ge \zeta_\gb\bigr)\\
        &= f(\gep)\,P_0(H_\gep<\zeta_\gb) + f(\gD)\,P_0(H_\gep\ge\zeta_\gb)\\
        &= \frac{1}{1+\gep\gb}\,f(\gep),
\end{align*}
where we used lemma~\ref{a_LT_lem4}, and the convention $f(\gD)=0$. Thus we obtain
\begin{align*}
    A^e f(0)
        &= \lim_{\gep\downarrow 0} \frac{1}{E_0(\tau_\gep)}
                \Bigl(\frac{1}{1+\gep\gb}\bigl(f(\gep)-f(0)\bigr)-\frac{\gep\gb}{1+\gep\gb}\,f(0)\Bigr)\\
        &= \lim_{\gep\downarrow 0} \frac{\gep}{E_0(\tau_\gep)}
                \frac{1}{1+\gep\gb}\,\Bigl(\frac{f(\gep)-f(0)}{\gep}-\gb f(0)\Bigr).
\end{align*}
We know that the limit on the right hand side exists, and clearly
\begin{equation*}
    \lim_{\gep\downarrow 0}\frac{1}{1+\gep\gb}\,\Bigl(\frac{f(\gep)-f(0)}{\gep}
            -\gb f(0)\Bigr) =f'(0+)-\gb f(0).
\end{equation*}
On the other hand we have $E_0(\tau_\gep)\le E_0(H_\gep)$, and the latter
expectation is equal to $\gep^2$ (e.g., \cite[Sect.~1.7, Problem~6]{ItMc74},
\cite[Chap.~II, Problems~5, 18]{DyJu69}, \cite[Example~III.24.3]{Wi79}), so that
\begin{equation*}
    \frac{\gep}{E_0(\tau_\gep)} \ge \frac{1}{\gep}.
\end{equation*}
Therefore we must have $f'(0+)-\gb f(0)=0$.
\end{proof}

\begin{remark}  \label{el_rem2}
In order to compare the boundary condition~\eqref{el10} with~\eqref{Wbc_0}
with $c_0=0$, for given $a_0$, $b_0 \ne 0$ with $a_0+b_0=1$, we just have to
make the choice of $\gb$ above as $\gb = a_0/(1-a_0)$.
\end{remark}

Next we compute the resolvent $R^e=(R^e_\gl,\,\gl>0)$ and the semigroup $U^e =
(U^e_t,\,t\ge 0)$ of the elastic Brownian motion. First we prepare with the
following lemma:

\begin{lemma}   \label{el_lem3}
For all $f\in C_0(\R_+)$, $\gl>0$, $x\in\R_+$, the following formula holds true
\begin{equation}    \label{el12}
    R^e_\gl f(x) = R^N_\gl f(x) - e_\gl(x)\, \hE_0\bigl(e^{-\gl\zeta_\gb}\bigr)\, R^N_\gl f(0),
\end{equation}
where
\begin{equation}    \label{el13}
    e_\gl(x) = E_x\bigl(e^{-\gl H_0}\bigr) = e^{-\sgl x},
\end{equation}
and $H_0$ is the hitting time of the origin by the Brownian motion $B$.
\end{lemma}

\begin{proof}
By construction of $B^e$,
\begin{align*}
    R^e_\gl f(x)
        &= \hE_x\Bigl(\int_0^{\zeta_\gb} e^{-\gl t} f(|B_t|)\,dt \Bigr)\\
        &= R^N_\gl f(x) - \hE_x\Bigl(\int_{\zeta_\gb}^\infty e^{-\gl t} f(|B_t|)\,dt \Bigr)\\
        &= R^N_\gl f(x) - \hE_x\Bigl(e^{-\gl\zeta_\gb} \int_0^\infty e^{-\gl t}
                                        f\bigl(|B_{t+\zeta_\gb}|\bigr)\,dt \Bigr).
\end{align*}
We compute the last expectation
\begin{align*}
    \hE_x\Bigl(e^{-\gl\zeta_\gb} &\int_0^\infty e^{-\gl t} f\bigl(|B_{t+\zeta_\gb}|\bigr)\,dt \Bigr)\\
        &= \gb \int_0^\infty e^{-\gb s} E_x\Bigl(e^{-\gl K_s} \int_0^\infty e^{-\gl t}
                f\bigl(|B_{t+K_s}|\bigr)\,dt\Bigr)\,ds\\
        &= \gb \int_0^\infty e^{-\gb s} E_x\Bigl(e^{-\gl K_s} \int_0^\infty e^{-\gl t}
                E_x\Bigl(f\bigl(|B_{t+K_s}|\bigr)\cond\cF_{K_s}\Bigr)\,dt\Bigr)\,ds\\
        &= \gb \int_0^\infty e^{-\gb s} E_x\bigl(e^{-\gl K_s}\bigr)
                E_0\Bigl(\int_0^\infty e^{-\gl t} f(|B_t|)\,dt\Bigr)\,ds\\
        &= \gb \int_0^\infty e^{-\gb s} E_x\bigl(e^{-\gl K_s}\bigr)\,ds\, R^N_\gl f(0),
\end{align*}
where we used relation~\eqref{el8}, the strong Markov property of $B$ with respect to the
$\cF$--stopping time $K_s$,
and $B_{K_s}=0$. Since $L$ only grows when $B$ is at the origin, we have for all $s\ge 0$,
$K_s\ge H_0$, whence $K_s = H_0 + K_s\comp \theta_{H_0}$. Therefore we get with the strong
Markov property of $B$ with respect to $H_0$
\begin{align*}
    E_x\bigl(e^{-\gl K_s}\bigr)
        &= E_x\bigl(e^{-\gl(H_0+K_s\comp\theta_{H_0})}\bigr)\\
        &= E_x\bigl(e^{-\gl H_0}\bigr)\,E_0\bigl(e^{-\gl K_s}\bigr)\\
        &= e^{-\sgl x}\,E_0\bigl(e^{-\gl K_s}\bigr),
\end{align*}
where we used the well-known Laplace transform of the density of $H_0$ under $P_x$,
e.g., \cite[p.~26]{ItMc74}, \cite[p.~96]{KaSh91} or \cite[p.~67]{ReYo91}. Inserting
the last expression above, we get formula~\eqref{el12}.
\end{proof}

\begin{remark}  \label{el_rem4}
We want to emphasize that the same calculation works for any Brownian motion on
$\R_+$ in the sense of definition~\ref{def_BM} with infinite lifetime having a local
time at the origin, i.e., a PCHAF (cf.\ appendix~\ref{app_killing}) whose
Lebesgue--Stieltjes measure is carried by the origin. In section~\ref{sect_gen} we
shall take advantage of this fact when with the corresponding analogue of
formula~\eqref{el12} we compute the resolvent of a Brownian motion with the boundary
condition~\eqref{Wbc_0} in its most general form, i.e., with all coefficients
non-vanishing.
\end{remark}

To get an explicit expression for the resolvent $R^e_\gl$, it remains to calculate
the expectation $\hE_0(\exp(-\gl\zeta_\gb))$. But this is easy with the help
formula~\eqref{a_LT3} in appendix~\ref{app_LT}:
\begin{align}
    \hE_0\bigl(e^{-\gl\zeta_\gb}\bigr)
        &= \gb\int_0^\infty e^{-\gb s}\,E_0\bigl(e^{-\gl K_s}\bigr)\,ds\nonumber\\
        &= \gb\int_0^\infty e^{-\gb s} e^{-\sgl s}\,ds\nonumber\\
        &= \frac{\gb}{\gb+\sgl}.    \label{el14}
\end{align}
Thus we have proved the following

\begin{corollary}   \label{el_cor5}
For all $f\in C_0(\R_+)$, $\gl>0$, $x\in\R_+$,
\begin{equation}    \label{el15}
    R^e_\gl f(x) = R^N_\gl f(x) - \frac{\gb}{\gb+\sgl}\, e_\gl(x)\, R^N_\gl f(0).
\end{equation}
\end{corollary}

With formula~\eqref{el15} it is now very easy to give another proof of theorem~\ref{el_thm1}:
Let $f\in C_0(\R_+)$, $\gl>0$. Then $R^N_\gl f$ satisfies the Neumann boundary condition at
$0$, and therefore we get with equation~\eqref{el13}
\begin{equation}    \label{el16}
    \bigl(R^e_\gl f\bigr)'(0+) = \frac{\gb\sgl}{\gb+\sgl}\,R^N_\gl f(0),
\end{equation}
while on the other hand
\begin{equation}    \label{el17}
    R^e_\gl f(0) = \frac{\sgl}{\gb+\sgl}\, R^N_\gl f(0),
\end{equation}
and therefore
\begin{equation}   \label{el18}
    \bigl(R^e_\gl f\bigr)'(0+) = \gb R^e_\gl f(0).
\end{equation}
Since for every $\gl>0$, $R^e_\gl$ maps $C_0(\R_+)$ onto the domain of the generator
of $B^e$, we have once again proved theorem~\ref{el_thm1}.

From equation~\eqref{el15} we can read off the resolvent kernel $r^e_\gl(x,y)$,
$x$, $y\in\R_+$, of the elastic Brownian motion
\begin{equation}    \label{el19}
    r^e_\gl(x,y)
        = r^N_\gl(x,y) - \frac{2\gb}{\gb +\sgl}\,\frac{1}{\sgl}\,e^{-\sgl(x+y)},
                \qquad x,\,y\in\R_+,
\end{equation}
where $r^N_\gl(x,y)$ is the Neumann kernel~\eqref{el5}. In order to compute the
transition kernel of the elastic Brownian motion, we want to find the inverse
Laplace transform of $r^e_\gl(x,y)$ as a function of $\gl>0$. It turns out that
it is more convenient to achieve this if we rewrite $r^e_\gl(x,y)$ in terms of the
Dirichlet kernel
\begin{equation}    \label{el20}
    r^D_\gl(x,y) = \frac{1}{\sgl}\,\bigl(e^{-\sgl |x-y|} - e^{-\sgl (x+y)}\bigr),
                        \qquad x,\,y\in\R_+.
\end{equation}
Then for $x$, $y\in\R_+$,
\begin{equation}    \label{el21}
    r^e_\gl(x,y)
        = r^D_\gl(x,y) + \frac{2}{\gb+\sgl}\,e^{-\sgl (x+y)}.
\end{equation}
Now the inverse Laplace transform of the second term on the right hand side is given
by formula~(5.6.12) in~\cite{ErMa54a}. So we find for the transition density
$p^e(t,x,y)$, $t>0$, $x$, $y\in\R_+$, of the elastic Brownian motion
\begin{equation}    \label{el22}
    p^e(t,x,y) = p^D(t,x,y) + 2 g_{\gb,0}(t,x+y),
\end{equation}
where $p^D(t,x,y)$ is the transition kernel of a Brownian motion on $\R_+$ which is
killed when reaching the origin, i.e., whose generator is one half times the
Laplacean on $\R_+$ with Dirichlet boundary conditions at the origin:
\begin{equation}    \label{el23}
    p^D(t,x,y) = p(t,x,y) - p(t,x,-y).
\end{equation}
Moreover, we have set
\begin{equation}    \label{el24}
    g_{\gb,0}(t,x)
        = g(t,x) - \frac{\gb}{2}\,\exp\Bigl(\gb x+ \frac{\gb^2 t}{2}\Bigr)
            \,\erfc\Bigl(\frac{x}{\sqrt{2t}}+\gb\sqrt{\frac{t}{2}}\,\Bigr),
\end{equation}
and $g(t,x)$ denotes the standard Gau{\ss}-kernel
\begin{equation}    \label{el25}
    g(t,x) = \frac{1}{\sqrt{2\pi t}}\,e^{-x^2/2t},\qquad t>0,\,x\in\R.
\end{equation}
Alternatively, we can write $p^e(t,x,y)$ in the form
\begin{equation}    \label{el26}
    p^e(t,x,y)
        = p^N(t,x,y) - \gb e^{\gb(x+y)+\gb^2 t/2}\,
                \erfc\Bigl(\frac{x+y}{\sqrt{2t}}+\gb\sqrt{\frac{t}{2}}\,\Bigr),
\end{equation}
where $p^N(t,x,y)$ is the Neumann heat kernel~\eqref{el3}.

\section{Brownian Motion With Sticky Boundary} \label{sect_st}
In this section we consider the case $a_0=0$, $b_0$, $c_0>0$  in the
boundary condition~\eqref{Wbc_0}. As mentioned in the introduction, this case is
handled by introducing a new time scale which slows down the reflecting Brownian
motion $|B|$ when it is at the origin. Again, the idea of how to implement
this with the help of L\'evy's local time goes back to the fundamental
paper~\cite{ItMc63} of It\^o and McKean. We continue to denote the local time of
the Brownian motion $B$ at the origin by $L=(L_t,\,t\ge 0)$. Let $\gg>0$, and set
\begin{equation}    \label{st1}
    \tau^{-1}(t) = t + \gg L_t,\qquad t\ge 0.
\end{equation}
Since $L$ is increasing, $\tau^{-1}$ is strictly increasing. Furthermore,
$\tau^{-1}(0)=0$ and $\tau^{-1}(+\infty) = +\infty$ are valid, which implies that
$\tau$ exists and is strictly increasing, too. In figure~\ref{fig4} the local time of
the path of
a reflected Brownian motion from above is depicted in red, the path of $\tau^{-1}$
(with $\gg=0.3$) in green, while the path of the new time scale $\tau$ is in
blue. (The grey diagonal line represents the ``deterministic time scale''.)
\begin{figure}[ht]
\begin{center}
    \includegraphics{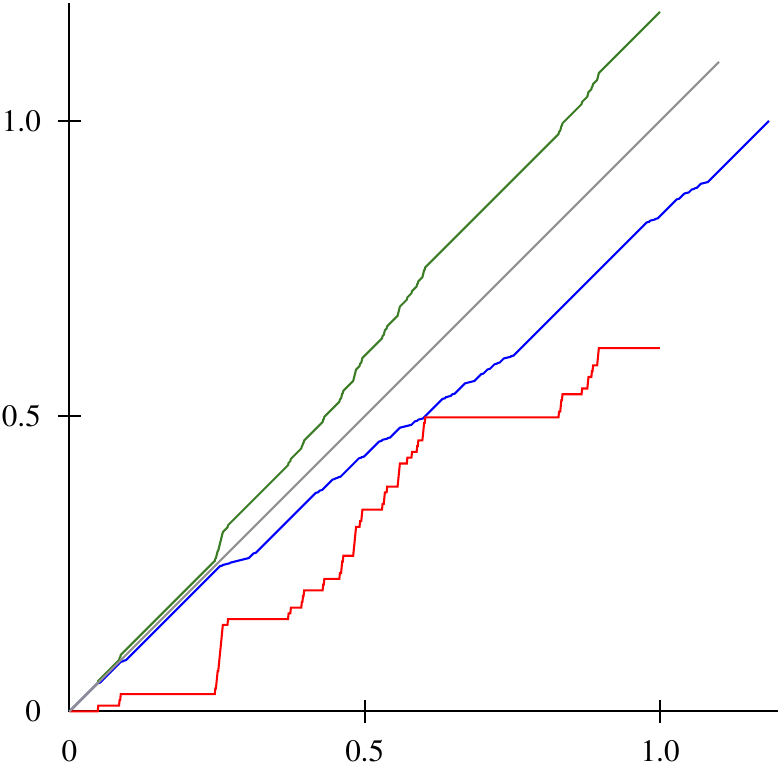}\\[0\baselineskip]
    \caption{Construction of the time scale $\tau$ (blue) from the local time
			$L$ (red) and $\tau^{-1}$ (grey)}  \label{fig4}
\end{center}
\end{figure}

Define a Brownian motion $B^s$ on $\R_+$ with \emph{sticky boundary} at the origin by
\begin{equation}    \label{st2}
    B_t^s = |B_{\tau(t)}|, \qquad t\ge 0.
\end{equation}
Figure~\ref{fig5} shows the path of a reflected Brownian motion from above in grey,
and the corresponding path of a Brownian motion with sticky boundary at $0$ in blue.
\begin{figure}[ht]
\begin{center}
    \includegraphics{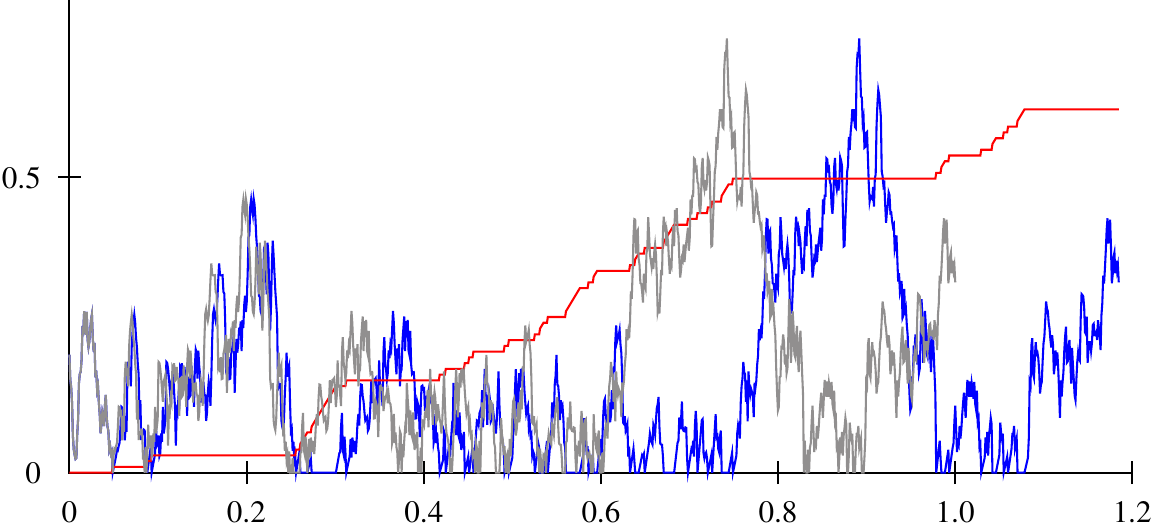}\\[0\baselineskip]
    \caption{A path of a ``sticky'' Brownian motion (blue) constructed from the
             path of a reflecting Brownian motion (grey). The local time at $0$
             of the ``sticky'' path is in red.}  \label{fig5}
\end{center}
\end{figure}
By formula~\eqref{st1} it is clear that $\tau^{-1}$ increases at a faster rate than
the ``deterministic time scale $t\mapsto t$'' whenever $B$ is at the origin, and
therefore $\tau$ increases slower than the deterministic time scale in these
instances. Hence --- at least in a heuristic sense --- $B^s$ experiences a slow down
when it is at the origin. Below this remark will be made precise.
With these considerations in mind, we shall also call $\gg$ the \emph{parameter of
stickiness} of $B^s$.

First we investigate the (strong) Markov property of $B^s$, and we fill in some
details of a proof given in~\cite[p.~160]{Kn81}. We recall the additivity property
\begin{equation} \label{addL}
L_{s+t} = L_s + L_t\comp\theta_s,\qquad s,t\in\R_+,
\end{equation}
of the local time $L$, which directly leads to the following formula, see~\cite[p.~160]{Kn81},
\begin{equation}    \label{st3}
    \tau(s+t) = \tau(s) + \tau(t)\comp\theta_{\tau(s)},\qquad s,\,t\ge 0,
\end{equation}
i.e., $\tau$ is additive on its own scale.

For $t\ge 0$ define $\cF^s_t=\cF_{\tau(t)}$. Then we have the following result

\begin{lemma}   \label{st_lem1}
For every $t\ge 0$, $\tau(t)$ is an $\cF$--stopping time. Furthermore,
$\cF^s=(\cF^s_t,\,t\ge 0)$ is a filtration of sub--$\gs$--algebras of
$\cF_{\infty}$, which is right continuous and complete relative to $P$, and such
that for every $t\ge 0$, $\cF^s_t\subset\cF_t$. Furthermore, $\cF^s_{\infty} =
\gs(\cF^s_t,\,t\ge 0)\subset\cF_\infty$.
\end{lemma}

\begin{proof}
First we show that for every $t\ge 0$, $\tau(t)$ is an $\cF$--stopping time: let
$s\ge 0$, then by~\eqref{st1}
\begin{equation*}
    \{\tau(t)\le s\} = \{t\le \tau^{-1}(s)\} = \{t \le s +\gg L_s\},
\end{equation*}
and the last set belongs to $\cF_s$. In particular, $\cF_{\tau(t)}$ is a
sub--$\gs$--algebra of $\cF_\infty$, and this entails the last statement of the
lemma. That $\cF^s$ is a filtration follows from the fact that $\tau$ is increasing,
and that for every $t\ge 0$, $\cF^s_t\subset\cF_t$ is due to $\tau(t)\le t$.
Moreover, $\tau(0)=0$ gives $\cF^s_0=\cF_0$, and therefore the completeness of
$\cF$ entails the completeness of $\cF^s$. It remains to show that $\cF^s$ is
right continuous. To this end let $t\ge 0$, and assume that $\gL\in\cF^s_{t+}$,
i.e., $\gL\in\cap_{n\in\N} \cF^s_{t+1/n}$.
We have to prove that $\gL\in\cF^s_t = \cF_{\tau(t)}$.
The fact that $\tau$ is (pathwise)
continuous implies that $\tau(t+1/n)$ converges from above to $\tau(t)$ as
$n$ tends to infinity. Hence for $u > 0$ we get
\begin{equation*}
    \{\tau(t)<u\} = \bigcup_{n\in\N} \{\tau(t+1/n) < u\}
        = \bigcup_{\gep>0,\,\gep\in\Q} \bigcup_{n\in\N} \{\tau(t+1/n) \le u-\gep\},
\end{equation*}
and
\begin{equation*}
    \gL\cap\{\tau(t)<u\}  = \bigcup_{\gep>0,\,\gep\in\Q} \bigcup_{n\in\N}
                                    \gL\cap\{\tau(t+1/n) \le u-\gep\}.
\end{equation*}
Because $\gL\in\cF_{\tau(t+1/n)}$ for all $n\in\N$, we find that for all $n\in\N$,
$\gep\in\Q$ with $\gep>0$,
\begin{equation*}
    \gL\cap\{\tau(t+1/n) \le u-\gep\} \in\cF_{u-\gep}\subset\cF_u.
\end{equation*}
Thus we have shown that for all $u>0$, $\gL\cap\{\tau(t)<u\}\in\cF_u$. But then
for all $u\ge 0$,  $\gL\cap\{\tau(t)\le u\}\in\cF_{u+}$. By assumption $\cF_{u+}=\cF_u$, and
this gives $\gL\in\cF_{\tau(t)}$.
\end{proof}

Since $B$ has continuous paths, $B$ is progressively measurable with respect to
$\cF$, and therefore (see, e.g., \cite[Proposition~I.4.9]{ReYo91}), for every $t\ge 0$,
$|B_{\tau(t)}|$ is measurable with respect to $\cF_{\tau(t)}$, i.e., $B^s$ is
adapted to $\cF^s$. Henceforth we shall always consider $B^s$ relative to the
filtration $\cF^s$, unless mentioned otherwise.

Next we define a family $\theta^s=(\theta^s_t,\,t\ge 0)$ of mappings from $\gO$
into itself by
\begin{equation}    \label{st4}
    \theta^s_t = \theta_{\tau(t)},\qquad t\ge 0.
\end{equation}
Then we obtain from equation~\eqref{st3} for $s$, $t\in\R_+$,
\begin{align}
    B^s_{s+t}
        &= \bigl|B_{\tau(s) + \tau(t)\comp \theta_{\tau(s)}}\bigr|\nonumber\\
        &= \bigl|B_{\tau(t)}\bigr|\comp \theta_{\tau(s)}    \label{st5}
\end{align}
and therefore
\begin{equation}    \label{st6}
    B^s_{s+t} = B^s_t\comp \theta^s_s,
\end{equation}
which shows that $\theta^s$ is a family of shift operators for $B^s$.

Now we claim that $B^s$ is a Markov process relative to $\cF^s$. First we remark
that it is an easy exercise to show that the reflecting Brownian motion $|B|$ is a
strong Markov process with respect to $\cF$. Let $x\in\R_+$, $s$, $t\ge 0$,
$C\in\cB(\R_+)$. Recall lemma~\ref{st_lem1} by which $\tau(t)$ is an $\cF$--stopping
time for every $t\ge 0$. Using equation~\eqref{st5} we can calculate as follows:
\begin{align*}
    P_x\bigl(B^s_{s+t}\in C \cond \cF^s_t\bigr)
        &= P_x\bigl(|B_{\tau(s)}|\comp\theta_{\tau(t)}\in C \cond \cF_{\tau(t)}\bigr)\\
        &= P_{|B_{\tau(t)}|}\bigl(|B_{\tau(s)}|\in C\bigr)\\
        &= P_{B^s_t}\bigl(B^s_s\in C\bigr),
\end{align*}
and our claim is proved.

Recall that $H_0$ is the hitting time of the origin by the standard Brownian
motion $B$. As a next step we want to prove that $B^s$ has a strong Markov property
with respect to the hitting time $H_0$ of the origin and the filtration $\cF$, i.e.,
we want to show that for all $x\in\R_+$, $t\ge 0$, $C\in\cB(\R_+)$,
\begin{equation}    \label{st7}
    P_x\bigl(B^s_{t+H_0}\in C\cond \cF_{H_0}\bigr)
        = P_0\bigl(B^s_t\in C\bigr),\qquad \text{$P_x$--a.s.\ on $\{H_0<+\infty\}$}.
\end{equation}
All considerations below are done on the set $\{H_0<+\infty\}$ which has
$P_x$--measure one for all $x\in\R_+$. Observe that $\tau(H_0)=\tau^{-1}(H_0) = H_0$
which follows from
\begin{equation*}
    \tau^{-1}(H_0) = H_0 + \gg L_{H_0} = H_0,
\end{equation*}
because $L_{H_0}=0$. In particular, it follows that $H_0$ is the hitting time
of the origin for both $B$ and $B^s$. Thus by~\eqref{st3}
\begin{align}
    B^s_{t+H_0}
        &= |B_{\tau(t+H_0)}| \nonumber\\
        &= |B_{\tau(H_0)+\tau(t)\comp \theta_{\tau(H_0)}}|  \nonumber\\
        &= |B_{H_0+\tau(t)\comp \theta_{H_0}}|  \nonumber\\
        &= |B_{\tau(t)}|\comp \theta_{H_0}  \nonumber\\
        &= B^s_t\comp \theta_{H_0}. \label{st8}
\end{align}
Using the fact shown above that for $t\ge 0$, $B^s_t$ is $\cF_\infty$--measurable,
and the strong Markov property of the reflecting Brownian motion $|B|$, we now
compute for $x\in\R_+$, $C\in\cB(\R_+)$, $t\ge 0$, as follows:
\begin{align*}
    P_x\bigl(B^s_{t+H_0}\in C\cond \cF_{H_0}\bigr)
        &= P_x\bigl(B^s_t\comp \theta_{H_0}\in C\cond \cF_{H_0}\bigr)\\
        &= P_x\bigl(|B_{\tau(t)}|\comp\theta_{H_0}\cond \cF_{H_0}\bigr)\\
        &= P_0\bigl(|B_{\tau(t)}|\in C\bigr)\\
        &= P_0\bigl(B^s_t\in C\bigr).
\end{align*}

Since  $H_0$ is also the hitting time of $0$ by $B^s$, it is also a stopping time
for $\cF^s$. Thus $\cF^s_{H_0}$ is a well-defined $\gs$--algebra, and
$\gL\in\cF^s_{H_0}$ if and only if $\gL\in\cF^s_\infty$, and for all $t\ge 0$,
$\gL\cap\{H_0\le t\}\in\cF^s_t$. Since $\cF^s_\infty\subset\cF_\infty$ (s.a.), we
immediately get $\gL\in\cF_{H_0}$, i.e., $\cF^s_{H_0}\subset\cF_{H_0}$. Consequently
$B^s$ is a strong Markov process relative to $H_0$ and the filtration $\cF^s$, too.

Using the facts established above that $(B^s,\cF^s)$ is Markovian, and strongly
Markovian for the hitting time of the origin, we are now in position to apply the
same arguments as in the proof of Theorem~6.1 in~\cite{Kn81} (cf.\ also the proof of
theorem~4.3 in~\cite{BMMG1} with a slightly different argument for the case of metric
graphs and somewhat more details than those given in~\cite{Kn81}) to conclude that $B^s$
is a Feller process. But this proves

\begin{theorem} \label{thm}
$B^s$ is a normal strong Markov process with respect to $\cF^s$.
\end{theorem}

The necessary path properties being obvious, we have also proved

\begin{corollary}   \label{cor}
$B^s$ is a Brownian motion on $\R_+$ in the sense of Definition~\ref{def_BM}.
\end{corollary}

Next we turn to calculate the boundary conditions of the generator $A^s$ of $B^s$.
Again we will use Dynkin's formula. We first argue that the
origin is not a trap for $B^s$: Let $Z\subset \R_+$ denote the set of zeroes of the
underlying Brownian motion $B$. Given $s\ge 0$, choose $t_0\ge s$ in the complement
$Z^c$ of $Z$. Consider $t = \tau^{-1}(t_0)$, i.e., $t = t_0 + \gg L_{t_0}$.
Obviously $t\ge s$, and $\tau(t)\in Z^c$. Therefore $B^s_t = |B_{\tau(t)}| \ne 0$.
Hence $0$ is indeed not a trap for $B^s$. For $\gep>0$ let $H_{\gg,\gep}$ denote the hitting
time of the complement of $[0,\gep)$ by $B^s$, that is, since by construction
$B^s$ has infinite lifetime, $H_{\gg,\gep}$ is the hitting time of the point
$\gep\in\R_+$ by $B^s$. Then for $f\in C_0(\R_+)$ we can calculate as follows:
\begin{align*}
    A^s f(0)
        &= \lim_{\gep\downarrow 0} \frac{E_0\Bigl(f\bigl(B^s_{H_{\gg,\gep}}\bigr)\Bigr)
                    - f(0)}{E_0(H_{\gg,\gep})}\\
        &= \lim_{\gep\downarrow 0} \frac{f(\gep) - f(0)}{E_0(H_{\gg,\gep})}.
\end{align*}
Now we use the following

\begin{lemma}   \label{st_lem4}
Let $H_\gep$ denote the hitting time of $\gep>0$ by the reflecting Brownian motion
$|B|$. Then under $P_0$,
\begin{equation}    \label{st9}
    H_{\gg,\gep} = H_\gep + \gg L_{H_\gep}.
\end{equation}
\end{lemma}

\begin{proof}
By definition
\begin{align*}
    H_{\gg,\gep}
        &= \inf\{s\ge 0,\, |B_{\tau(s)}| = \gep\},\\
    H_\gep
        &= \inf\{t\ge 0,\, |B_t|=\gep\}.
\end{align*}
Set $s = H_\gep + \gg L_{H_\gep}$, then $\tau(s) = H_\gep$, i.e.,
$|B_{\tau(s)}|=\gep$, and therefore we get the inequality
\begin{equation*}
    H_{\gg,\gep} \le H_\gep + \gg L_{H_\gep}.
\end{equation*}
Now assume that the last inequality is strict. This would entail
\begin{equation*}
    \tau(H_{\gg,\gep}) < \tau\bigl(H_\gep + \gg L_{H_\gep}\bigr) = H_\gep
\end{equation*}
because $\tau$ is strictly increasing. But $|B_{\tau(H_{\gg,\gep})}|=\gep$,
in contradiction to the definition of $H_\gep$ as the hitting time of $\gep$
by $|B|$. Thus formula~\eqref{st9} is proved.
\end{proof}

We already know (cf.\ the proof of theorem~\ref{el_thm1}) that $E_0(H_\gep)=\gep^2$.
On the other hand, lemma~\ref{a_LT_lem2} states that under $P_0$,  $L_{H_\gep}$
is exponentially distributed with mean $\gep$. Therefore we find by lemma~\ref{st_lem4},
\begin{equation}    \label{st10}
    E_0\bigl(H_{\gg,\gep}\bigr) = \gep^2 + \gg \gep.
\end{equation}
Thus
\begin{equation}    \label{st11}
    \frac{\gg}{2}\,f''(0+) = \gg A^s f(0) = f'(0+),
\end{equation}
and we have proved

\begin{theorem} \label{st_thm5}
The domain $\cD(A^s)$ of the generator $A^s$ of the Brownian motion $B^s$ with sticky
origin and
stickiness parameter $\gg>0$ is equal to the space of functions $f\in C_0^2(\R_+)$ so that
the boundary condition~\eqref{st11} holds true.
\end{theorem}

\begin{remark}  \label{st_rem6}
Assume that in~\eqref{Wbc_0} we are given $a_0=0$, $b_0$, $c_0\ne 0$. Then  we choose
$\gg = c_0/(1-c_0)$, and the so constructed process $B^s$ has a generator with
a boundary condition given by~\eqref{Wbc_0}.
\end{remark}

\begin{remark}  \label{st_rem7}
Equation~\eqref{st10} shows that when starting at the origin the Brownian motion
$B^s$ in the average needs the time $\gep^2+\gg\gep$ to leave the interval $[0,\gep)$,
while the reflecting Brownian motion needs the time $\gep^2$. On the other hand,
consider the case that $B^s$ starts at $x>0$, and let $\gep>0$ be small enough so that
$\gep<x$. Then on the average $B^s$ takes the time $\gep^2$ to leave the interval
$(x-\gep,x+\gep)$, because there it is equivalent to $|B|$. This shows very clearly the effect
of ``stickiness'' of the origin for $B^s$: Not only takes $B^s$ longer than $|B|$ to
leave a neighborhood of the origin, even it does this on a time scale different from
that of $|B|$.
\end{remark}

Finally we compute the resolvent $R^s=(R^s_\gl,\,\gl>0)$ and the transition kernel $p^s(t,x,y)$,
$t>0$, $x$, $y\in\R_+$, of $B^s$. It will be convenient to introduce
for measurable functions $f$, $g$ on $\R_+$ the inner product
\begin{equation*}
    (g,f) = \int_0^\infty g(x) f(x)\,dx,
\end{equation*}
whenever the integral on the right hand side is defined.

\begin{lemma}   \label{st_lem8}
For all $f\in C_0(\R_+)$, the relation
\begin{equation}    \label{st12}
    \frac{1}{2}\,(R^s_\gl f)''(0+)
        = \frac{1}{\sgl + \gg\gl}\,\bigl(2\gl\,(e_\gl,f)-\sgl\,f(0)\bigr),\qquad \gl>0,
\end{equation}
holds true, where $e_\gl$ is given by equation~\eqref{el13}.
\end{lemma}

\begin{proof}
The proof given here is basically the proof in~\cite{Kn81}, though somewhat
reorganized and simplified. Let $A^s$ denote the generator
of $B^s$. Since for $\gl>0$, $R^s_\gl f \in \cD(A^s)$, $A^s R^s_\gl f$ belongs to
$C_0(\R_+)$. Also for $x>0$, $A^s R^s_\gl f(x) = 1/2\,(R^s_\gl f)''(x)$. Thus
$1/2\,(R^s_\gl f)''(0+) = A^s R^s_\gl f(0)$, and with
$A^s R^s_\gl = \gl R^s_\gl -\text{id}$ we can compute as follows
\begin{align*}
    \frac{1}{2}\,(R^s_\gl f)''(0+)
        &= \gl R^s_\gl f(0) - f(0)\\
        &= \gl E_0\Bigl(\int_0^\infty e^{-\gl t} \bigl(f(B^s_t)-f(0)\bigr)\,dt\Bigr)\\
        &= \gl E_0\Bigl(\int_0^\infty e^{-\gl (s+\gg L_s)}
                \bigl(f(|B_s|)-f(0)\bigr)\,(ds+\gg dL_s)\Bigr)\\
        &= \gl E_0\Bigl(\int_0^\infty e^{-\gl (s+\gg L_s)}
                \bigl(f(|B_s|)-f(0)\bigr)\,ds\Bigr).
\end{align*}
In the last step we made use of the fact that $L$ only grows when $|B|$ is at
the origin. Therefore the terms with $f(|B_s|)$ and $f(0)$ cancel each other in the
integral with respect to $dL_s$. Now we use the well-known joint law of  $(|B_s|,
L_s)$ under $P_0$
\begin{equation}    \label{st13}
    P_0(|B_s|\in dx, L_s\in dy) = 2\,\frac{x+y}{\sqrt{2\pi s^3}}\,e^{-(x+y)^2/2s}\,dx\,dy,
                                    \qquad x,\,y\ge 0,
\end{equation}
which follows directly from L\'evy's theorem (see, e.g., \cite[Theorem~3.6.17]{KaSh91}),
and the joint law of a standard Brownian motion and its running maximum under $P_0$
(e.g., \cite[Equation~2.8.2]{KaSh91}). Thus we get
\begin{equation*}
\begin{split}
    \frac{1}{2}\,(R^s_\gl f)''(0+)
        = 2\gl \int_0^\infty \int_0^\infty \int_0^\infty
                &e^{-\gl(s+\gg y)}\,\bigl(f(x)-f(0)\bigr)\\
                &\hspace{2em}\times \frac{x+y}{\sqrt{2\pi s^3}}\,
                     e^{-(x+y)^2/2s}\,ds\,dx\,dy.
\end{split}
\end{equation*}
We make use of the well-known Laplace transform (e.g., \cite[1.5.30]{ObBa73})
\begin{equation}    \label{st14}
    \int_0^\infty e^{-\gl s}\,\frac{a}{\sqrt{2\pi s^3}}\,e^{-a^2/2s}\,ds
        = e^{-\sgl a},\qquad a>0,
\end{equation}
and get
\begin{align*}
    \frac{1}{2}\,(R^s_\gl f)''(0+)
        &= 2\gl\int_0^\infty e^{-\gl \gg y}\,e^{-\sgl (x+y)}\bigl(f(x)-f(0)\bigr)\,dx\,dy\\
        &= \frac{1}{\sgl +\gg\gl}\,\bigl(2\gl\,(e_\gl,f) - \sgl f(0)\bigr),
\end{align*}
which proves the lemma.
\end{proof}

\begin{corollary}   \label{st_cor9}
For all $f\in C_0(\R_+)$, $\gl>0$,
\begin{equation}    \label{st15}
    R^s_\gl f(0) = \frac{1}{\sgl + \gg\gl}\,\bigl(2(e_\gl,f) + \gg f(0)\bigr).
\end{equation}
\end{corollary}

\begin{proof}
Formula~\eqref{st15} follows directly from equation~\eqref{st12}, the identity
$R^s_\gl f = \gl^{-1}(f+A^s R^s_\gl f)$, and the
fact that $A^s$ acts as one half the second derivative on its domain.
\end{proof}

Now we apply the well-known first passage time formula (e.g., \cite[p.~94]{ItMc74})
in the following form (for the convenience of the reader we include a detailed proof
of this formula within a general context in appendix~\ref{app_FPTF})
\begin{equation}    \label{st16}
    R^s_\gl f(x) = R^D_\gl f(x) + E_x\bigl(e^{-\gl H_0}\bigr)\,R^s_\gl f(0),
\end{equation}
which combined with~\eqref{el13} simplifies to
\begin{equation}    \label{st17}
    R^s_\gl f(x) = R^D_\gl f(x) + e_\gl(x)\,R^s_\gl f(0).
\end{equation}
From this formula we can read off the resolvent kernel $r^s_\gl(x,dy)$ of the
Brownian motion with sticky origin:

\begin{corollary}   \label{st_cor10}
For all $\gl>0$, $x$, $y\in\R_+$,
\begin{equation}    \label{st18}
    r^s_\gl(x,dy)
        = r^D_\gl(x,y)\,dy + \frac{1}{\sgl + \gg \gl}\,\bigl(2e^{-\sgl(x+y)}
                    + \gg e^{-\sgl x}\,\gep_0(dy)\bigr),
\end{equation}
where $r^D_\gl(x,y)$ is defined in~\eqref{el20}, and $\gep_0$ is the Dirac measure in $0$.
\end{corollary}

From the first passage time formula~\eqref{st17}, we immediately get for all $f\in
C_0(\R_+)$
\begin{equation}    \label{st19}
    \bigl(R^s_\gl f\bigr)'(0+) = 2 (e_\gl, f) -\sgl\,R^s_\gl f(0),\qquad \gl>0.
\end{equation}
Inserting formula~\eqref{st15} into the right hand side and comparing the result
with~\eqref{st12}, we find
\begin{equation}    \label{st20}
    \bigl(R^s_\gl f\bigr)'(0+) = \frac{\gg}{2} \bigl(R^s_\gl f\bigr)''(0+),
\end{equation}
which is another proof of theorem~\ref{st_thm5}.

The inverse Laplace transform of the second and the third term on the right hand side
of equation~\eqref{st18} can be done with the help of formula~(5.6.16) in~\cite{ErMa54a}.
The result is the following

\begin{corollary}   \label{st_cor11}
For $t>0$, $x$, $y\in\R_+$, the transition kernel $p^s(t,x,dy)$ of the Brownian
motion with sticky origin is given by
\begin{equation}    \label{st21}
    p^s(t,x,dy) = p^D(t,x,y)\,dy + 2 g_{0,\gg}(t,x+y)\,dy + \gg g_{0,\gg}(t,x)\,\gep_0(dy),
\end{equation}
where $p^D(t,x,y)$ is given in~\eqref{el23}, and
\begin{equation}    \label{st22}
    g_{0,\gg}(t,x) = \frac{1}{\gg}\,\exp\Bigl(\frac{2x}{\gg} + \frac{2t}{\gg^2}\Bigr)\,
                    \erfc\Bigl(\frac{x}{\sqrt{2t}}+\frac{\sqrt{2t}}{\gg}\,\Bigr),
                        \qquad t>0,\,x\ge 0.
\end{equation}
\end{corollary}

\begin{remark}  \label{st_rem12}
$g_{0,\gg}$ is a special case of a more general heat kernel which will play an
essential role in the next section, and which is investigated in detail in
appendix~\ref{app_LTHK}. In particular, it is shown there, that as $\gg\downarrow 0$,
$g_{0,\gg}(t,x)$ converges pointwise to the usual Gau{\ss} kernel
$g(t,x)$~\eqref{el25}. Therefore, as $\gg\downarrow 0$, the transition kernel of the
Brownian motion with sticky origin converges pointwise to the Neumann transition
kernel, i.e., to the transition kernel of the reflecting Brownian motion, as it
should.
\end{remark}

We conclude this section with a discussion of the local time of $(B^s,\cF^s)$ at the
origin, and this will also provide results needed for the construction of the general
Brownian motion on $\R_+$ in the next section. Define $L^s = (L^s_t,\,t\ge 0)$ with
\begin{equation}    \label{st23}
    L^s_t = L_{\tau(t)}.
\end{equation}
Observe that $L^s$ has continuous non-decreasing paths and that $L^s_0=0$, because
$\tau(0)=0$ and $L_0=0$. Moreover, $L^s$ is $\cF^s$--adapted since for every $t\ge
0$, $\tau(t)$ is an $\cF$--stopping time, so that the progressive measurability of
$L$ relative to $\cF$ entails that $L_{\tau(t)}$ is measurable with respect to
$\cF_{\tau(t)}=\cF^s_t$. We claim that $L^s$ is additive with respect to the shift
operators $\theta^s$. Namely, $L^s_{s+t} = L^s_s + L^s_t\comp\theta^s_s$ holds for all
$s$, $t\in\R_+$, as a consequence of the relations~\eqref{addL}, \eqref{st3}, and the
definition $\theta^s_s=\theta_{\tau(s)}$. Indeed, by evaluating at $\go\in\gO$
we obtain
\begin{align*}
    L^s_{s+t}(\go)
        &= L_{\tau(s+t)(\go)}(\go)\\
        &= L_{\tau(s)(\go) + \tau(t)(\theta_{\tau(s)(\go)}(\go))}(\go)\\
	   &= L_{\tau(s)(\go)}(\go) + L_{\tau(t)(\theta_{\tau(s)(\go)}(\go))}
					\bigl(\theta_{\tau(s)(\go)}(\go)\bigr)\\	
        &= L^s_s(\go) + L^s_t\bigl(\theta^s_s(\go)\bigr).
\end{align*}
Altogether we have proved that $L^s$ is a PCHAF for
$(B^s,\cF^s)$. Finally we show that the Lebesgue--Stieltjes measures of its paths are
carried by the origin, that is, we have to prove that $L^s$ only increases when $B^s$
is at the origin. But $L^s$ increases at time $t\ge 0$ if and only if $L$ increases
at time $\tau(t)$ which is true only if $|B_{\tau(t)}|=B^s_t$ is at the origin.
Hence $L^s$ is a local time at $0$ for $B^s$. We determine its normalization by
computing its $\ga$--potential (see e.g., \cite{BlGe68}): Let $\ga>0$, then
\begin{align*}
    E_0\Bigl(\int_0^\infty e^{-\ga t}\,dL^s_t\Bigr)
        &= E_0\Bigl(\int_0^\infty e^{-\ga t}\,dL_{\tau(t)}\Bigr)\\
        &= E_0\Bigl(\int_0^\infty e^{-\ga t -\ga\gg L_t}\,dL_t\Bigr)\\
        &= -\frac{1}{\ga\gg}E_0\Bigl(\int_0^\infty e^{-\ga t}\,de^{-\ga\gg L_t}\Bigr)\\
        &= \frac{1}{\ga\gg} - \frac{1}{\gg}E_0\Bigl(\int_0^\infty e^{-\ga t - \ga\gg L_t}\,dt\Bigr).
\end{align*}
The well-known law of $L_t$ under $P_0$ is given by
\begin{equation}    \label{st24}
    P_0(L_t\in dy) = \frac{2}{\sqrt{2\pi t}}\,e^{-y^2/2t}\,dy,\qquad y\ge 0,
\end{equation}
as, for example, we can easily deduce from the joint law~\eqref{st13}. Thus
we can compute the above expectation
\begin{align*}
    E_0\Bigl(\int_0^\infty e^{-\ga t - \ga\gg L_t}\,dt\Bigr)
        &= 2 \int_0^\infty e^{-\ga\gg y}\,\Bigl(\int_0^\infty e^{-\ga t}\,g(t,y)\,dt \Bigr)\,dy\\
        &= \frac{2}{\sqrt{2\ga}\,(\sqrt{2\ga} + \ga\gg)},
\end{align*}
where $g(t,y)$ is the Gau{\ss} kernel~\eqref{el25}. Hence
\begin{equation*}
    E_0\Bigl(\int_0^\infty e^{-\ga t}\,dL^s_t\Bigr)
        = \frac{1}{\sqrt{2\ga}+\ga\gg}, \qquad x\in\R_+,\,\ga>0.
\end{equation*}
Now consider $x\ge 0$. Recall that above we had argued that on $[0,H_0]$ the paths
of $B$ and $B^s$ coincide, so that $H_0$ is also the hitting time of $B^s$,
and $L^s$ is identically zero on $[0,H_0]$, too. Hence the strong Markov
property of $B^s$ gives
\begin{align*}
    E_x\Bigl(\int_0^\infty e^{-\ga t}\,dL^s_t\Bigr)
        &= E_x\Bigl(e^{-\ga H_0}\int_0^\infty e^{-\ga t}\,dL^s_{t+H_0}\Bigr)\\
        &= E_x\Bigl(e^{-\ga H_0}E_x\Bigl(\int_0^\infty e^{-\ga t}\,dL^s_{t+H_0}\,\Big|
                    \,\cF^s_{H_0}\Bigr)\Bigr)\\
        &= E_x\bigl(e^{-\ga H_0}\bigr)\,E_0\Bigl(\int_0^\infty e^{-\ga t}\,dL^s_t\Bigr).
\end{align*}
Thus we finally obtain
\begin{equation}    \label{st25}
    E_x\Bigl(\int_0^\infty e^{-\ga t}\,dL^s_t\Bigr)
        = \frac{e^{-\sqrt{2\ga}x}}{\sqrt{2\ga}+\ga\gg}, \qquad x\in\R_+,\,\ga>0.
\end{equation}

\section{The General Brownian Motion on $\R_+$} \label{sect_gen}
In this section we construct and investigate the most general Brownian motion on
$\R_+$ in the sense that all three coefficients in the boundary
condition~\eqref{Wbc_0} are non-zero. That is, in contradistinction to our
discussion in the previous section now the term $f(0)$ appears, and we have already
seen in section~\ref{sect_el} that such a term appears when killing the process
at the origin on the scale of the local time at $0$. We proceed as follows.

Fix $\gg>0$ and consider the Brownian motion $B^s$ defined in section~\ref{sect_st}.
Fix $\gb>0$, and consider the product spaces $\bigl(\hgO,\hcA, (\hP_x,\,x\in\R_+)\bigr)$,
and the random variable $S$ as in section~\ref{sect_el}. As there, extend the
random variables $B^s_t$, $L^s_t$, $t\ge 0$, to $\hgO$ in the trivial way. Set
\begin{equation}    \label{gen1}
    \zeta_{\gb,\gg} = \inf\,\{t\ge 0,\,L^s_t>S\},
\end{equation}
and
\begin{equation}    \label{gen2}
    B^g_t = \begin{cases}
                B^s_t, &  \text{if $t<\zeta_{\gb,\gg}$,}\\
                \gD,   &  \text{if $t\ge \zeta_{\gb,\gg}$.}
            \end{cases}
\end{equation}
Figure~\ref{fig6} shows this construction on the basis of the path of a ``sticky''
Brownian motion depicted in figure~\ref{fig5} (which in turn was constructed from
the path of a reflecting Brownian motion in figure~\ref{fig2}). In figure~\ref{fig6}
the local time $L^s$ at $0$ of the motion $B^s$ is in red, its path
in blue and grey. When $L^s$ increases above level $S$ (given in the figure by
$S=0.58$) at time $\zeta_{\gb,\gg}=1.06$, the process is killed, i.e., we obtain the
blue piece of path, and after time $\zeta_{\gb,\gg}=1.06$ the path is constant and
equal to $\gD$ .
\begin{figure}[ht]
\begin{center}
    \includegraphics{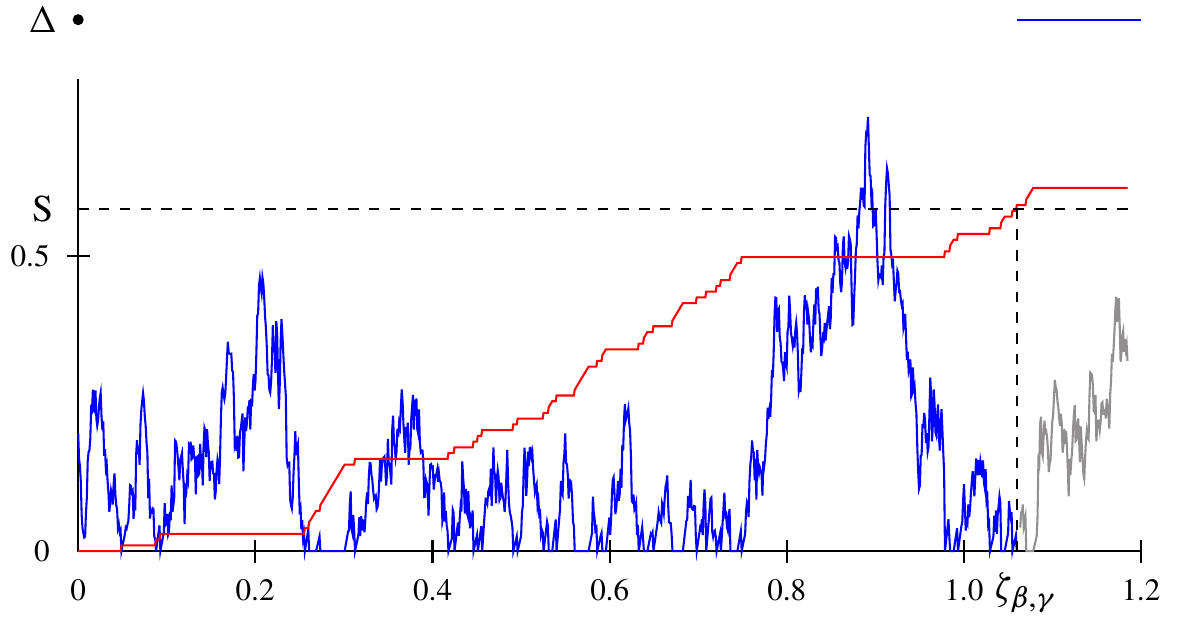}\\
    \caption{A path of the general Brownian motion on $\R_+$ (blue) constructed
             from the path of a sticky Brownian motion (grey) and its local time (red)}
             \label{fig6}
\end{center}
\end{figure}

Again we can apply the results in appendix~\ref{app_killing}. In particular, we can
equip $(\hgO,\hcA)$ with a right continuous, complete filtration $\hcF^g$ which is
constructed there, so that $B^g$ is adapted and strongly Markovian relative to
$\hcF^g$. Moreover, $B^g$ has obviously the right path properties, and it is equivalent
to a standard Brownian motion up to its hitting time of the origin. Thus $B^g$
is a Brownian motion on $\R_+$ in the sense of definition~\ref{def_BM}.

Next we compute the boundary conditions of the generator $A^g$ of $B^g$.
In this case it seems to be difficult to do this
via Dynkin's formula. Therefore this time we only give an argument based on the
computation on the resolvent $R^g$ of $B^g$. However, instead of using the first
passage time formula as is done in~\cite{ItMc63, ItMc74, Kn81}, here we use the
analogue of formula~\eqref{el12}, which in the case at hand reads
\begin{equation}    \label{gen3}
    R^g_\gl f(x) = R^s_\gl f(x) - e_\gl(x)\,
                        \hE_0\bigl(e^{-\gl\zeta_{\gb,\gg}}\bigr)\, R^s_\gl f(0),
\end{equation}
where $R^s$ is the resolvent of $B^s$. As mentioned in remark~\ref{el_rem4}, the
proof is the same as for the elastic case discussed in section~\ref{sect_el}. Thus we
only need to compute the expectation occuring in equation~\eqref{gen3}. To this end we
first prove the following result.

\begin{lemma}   \label{gen_lem1}
For all $\gb$, $\gg> 0$,
\begin{equation} \label{gen4}
    \zeta_{\gb,\gg} = \zeta_\gb + \gg S
\end{equation}
holds true.
\end{lemma}

\begin{proof}
Define the subsets
\begin{align*}
    J_\gb(S)        &= \{t\ge 0,\,L_t>S\}\\
    J_{\gb,\gg}(S)  &= \{t\ge 0,\,L^s_t>S\}
\end{align*}
of $\R_+$. Since $L^s_t=L_{\tau(t)}$ and $\tau$ is a bijection from $\R_+$ onto
itself (cf.\ section~\ref{sect_st}), we get $J_\gb(S)=\emptyset$ if and only if
$J_{\gb,\gg}(S)=\emptyset$, and in this case relation~\eqref{gen4} is trivially valid
due to the convention $\inf\emptyset = +\infty$.
Now assume that one of the sets is non-empty, and therefore both are non-empty. Then we
have $t\in J_{\gb,\gg}(S)$ if and only if $\tau(t)\in J_\gb(S)$, and therefore
$\tau$ is also a bijection from $J_{\gb,\gg}(S)$ onto $J_\gb(S)$. Because $\tau$ is
strictly increasing (cf.\ section~\ref{sect_st}) it follows that
\begin{equation*}
    \tau\bigl(\inf J_{\gb,\gg}(S)\bigr) = \inf J_\gb(S),
\end{equation*}
or in other words
\begin{equation*}
    \zeta_{\gb,\gg} = \tau^{-1}(\zeta_\gb) = \zeta_\gb + \gg L_{\zeta_\gb}.
\end{equation*}
The continuity of $L$ implies $L_{\zeta_\gb}=S$, and the proof is concluded.
\end{proof}

\begin{corollary}   \label{gen_cor2}
For all $\gl >0$,
\begin{equation}    \label{gen5}
    \hE_0\bigl(e^{-\gl \zeta_{\gb,\gg}}\bigr) = \gb \rho(\gl),
\end{equation}
holds true, where $\rho$ is given by
\begin{equation}    \label{gen6}
    \rho(\gl) = (\gb + \sgl +\gg\gl)^{-1},\qquad \gl>0.
\end{equation}
\end{corollary}

\begin{proof}
By formula~\eqref{el8} and lemma~\ref{gen_lem1},
\begin{align*}
    \hE_0\bigl(e^{-\gl \zeta_{\gb,\gg}}\bigr)
        &= \hE_0\bigl(e^{-\gl K_S - \gl\gg S}\bigr)\\
        &= \gb\int_0^\infty e^{-(\gb + \gg\gl)s}\,E_0\bigl(e^{-\gl K_s}\bigr)\,ds\\
        &= \gb\int_0^\infty e^{-(\gb+\sgl+\gg\gl)s}\,ds\\
        &= \frac{\gb}{\gb + \sgl + \gg\gl},
\end{align*}
where again we used equation~\eqref{a_LT3}.
\end{proof}

Equation~\eqref{gen3} immediately yields for $f\in C_0(\R_+)$, $\gl>0$, $x\in\R_+$,
\begin{equation}    \label{gen7}
    R^g_\gl f(x) = R^s_\gl f(x) - \gb \rho(\gl) e_\gl(x) R^s_\gl f(0),
\end{equation}
and in particular for $x=0$
\begin{equation}    \label{gen8}
    R^g_\gl f(0) = (\sgl+\gg\gl)\rho(\gl) R^s_\gl f(0).
\end{equation}
Differentiation of the right hand side of equation~\eqref{gen7} at $0$ from the
right gives
\begin{equation}    \label{gen9}
    \bigl(R^g_\gl f\bigr)'(0+) = \bigl(R^s_\gl f\bigr)'(0+) + \gb\sgl\rho(\gl) R^s_\gl f(0).
\end{equation}
In order to compute the second derivative of $R^g_\gl f$ at $0$ from the right,
we make use of the fact that $R^s_\gl f$ satisfies the boundary condition~\eqref{st11}:
\begin{align}
   \frac{\gg}{2}\, \bigl(R^g_\gl &f\bigr)''(0+)\nonumber\\
        &= \bigl(R^s_\gl f\bigr)'(0+) - \gg\gl\gb\rho(\gl) R^s_\gl f(0)\nonumber\\
        &= \bigl(R^s_\gl f\bigr)'(0+) + \gb\sgl\rho(\gl) R^s_\gl f(0)
                - \gb(\sgl+\gg\gl)\rho(\gl) R^s_\gl f(0)   \label{gen10}.
\end{align}
Upon comparison of equation~\eqref{gen10} with~\eqref{gen8} and~\eqref{gen9}, we find
the following result.

\begin{theorem} \label{gen_thm3}
The domain $\cD(A^g)$ of the generator $A^g$ of the Brownian motion $B^g$ with parameters
$\gb$, $\gg>0$ is equal to the space of functions $f\in C_0^2(\R_+)$ so that
\begin{equation}    \label{gen11}
    \gb f(0) - f'(0+) + \frac{\gg}{2}\,f''(0+) = 0.
\end{equation}
holds true.
\end{theorem}

\begin{remark}  \label{gen_rem4}
In order to compare the boundary condition~\eqref{gen11} with~\eqref{Wbc_0}, for
given non-zero $a_0$, $b_0$, $c_0$, as in theorem~\ref{Feller-bc}, we only need to
choose $\gb = a_0/b_0$, and $\gg=c_0/b_0$, to get the boundary condition in the
form~\eqref{Wbc_0}.
\end{remark}

As another byproduct of equation~\eqref{gen7} we get in combination with formulae~\eqref{st15},
\eqref{st17} the following result.

\begin{corollary}   \label{gen_cor5}
For all $\gl>0$, $f\in C_0(\R_+)$,
\begin{equation}    \label{gen12}
    R^g_\gl f(x)
        = R^D_\gl f(x) + \rho(\gl)\bigl(2(e_\gl,f)+\gg f(0)\bigr) e_\gl(x),\qquad x\in\R_+,
\end{equation}
and in particular, $R^g_\gl$ has the integral kernel
\begin{equation}    \label{gen13}
\begin{split}
    r^g_\gl(x,dy)
        = r^D_\gl(x,y)\,dy + 2\rho(\gl)\, &e_\gl(x+y)\,dy\\
            &+ \gg\rho(\gl)\,e_\gl(x)\,\gep_0(dy), \qquad x,\,y\in\R_+,
\end{split}
\end{equation}
where $\rho(\gl)$ is given in~\eqref{gen6}, $e_\gl$ in~\eqref{el13}, and $r^D_\gl$
in~\eqref{el20}.
\end{corollary}

In order to determine the transition kernel of $B^g$ we have to compute the
inverse Laplace transform of the right hand side of equation~\eqref{gen13}.
This can be done with formulae provided and proved in appendix~\ref{app_LTHK}.
For $t>0$ and $x\in\R_+$ set
\begin{equation}    \label{gen14}
    g_{\gb,\gg}(t,x)
        = \frac{1}{\gg^2}\,\frac{1}{\sqrt{2\pi}}\,\int_0^t \frac{s+\gg x}{(t-s)^{3/2}}\,
            \exp\Bigl(-\frac{(s+\gg x)^2}{2\gg^2(t-s)}\Bigr)\,e^{-\gb s/\gg}\,ds,
\end{equation}
Then we get from corollary~\ref{a_LTHK_cor2}

\begin{corollary}   \label{gen_cor6}
For $t>0$, $x$, $y\in\R_+$, the transition kernel of $B^g$ is given by
\begin{equation}    \label{gen15}
    p^g(t,x,dy)
        = p^D(t,x,y)\,dy + 2 g_{\gb,\gg}(t,x+y)\,dy + \gg g_{\gb,\gg}(t,x)\,\gep_0(dy).
\end{equation}
\end{corollary}

\begin{remark}  \label{gen_rem7}
In appendix~\ref{app_LTHK} it is argued that for all $t>0$, $x\in\R_+$,
\begin{align*}
    \lim_{\gb\downarrow 0} g_{\gb,\gg}(t,x)
        &= g_{0,\gg}(t,x)\\
    \lim_{\gg\downarrow 0} g_{\gb,\gg}(t,x)
        &= g_{\gb,0}(t,x)\\
    \lim_{\gb\downarrow 0} \lim_{\gg\downarrow 0} g_{\gb,\gg}(t,x)
        &= \lim_{\gg\downarrow 0} \lim_{\gb\downarrow 0} g_{\gb,\gg}(t,x)
         = g(t,x),
\end{align*}
where $g_{\gb,0}(t,x)$ is defined in~\eqref{el24}, $g_{0,\gg}(t,x)$ in~\eqref{st22},
and $g(t,x)$ in~\eqref{el25}. Thus, in the limit that we let one or both parameters
$\gb$ and $\gg$ converge to zero from above, we recover the Brownian motions
discussed previously, as we should.
\end{remark}

\section{The General Brownian Motion on a Finite Interval} \label{sect_fii}
In this section we construct a Brownian motion on a finite interval following an
idea of Knight sketched in~\cite{Kn81}. There will be no loss of generality to choose
the interval $[0,1]$. From the discussion in the previous sections we may
consider as given the following building blocks
\begin{enum_a}
    \item a standard one-dimensional Brownian motion $O=(O_t,\,t\in\R_+)$,
    \item a Feller Brownian motion $A = (A_t,\,t\in\R_+)$ on $[0,+\infty)$
            which implements the boundary conditions~\eqref{Wbc_0} at $0$,
    \item and a Feller Brownian motion $B = (B_t,\,t\in\R_+)$ on $(-\infty, 1]$,
            which implements the boundary conditions~\eqref{Wbc_1} at $1$.
\end{enum_a}
Roughly speaking, the idea of construction is as follows. The process is obtained as
the standard Brownian motion $O$ starting at some point $x\in[0,1]$, and it runs
until it hits one of the boundary points $0$ or $1$, say $0$. From that moment on,
the process continues as the Feller Brownian motion $A$. If $A$ hits $1$ before its
lifetime expires, it continues from the moment of hitting $1$ as the Feller Brownian
motion $B$. If this process hits $0$ before its lifetime expires, another
independent copy of the Feller Brownian motion $A$ is started at this moment in $0$,
and so on.

This construction has similarity with some aspects involved in It\^o's excursion
theory~\cite{It71a}, cf.\ also~\cite{Sa86a, Sa86b, Ro83}. But the situation
considered here is much simpler, and can be handled with relatively elementary ---
even though somewhat lengthy --- arguments.

In the next subsection the construction sketched above is formalized. In
subsection~\ref{ssect_Markov} it is proved that the resulting process has the
simple Markov property, some technical details being deferred to
appendix~\ref{app_int}. The strong Markov property is proved in
subsection~\ref{ssect_SMarkov}, where we also compute the generator of the process.

\subsection{Construction}   \label{ssect_constr}
Let $I$ denote either of the intervals $(-\infty,1]$, $[0,+\infty)$ or $[0,1]$, and
write $C_\gD(\R_+,I)$ for the space of mappings $f$ from $\R_+$ into $I\cup\{\gD\}$
which are right continuous, have left limits in $I$, are continuous up to their
lifetime \begin{equation*}
    \zeta_f = \inf\,\{t\ge 0,\,f(t)= \gD\},
\end{equation*}
and which are such that $f(t)=\gD$ implies $f(s)=\gD$ for all $s\ge t$.
If $I=[0,1]$ we shall also simply write $C_\gD(\R_+)$.

Assume we are given a family $(M,\cM,\mu=(\mu_x,\,x\in[0,1]))$ of
probability spaces, on which a standard one-dimensional Brownian motion
$O=(O_t,\,t\in\R_+)$ with exclusively continuous paths is defined. Moreover, we
assume that $A=(A_t,\,t\in\R_+)$ and $B=(B_t,\,t\in\R_+)$ are two Feller Brownian
motions on $[0,+\infty)$, $(-\infty, 1]$ respectively, constructed from $O$ as
in the  previous sections, and implementing the boundary conditions~\eqref{Wbc_0},
\eqref{Wbc_1}, at $0$, $1$ respectively. It follows from the construction that $A$
and $B$ have all their paths in $\CD(\R_+,[0,+\infty))$ and $\CD(\R_+,(-\infty,1])$
respectively.

The hitting time of $c=0$ or $1$ by $O$ will be denoted by $\rho_c$, and we set
$\rho=\rho_0\land \rho_1$. We write $\gs$ for the hitting time of $1$ by $A$,
and $\tau$ for the hitting time of $0$ by $B$.

Below it will be convenient to have somewhat specialized versions of these processes
$A$ and $B$ at our disposal, and we construct these now. Let $(\Xi^1_c, \cC^1_c,
Q^1_c)$, $c=0$, $1$, be two probability spaces. Furthermore, let $A^1$ be a Feller
Brownian motion on $(\Xi^1_0, \cC^1_0, Q^1_0)$, with all paths belonging to
$\CD(\R_+,[0,+\infty))$, starting in $0$, and such that under $Q^1_0$, $A^1$ is
equivalent to $A$ under $\mu_0$. Analogously, $B^1$ is a Feller Brownian motion on
$(\Xi^1_1, \cC^1_1, Q^1_1)$, which under $Q^1_1$ is equivalent to $B$ under $\mu_1$, and
which has exclusively paths starting in $1$ and belonging to $\CD(\R_+,(-\infty,1])$.
Let $(\Xi^1,\cC^1, Q^1)$ denote the product probability space, and view $A^1$, and
$B^1$ as well as their hitting times $\gs^1$, $\tau^1$, of $1$, $0$ respectively, as
random variables on the product space.

Let
\begin{equation*}
    \bigr(\Xi^0,\cC^0,Q^0=(Q^0_x,\,x\in [0,1]), O^0, \rho_1^0, \rho_0^0,\rho^0\bigr)
\end{equation*}
be a copy of
\begin{equation*}
    \bigl(M,\cM,\mu=(\mu_x,\,x\in[0,1]), O, \rho_0, \rho_1,\rho\bigr),
\end{equation*}
and let
\begin{equation*}
    \bigl((\Xi^k,\cC^k,Q^k, A^k, \gs^k, B^k, \tau^k),\,k\in\N\bigr)
\end{equation*}
be a sequence of copies of
\begin{equation*}
    (\Xi^1,\cC^1,Q^1, A^1, \gs^1, B^1, \tau^1).
\end{equation*}
Define the product space $(\Xi,\cC)$ as
\begin{equation*}
    \Xi = \BigCart_{k\in\N_0} \Xi^k,\quad \cC = \bigotimes_{k\in\N_0}\cC^k.
\end{equation*}
Furthermore, on $(\Xi,\cC)$ we consider the family $Q=(Q_x,\,x\in[0,1])$
of product probability measures defined by
\begin{equation}    \label{int_eqi}
    Q_x = Q^0_x \otimes \Bigl(\bigotimes_{k\in\N} Q^k\Bigr).
\end{equation}
All stochastic processes and random variables introduced above on the measurable
spaces $(\Xi^k,\cC^k)$, $k\in\N_0$, are also viewed as defined on $(\Xi,\cC)$.

Let us set
\begin{equation}    \label{int_eqii}
    \Xi_0 = \{\rho^0_0<\rho^0_1\},\qquad \Xi_1 = \{\rho^0_1\le \rho^0_0\},
\end{equation}
and observe that both sets belong to $\cC$. (We remark that in the second event
equality occurs if and only if a path of $O$ never hits both endpoints of the
interval, i.e., if $\rho^0_0=\rho^0_1=+\infty$. Of course, these paths form
a negligible event.)

Next we define on $(\Xi,\cC,Q)$ a sequence $(S_k,\,k\in\N_0)$ of random variables
with values in $[0,+\infty]$ which will serve as crossover times for the pieces of
paths from which we shall construct the first version of the stochastic process we
aim at. We set $S_0=0$, $S_1=\rho^0$, and for $k\in\N$, $\go\in\Xi$,
\begin{equation}    \label{int_eqiii}
    S_{k+1}(\go) = S_k(\go) +
              \begin{cases}
                \gs^k(\go),    & \text{$\go\in \Xi_0$, $k$ odd,
                                            or $\go\in \Xi_1$, $k$ even,}\\[.5ex]
                \tau^k(\go),   & \text{otherwise.}
              \end{cases}
\end{equation}
By an application of the Borel--Cantelli--Lemma, for example, it is easy to see that
there exists a set $\Xi'\in\cC$, so that for every $x\in [0,1]$, $Q_x(\Xi')=1$,
and for all $\go\in\Xi'$, the sequence $(S_k(\go),\,k\in\N_0)$ is strictly increasing
to~$+\infty$.

Now we can construct a ``raw model'' $Y=(Y_t,\,t\in\R_+)$ of the stochastic process
desired. On $\Xi\setminus\Xi'$ we define $Y_t = \gD$ for all $t\in\R_+$. Given
$t\in\R_+$ and $\go\in\Xi'$, there exists $k\in\N_0$ so that $t\in [S_k(\go),S_{k+1}(\go))$.
Then we set $Y_t(\go) = O^0_t(\go)$ if $k=0$, and if $k\in\N$,
\begin{equation}    \label{int_eqiv}
    Y_t(\go) = \begin{cases}
                    A^k_{t-S_k}(\go),   &\text{$\go\in \Xi_0$, $k$ odd,
                                            or $\go\in \Xi_1$, $k$ even,}\\[1ex]
                    B^k_{t-S_k}(\go),   &\text{otherwise.}
               \end{cases}
\end{equation}
Moreover, we make the convention $Y_{+\infty}=\gD$.
(Note that due to their path properties, $A^k$ and $B^k$ are measurable stochastic
processes, so that $A^k_{t-S_k}$ and $B^k_{t-S_k}$ are well defined measurable
mappings, and hence $Y_t$ is indeed for every $t\ge 0$ a random variable on
$(\Xi,\cC)$ with values in $[0,1]\cup\{\gD\}$.) Figure~\ref{fig7} schematically
explains the idea: The piece of the (unrealistic) trajectory in green at the
beginning belongs to $O^0$, the blue pieces to trajectories of copies of $A$, and the
red pieces to trajectories of copies of $B$.
\begin{figure}[ht]
\begin{center}
    \includegraphics{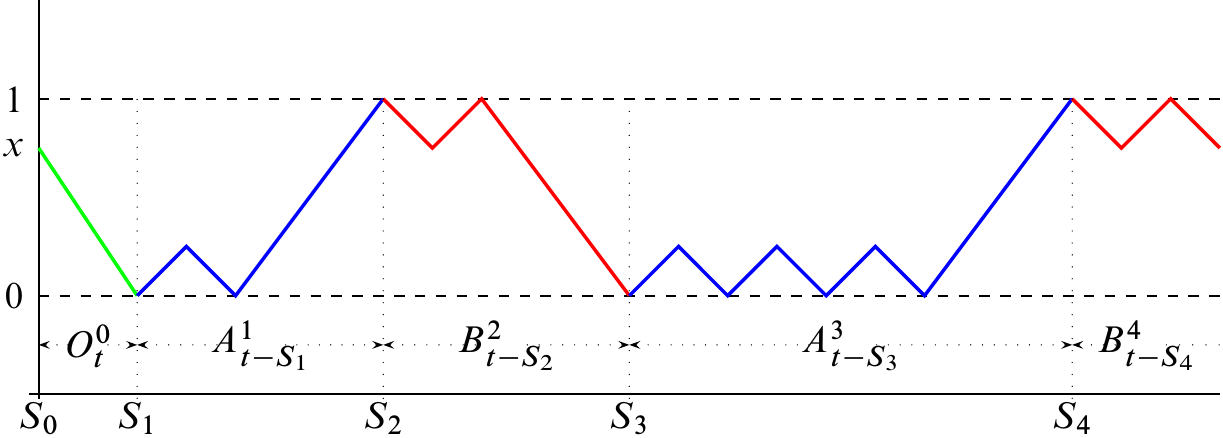}
    \caption{Construction of $Y$}\label{fig7}
\end{center}
\end{figure}

We want to argue that once $Y$ has reached the cemetery point $\gD$, it cannot
leave it. To this end, let us assume that $\go\in\Xi_0$, that $k$ is odd, and
that $t\in[S_k(\go),S_{k+1}(\go))$. Then $Y_t(\go) = A^k_{t-S_k}(\go)$.
Suppose that the lifetime of $A^k_{\,\cdot\,}(\go)$ expires before this
trajectory hits $1$, and hence $Y_s(\go)=\gD$ for some $s$ in the interval
$[S_k(\go),S_{k+1}(\go))$. Then $\gs^k(\go)=+\infty$, and therefore also
$S_{k+1}(\go)=+\infty$. Thus there are no more finite crossover times for this
$\go$, and hence after $s$ the trajectory $Y_{\,\cdot\,}(\go)$ is identical to the
trajectory of $A^k_{\,\cdot\,-S_k}(\go)$, namely constantly equal to $\gD$, as it should.
It is clear that in all the other cases one can argue in the same way.

Observe that by construction all paths of $Y$ belong to $\CD(\R_+)$, and in
particular they are all right continuous. Hence $Y$ is a measurable stochastic
process, and it follows that if $R$ is any random time (i.e., a random variable on
$(\Xi,\cC)$ with values in $\R_+\cup\{+\infty\}$) the stopped process
$Y_{\,\cdot\,\land R}=(Y_{t\land R},\,t\ge0)$ is a well defined stochastic process.

Consider the measurable space $(\gO,\cA)$, where $\gO=\CDR$, and $\cA$ is the
$\gs$--algebra generated by the cylinder sets of $\CDR$. Clearly, $Y$, considered
as a mapping from $\Xi$ into $\gO=\CDR$ is $\cC/\cA$--measurable. We denote by
$P_x$, $x\in [0,1]$, the image of $Q_x$ under $Y$. Furthermore, by $X=(X_t,\,t\in\R_+)$
we denote the canonical coordinate process on $\CDR$. We make the convention
$X_{+\infty}=\gD$.

Assume that we are given a sequence $(U_k,\,k\in\N)$ of random times on $(\gO,\cA)$,
and a sequence of random times $(R_k,\,k\in\N)$ on $(\Xi,\cC)$, which are such that
for every $k$, $R_k = U_k\comp Y$ holds true. Set $U_\infty=R_\infty=+\infty$, and
consider the families of random variables
\begin{equation*}
    \cY = \bigl\{Y_{t\land R_k},\,R_l,\,t\in\R_+,\,k,l\in\N_\infty\bigr\},\
    \cX = \bigl\{X_{t\land U_k},\,U_l,\,t\in\R_+,\,k,l\in\N_\infty\bigr\},
\end{equation*}
where $\N_\infty = \N\cup\{+\infty\}$. Then we obtain the following trivial
statement, which for ease of later reference we record here as a lemma.

\begin{lemma}   \label{int_lemi}
For every $x\in [0,1]$, the family $\cY$ has the same finite dimensional
distributions under $Q_x$, as $\cX$ has under $P_x$, that is, for all $m$, $n\in \N$,
$t_1$, \dots, $t_m\in\R_+$, $i_1$, \dots, $i_m$, $j_1$, \dots, $j_n\in\N_\infty$,
$B_1$, \dots, $B_m\in\cB([0,1]\cup\{\gD\})$, $C_1$, \dots,
$C_n\in\cB(\R_+\cup\{+\infty\})$,
\begin{equation}    \label{int_eqv}
\begin{split}
    Q_x\bigl(&Y_{t_1\land R_{i_1}}\in B_1,\dotsc, Y_{t_m\land R_{i_m}}\in B_m,
                R_{j_1}\in C_1,\dotsc, R_{j_n}\in C_n\bigr)\\
        &= P_x\bigl(X_{t_1\land U_{i_1}}\in B_1,\dotsc, X_{t_m\land U_{i_m}}\in B_m,
                U_{j_1}\in C_1,\dotsc, U_{j_n}\in C_n\bigr)
\end{split}
\end{equation}
holds true.
\end{lemma}

We shall denote by $\FX=(\FX_t,\,t\in\R_+)$ the natural filtration of $X$,
and set $\cF^X_\infty = \gs(X_t,\,t\in\R_+)$.

Below it will be convenient to work with both versions, $X$ and $Y$, of the
stochastic process simultaneously. The reason is that $Y$, having been constructed on
product spaces, inherits a strong independence structure. This will turn out to be
very convenient to employ. On the other hand, on the path space $\gO=\CDR$ we have
the canonical families $\theta = (\theta_t,\,t\in\R_+)$, $\gg = (\gg_t,\,t\in\R_+)$
of shift and stop operators, respectively:
\begin{align*}
    \theta_t(\go)(s) &= \go(s+t),\qquad s,\,t\in\R_+,\,\go\in\gO\\
    \gg_t(\go)(s)    &= \go(s\land t),\qquad s,\,t\in\R_+,\,\go\in\gO,
\end{align*}
and we get
\begin{align*}
    X_s\comp \theta_t &= X_{s+t},\\
    X_s\comp\gg_t     &= X_{s\land t}.
\end{align*}
This implies that on $\CDR$ we can use Galmarino's theorem (cf., e.g.,
\cite[p.~458]{Ba91}, \cite[p.~86]{ItMc74}, \cite[p.~43~ff]{Kn81},
\cite[p.~45]{ReYo91}), which in particular gives a very useful characterization of
the $\gs$--algebra $\FX_{T}$ in the case when $T$ is a stopping time with respect to
$\FX$:
\goodbreak

\begin{lemma}[Galmarino's Theorem]   \label{int_lemii}
{\ }
\begin{enum_a}
    \item Let $T$ be an $\overline\R_+$--valued random variable on $(\gO,\cA)$.
            Then $T$ is an $\FX$--stopping time if and only if the following
            statement holds true: for all $\go_1$, $\go_2\in\gO$, $t\in\R_+$,
            $\gg_t(\go_1) = \gg_t(\go_2)$, and $T(\go_1)\le t$ imply
            $T(\go_1)=T(\go_2)$.
    \item Let $T$ be an $\FX$--stopping time, and let $\FX_T$ be the
            $\gs$-algebra of the $T$--past. Then $\FX_T = \gs(X_{t\land
            T},\,t\in\R_+)$.
\end{enum_a}
\end{lemma}

In order not to burden our notation too much, for both stochastic processes,
$Y$ and $X$, we shall denote the lifetimes by $\zeta$ and the hitting times of sets
$C\subset \R$ by $H_C$, and if $C=\{x\}$, $x\in\R$, by $H_x$ for simplicity.
Also, in both cases we shall write for the expectation $E_x(\,\cdot\,)$,
$x\in [0,1]$. Whenever necessary, we add a corresponding superscript to avoid confusion.

We bring in the notation
\begin{equation*}
    c^* = \begin{cases}
            1, & \text{if $c=0$},\\
            0, & \text{if $c=1$},
          \end{cases}
\end{equation*}
and set $T_0=0$, $T_1 = H_0\land H_1$. For $k\in\N$, we define on the event
$\{T_k<+\infty\}$ and with $c=X_{T_k}$
\begin{equation}    \label{int_eqvi}
    T_{k+1} = T_k + \begin{cases}
                            H_c\comp \theta_{T_k}, & \text{if $k$ is odd,}\\[1ex]
                            H_{c^*}\comp \theta_{T_k}, & \text{if $k$ is even.}
                    \end{cases}
\end{equation}
On $\{T_k=+\infty\}$ put $T_{k+1}=+\infty$. Obviously, $S_k = T_k\comp Y$, $k\in\N_0$,
and it follows that under $Q_x$, $x\in [0,1]$, $S_k$ has the same law as $T_k$ under
$P_x$. Hence a.s.\ the sequence $(T_k,\,k\in\N_0)$ is strictly increasing to~$+\infty$.
Moreover, it is easy to check --- for example with the help of lemma~\ref{int_lemii} ---
that for every $k\in\N_0$, $T_k$ is an $\cF^X$--stopping time.

Consider the process $\hat X$ defined as $X$ being stopped at the stopping time
$T_1$, i.e., $X$ with absorption in the endpoints $0$, $1$ of the interval $[0,1]$.
By lemma~\ref{int_lemi} we find that $\hat X$ is equivalent to $Y$ with absorption
in $\{0,1\}$, and hence by the construction of $Y$, $\hat X$ is equivalent to
a standard Brownian motion with absorption in $\{0,1\}$. Then it is easy to
check that $\hat X$ is a Markov process relative to $\cF^X$. Futhermore, if instead
of stopping we consider a process $\tilde X$ which is defined as $X$ with killing
at the time $T_1$, we analogously get that $\tilde X$ is equivalent to a standard
Brownian motion which is killed at the boundary of the interval $[0,1]$, and $\tilde X$,
too, has the Markov property relative to $\cF^X$.

\subsection{Simple Markov Property} \label{ssect_Markov}
In this subsection we prove the following

\begin{theorem} \label{int_thmiii}
$X$ is a normal, homogeneous Markov process.
\end{theorem}

The remainder of this subsection is devoted to the proof of theorem~\ref{int_thmiii}. We
first prepare some lemmas, whose statements heuristically are quite obvious.
Unfortunately, their formal proofs are rather technical, and therefore they are
deferred to appendix~\ref{app_int}.

Throughout this subsection $\cF=(\cF_t,\,t\ge 0)$ denotes the filtration $\cF^X$ of
$X$, $f$ is a bounded, measurable function on $[0,1]$ (extended to $[0,1]\cup\{\gD\}$
in the usual way via $f(\gD)=0$), $s$, $t$ belong to $\R_+$, and $x$ to $[0,1]$.
The first two lemmas, lemma~\ref{int_lemiv} and lemma~\ref{int_lemv} below, form the
key to the proof of theorem~\ref{int_thmiii}.

\begin{lemma}   \label{int_lemiv}
$X$ is a strong Markov process with respect to the stopping times $T_k$, $k\in\N$.
That is, for every bounded $\cF_\infty$--measurable random variable $V$, all
$k\in\N$, $x\in[0,1]$,
\begin{equation}    \label{int_eqvii}
    E_x\bigl(V\comp\theta_{T_k}\cond \cF_{T_k}\bigr)
        = E_{X_{T_k}}(V),\qquad \text{$P_x$--a.s.\ on $\{X_{T_k}\ne \gD\}$}.
\end{equation}
\end{lemma}

We write $\tilde X^c$ for the process $X$ with killing at $c=0$ or $c=1$. Thus
\begin{equation*}
    E_x\bigl(f(\tilde X^c_t)\bigr) = E_x\bigl(f(X_t);\,t<H_c\bigr),\qquad c=0,\,1.
\end{equation*}
Similarly, $\tilde A$ and $\tilde B$ denote the processes $A$, $B$, with killing
at $1$, $0$ respectively.

\begin{lemma}   \label{int_lemv}
Under $P_x$, $x\in[0,1]$, $\tilde X^1$ is equivalent to $\tilde A$ under $\mu_x$, and
$\tilde X^0$ is equivalent to $\tilde B$ under $\mu_x$. In particular, $\tilde X^1$ and
$\tilde X^0$ are Markov processes.
\end{lemma}

In appendix~\ref{app_int} we derive from these two lemmas the following three results.
\begin{lemma}   \label{int_lemvi}
On the event $\{0\le t< T_1,\,X_t\ne \gD\}$,
\begin{equation}    \label{int eqviii}
    E_x\bigl(f(X_{s+t})\cond \cF_t\bigr)
        = E_{X_t}\bigl(f(X_s)\bigr),\qquad \text{$P_x$--a.s., $x\in [0,1]$}.
\end{equation}
\end{lemma}

\begin{lemma}   \label{int_lemvii}
On the event $\{X_{T_1}=c,\,T_k\le t<T_{k+1},\,X_t\ne \gD\}$, $k\in\N$, the following
holds true $P_x$--a.s., $x\in [0,1]$:
\begin{equation}    \label{int_eqix}
\begin{split}
    E_x\bigl(&f(X_{s+t});\,T_k\le s+t <T_{k+1}\cond \cF_t\bigr)\\
        &= E_{X_t}\bigl(f(X_s);\,s<T_1\bigr)\\
        &\hspace{4em}
            + \begin{cases}
E_{X_t}\bigl(f(X_s);\,T_1\le s <T_2,\,X_{T_1}=c\bigr), & \text{$k$
odd,}\\[1ex]E_{X_t}\bigl(f(X_s);\,T_1\le s <T_2,\,X_{T_1}=c^*\bigr), & \text{$k$
even.} \end{cases}
\end{split}
\end{equation}
\end{lemma}

\begin{lemma}   \label{int_lemviii}
On the event $\{X_{T_1}=c,\,T_k\le t<T_{k+1},\,X_t\ne \gD\}$, $k\in\N$, the following
holds true $P_x$--a.s., $x\in [0,1]$:
\begin{equation}    \label{int_eqx}
\begin{split}
    E_x\bigl(&f(X_{s+t});\,T_{k+1}\le s+t \cond \cF_t\bigr)\\
        &= \begin{cases}
E_{X_t}\bigl(f(X_s);\,T_2\le s, X_{T_1}=c, \text{ or } T_1\le s,\,X_{T_1}=c^*\bigr),
& \text{$k$ odd},\\[1ex]
E_{X_t}\bigl(f(X_s);\,T_2\le s, X_{T_1}=c^*, \text{ or } T_1\le s,\,X_{T_1}=c\bigr),
& \text{$k$ even}.
          \end{cases}
\end{split}
\end{equation}
\end{lemma}

Now we are ready to prove theorem~\ref{int_thmiii}. Write
\begin{align*}
    E_x\bigl(f(X_{s+t})\cond \cF_t\bigr)
&= \sum_{k=0}^\infty E_x\bigl(f(X_{s+t});\,T_k\le t < T_{k+1}\cond \cF_t\bigr)\\&=
\sum_{k=0}^\infty E_x\bigl(f(X_{s+t})\cond \cF_t\bigr)\,1_{\{T_k\le t <
T_{k+1}\}},
\end{align*}
and therefore it is sufficient to prove the formula
\begin{equation}    \label{int_eqxi}
E_x\bigl(f(X_{s+t})\cond \cF_t\bigr) = E_{X_t}\bigl(f(X_s)\bigr),\qquad
\text{$P_x$--a.s.}
\end{equation}
on each of the events $\{T_k\le t< T_{k+1},\, X_t\ne \gD\}$, $k\in\N_0$. For $k=0$
this is being taken care of by lemma~\ref{int_lemvi}. Hence we consider $k\in\N$
from now on. On each of the events $\gL_k=\{T_k\le t < T_{k+1}\}$ we have $T_1\le t$. This
implies $\cF_{T_1}\cap \gL_k\subset \cF_t\cap\gL_k$, and that $X_{T_1}$ is
$\cF_t$--measurable. This is so, because it is $\cF_{T_1}$--measurable (e.g.,
\cite[Proposition~I.4.9]{ReYo91}, which can be applied because $X$ being $\cF$-adapted
and having right continuous paths is progressively measurable relative to $\cF$).
Therefore we get on $\{T_k\le t <T_{k+1},\,X_t\ne \gD\}$
\begin{align*}
    E_x\bigl(f(&X_{s+t})\cond \cF_t\bigr)\\
        &= E_x\bigl(f(X_{s+t});\,X_{T_1}=0\cond \cF_t\bigr)
                +  E_x\bigl(f(X_{s+t});\,X_{T_1}=1\cond \cF_t\bigr)\\
        &= E_x\bigl(f(X_{s+t})\cond \cF_t\bigr)\,1_{\{X_{T_1}=0\}}
                +  E_x\bigl(f(X_{s+t})\cond \cF_t\bigr)\,1_{\{X_{T_1}=1\}}.
\end{align*}
Hence it is sufficient to prove equality~\eqref{int_eqxi} on each of the events
$\{T_k\le t < T_{k+1},\,X_t\ne \gD\}\cap \{X_{T_1}=0\}$, and
$\{T_k\le t < T_{k+1},\,X_t\ne \gD\}\cap \{X_{T_1}=1\}$, $k\in\N$. We shall only
consider the first of these under the additional assumption that $k$ is odd. The
case when $k$ is even, and the second term above can be treated in the same manner.

Fix $k\in\N$, $k$ odd. On $\{T_k\le t < T_{k+1},\,X_t\ne \gD\}\cap \{X_{T_1}=0\}$
we write
\begin{equation*}
\begin{split}
        E_x\bigl(f(X_{s+t})\cond \cF_t\bigr)
            = E_x\bigl(f(X_{s+t});\,&T_k\le s+t < T_{k+1}\cond \cF_t\bigr)\\
               & + E_x\bigl(f(X_{s+t});\,T_{k+1}< s+t\cond \cF_t\bigr).
\end{split}
\end{equation*}
To the first of the terms on the right hand side we apply lemma~\ref{int_lemvii}, to
the second lemma~\ref{int_lemviii}, and we obtain ($P_x$--a.s.)
\begin{align*}
    E_x\bigl(f(&X_{s+t})\cond \cF_t\bigr)\\
		&= E_{X_t}\bigl(f(X_s);\,\{T_1\le s < T_2, X_{T_1}=0\}\uplus \{s<T_1\}\bigr)\\
		&\hspace{4em} + E_{X_t}\bigl(f(X_s);\,\{T_2\le s, X_{T_1}=0\}
                                    \uplus \{T_1\le s, X_{T_1}=1\}\bigr)\\
         &= E_{X_t}\bigl(f(X_s);\,\{s<T_1\}\uplus\{T_1\le s,X_{T_1}=0\}
                                    \uplus\{T_1\le s, X_{T_1}=1\}\bigr)\\
         &=  E_{X_t}\bigl(f(X_s);\,\{s<T_1\}\uplus\{T_1\le s\}\bigr)\\
         &= E_{X_t}\bigl(f(X_s)\bigr),
\end{align*}
which concludes the proof of theorem~\ref{int_thmiii}.

\subsection{Strong Markov Property and Generator}   \label{ssect_SMarkov}
In this section we prove the strong Markov property of $X$, and calculate
its generator.

\begin{lemma}	\label{int_lemix}
$X$ is a Feller process.
\end{lemma}

\begin{proof}	
It is well-known, that it is sufficient to prove~(i) that the resolvent of $X$
preserves $C([0,1])$, and~(ii) that for all $f\in C([0,1])$, $x\in[0,1]$,
$E_x\bigl(f(X_t)\bigr)$ converges to $f(x)$ as $t$ decreases to $0$.
(For example, a complete proof can be found in~\cite{BMMG1}.) Statement~(ii) immediately
follows by an application of the dominated convergence theorem and the fact
that $X$ is a normal process with right continuous paths.

Suppose that $f\in C([0,1])$, and consider the resolvent $R=(R_\gl,\,\gl>0)$ of $X$.
Because $X$ is a strong Markov process with respect to the stopping times $T_k$, $k\in\N$,
(cf.\ lemma~\ref{int_lemiv}) we get the first passage time formula (see, e.g.,
appendix~\ref{app_FPTF}) for $T_1$ in the form
\begin{align*}
 	R_\gl f(x)
		&= E_x\Bigl(\int_0^{T_1} e^{-\gl t} f(X_t)\,dt\Bigr)
			+ E_x\bigl(e^{-\gl T_1}\,R_\gl f(X_{T_1})\bigr)\\
		&= E_x\Bigl(\int_0^{T_1} e^{-\gl t} f(X_t)\,dt\Bigr)\\
		&\hspace{2em}	+ E_x\bigl(e^{-\gl T_1};\,X_{T_1}=0\bigr)\,R_\gl f(0)
			+ E_x\bigl(e^{-\gl T_1};\,X_{T_1}=1\bigr)\,R_\gl f(1).
\end{align*}
Lemma~\ref{int_lemv} states that $X$ stopped at $T_1$ is equivalent to a standard
Brownian motion $B$ which is stopped when reaching any of the boundary points of the
interval $[0,1]$. Therefore we get for $x\in (0,1)$
\begin{equation}	\label{eq4i}
\begin{split}
 	R_\gl f(x)
        = R^D_\gl f(x) &+ E_x\bigl(e^{-\gl H^B_0};\,H^B_0<H^B_1\bigr)\,R_\gl f(0)\\
					 &+ E_x\bigl(e^{-\gl H^B_1};\,H^B_1<H^B_0\bigr)\,R_\gl f(1),
\end{split}
\end{equation}
where $H^B_c$, $c=0$, $1$, denotes the hitting time of $c$ by $B$, and $R^D$
denotes the resolvent of $B$ with Dirichlet boundary conditions at $0$, $1$
(which is equal to the resolvent of $X$ with Dirichlet boundary conditions
at $0$, $1$, so that the notation is consistent). All expressions above are
known explicitly (see, e.g., \cite{DyJu69, ItMc74}): $R_\gl^D$ has the integral kernel
\begin{equation*}
 	r_\gl^D(x,y)
		= \frac{1}{\sgl}\sum_{k\in\Z}\Bigl(e^{-\sgl |x-y+2k|}
				- e^{-\sgl |x+y+2k|}\Bigr),
\end{equation*}
while
\begin{align*}
	E_x\bigl(e^{-\gl H^B_0};\,H^B_0<H^B_1\bigr)
		&= \frac{\sinh\bigl(\sgl(1-x)\bigr)}{\sinh\bigl(\sgl\bigr)},\\	
	E_x\bigl(e^{-\gl H^B_1};\,H^B_1<H^B_0\bigr)
		&= \frac{\sinh\bigl(\sgl x\bigr)}{\sinh\bigl(\sgl\bigr)}.	
\end{align*}
It is now obvious that $R_\gl$ maps $C([0,1])$ into itself, and the proof
is finished.
\end{proof}

Now we can apply standard results (cf., e.g, \cite[Theorem~III.3.1]{ReYo91}, or
\cite[Theorem~III.15.3]{Wi79}) to obtain

\begin{corollary}	\label{int_corx}
$X$ is a strong Markov process relative to its natural filtration. In particular,
$X$ is a Brownian motion on $[0,1]$ in the sense of definition~\ref{def_BM}.
\end{corollary}

\begin{remark}	\label{int_remxi}
As discussed in, e.g., \cite{KaSh91, ReYo91,
Wi79}, on the basis of corollary~\ref{int_corx} (and the path properties of $X$) one
can prove that $X$ is strongly Markovian relative to the canonical right continuous
augmentation of its natural filtration.
\end{remark}

It remains to calculate the domain of the generator of $X$. We only discuss the
boundary point $0$, the arguments for $1$ are similar. From lemma~\ref{int_lemv} we
know that $X$ with killing at $1$ is equivalent to $A$ with killing at $1$.
Therefore, starting in $0$, if $A$ has $0$ as a trap, then so has $X$, and if $A$
jumps from $0$ to $\gD$ after an exponential holding time, then the same is true for
$X$. The remaining case is the one, where $A$ and hence $X$ leave $0$ immediately
(into $(0,1)$), and in this case we can use Dynkin's formula to compute the
generator of $X$, because it is strongly Markovian (corollary~\ref{int_corx}). But
this only involves the behavior of $X$ before leaving an arbitrarily small
neighborhood of $0$, and hence at $0$ the generator of $X$ is computed in the same
way as for $X$ with killing at $1$, that is for the process $A$. Thus we have proved

\begin{theorem}	\label{int_thmxii}
$X$ is a Brownian motion on $[0,1]$ in the sense of definition~\ref{def_BM},
and its generator is the second derivative on $C^2([0,1])$ with boundary
conditions at $0$ and $1$ given by~\eqref{Wbc_0} and \eqref{Wbc_1} respectively.
\end{theorem}

\begin{appendix}
\section{Killing With A PCHAF}  \label{app_killing}
This appendix provides a short account on the killing of a Markov process
$X=(X_t,\,t\ge 0)$ via a perfect continuous homogeneous additive functional (PCHAF)
$A=(A_t,\,t\ge 0)$ of the process. Killing via the Feynman--Kac functional
associated with $A$ has been treated in detail for example in the books by
Blumenthal--Getoor~\cite{BlGe68} and Dynkin~\cite{Dy65a}. Here we use a mechanism
which has also been employed in~\cite{Kn81, KaSh91}, and which is slightly different
--- in a sense somewhat more explicit than the method of the Feynman--Kac functional
just mentioned. In particular, we shall show how the simple and the strong
Markov properties are preserved under this way of killing. This has also been argued
in~\cite{Kn81} for the case of Brownian motions, but we find the arguments given
there quite difficult to follow. This appendix is based on the treatment in~\cite{BlGe68}
of killing with the Feynman--Kac weight which, however, needs a number of
non-trivial modifications.

Throughout this appendix we shall use the terminology and notations
of~\cite{BlGe68}, except where otherwise indicated. (However, in order to be
consistent with the other parts of this paper, there will be a few trivial
variations in the notation, such as changing some superscripts into subscripts etc.)

Assume that $X=(\gO,\cM,\cM_t,X_t,\theta_t, P_x)$ is a temporally homogeneous Markov
process with state space $(E,\cE)$ as in section~I.3 of~\cite{BlGe68}. It will be
convenient --- and without loss of generality for the purposes of this paper --- to
assume that for all $x\in E$, $X$ has $P_x$--a.s.\ infinite lifetime. The natural
filtration of $X$ is denoted by $\cF^0=(\cF^0_t,\,t\ge 0)$, its universal
augmentation is denoted by $\cF=(\cF_t,\,t\ge 0)$. In particular we have for all
$t\ge 0$, $\cF^0_t\subset \cM_t$. We set $\cF^0_{\infty}
= \gs\bigl(\cF^0_t,\,t\in\R_+\bigr)$, and similarly for $\cF_{\infty}$ and
$\cM_{\infty}$.

For the definition of a PCHAF we take the point of view of Williams~\cite{Wi79},
which has the advantage that the meaning of the word ``perfect'' is valid
in the sense of both, Dynkin~\cite{Dy65a} (i.e., the additive functional is adapted to $X$),
and of Blumenthal--Getoor~\cite{BlGe68} (i.e., the exceptional set where the additivity
relation does not hold does not depend on the time variables):

\begin{definition}  \label{a_kill_def1}
A \emph{perfect continuous homogeneous additive functional} (PCHAF) of $X$ is an
$\cF$--adapted process $A$ with values in $\R_+$ such that for some set
$\gO_0\in\cM$ with $P_x(\gO_0)=1$ for all $x\in E$, the following two properties
hold for all $\go\in\gO_0$:
\begin{enum_i}
    \item $t\mapsto A_t(\go)$ is continuous from $\R_+$ into itself, non-decreasing,
            and $A_0(\go)=0$;
    \item for all $s$, $t\ge0$, $A_{s+t}(\go) = A_s(\go) + A_t\comp \theta_s(\go)$.
\end{enum_i}
\end{definition}

There will be no loss of generality for the sequel to assume that $\gO_0=\gO$: for
example, for the purposes of this note one could put $A_t =0$, for all $t\ge 0$, on
$\complement\gO_0$. Moreover, we make the convention that for all $\go\in\gO$,
$A_{+\infty}(\go)=+\infty$.

\begin{remark}  \label{a_kill_rem2}
Since $t\mapsto A_t(\go)$ is continuous on $\R_+$ for all $\go\in\gO$, we find that
$A$ is $\cB(\Rbp)\otimes\cA / \cB(\Rbp)$--measurable. Therefore, if $T$ is any
random time, i.e., a random variable with values in $\Rbp$, then $A_T(\go) =
A_{T(\go)}(\go)$ is a well defined random variable. Moreover,  we also get the
additivity relation in the form
\begin{equation}    \label{a_kill1}
    A_{t+T(\go)}(\go)
        = A_{T(\go)}(\go) + A_t\comp \theta_{T(\go)}(\go),\qquad \go\in\gO,\,t\in\R_+.
\end{equation}
\end{remark}

\begin{examples}    \label{a_kill_ex3}
Consider a standard Brownian family $B=(B_t,\,t\ge 0)$ on $\R^d$, $d\in\N$, and
suppose that $B$ has exclusively continuous paths. For example, one could realize
$B$ as the canonical coordinate mapping on Wiener space. Assume that $V$ is any
non-negative, measurable function on $\R^d$. Consider
\begin{equation}    \label{a_kill2}
    A_t = \int_0^t V(B_s)\,ds,\qquad t\ge 0.
\end{equation}
Then $A$ is obviously a PCHAF. Moreover, in the literature (e.g., \cite{ItMc74,
Kn81, KaSh91, ReYo91}) it is proved that for $d=1$ and every $x\in\R$, the local
time $L_t(x)$, $t\ge 0$, of $B$ in $x$ is a PCHAF.
\end{examples}

\subsection{Killing} \label{aappZi}
Define a new family of filtered probability spaces as follows.
\begin{equation}    \label{a_kill3}
    \hgO = \gO\times\Rbp,
\end{equation}
and denote a typical element in $\hgO$ by $\hgo = (\go,s)$, $\go\in\gO$, $s\in\Rbp$.
Define the following $\gs$--algebras on $\hgO$:
\begin{equation}    \label{a_kill4}
    \hcM = \cM\otimes\cB(\Rbp),
\end{equation}
and
\begin{equation}    \label{a_kill5}
    \hcM^0 = \cM\times \Rbp.
\end{equation}
In order to avoid that our notation becomes too unwieldy, we continue to denote the
canonical extension of any random variable $Z$ defined on $(\gO,\cM)$ to
$(\hgO,\hcM)$ by the same symbol, i.e., we simply set $Z(\hgo) = Z(\go)$, for
$\hgo=(\go,s)\in\hgO$. Clearly, for every $t\ge 0$ the extensions of $X_t$ and $A_t$
are in $\hcM^0/\cE$ and $\hcM^0/\cB(\R_+)$, respectively.

For $x\in E$ set
\begin{equation}    \label{a_kill6}
    \hPx = P_x\otimes P_e,
\end{equation}
where $P_e$ is the exponential law on $(\Rbp, \cB(\Rbp))$, i.e., it has density
$\exp(-s)\,1_{\R_+}(s)$, $s\in\Rbp$. Furthermore for $t\ge 0$, $\hgo\in\hgO$,
$\hgo=(\go,s)$, define
\begin{equation}    \label{a_kill7}
    \hth_t(\hgo) = \bigl(\theta_t(\go),s-A_t(\go)\land s\bigr).
\end{equation}
Then we have for all $s$, $t\ge 0$, $X_s\comp \hth_t = X_{s+t}$, and
$A_{s+t} = A_t + A_s\comp \hth_t$.
Moreover, if we set
\begin{equation}    \label{a_kill8}
    S(\hgo) = s,\qquad \hgo=(\go,s)\in\hgO
\end{equation}
then for all $t\ge 0$
\begin{equation}    \label{a_kill9}
    S\comp \hth_t = S - A_t\land S.
\end{equation}

Define
\begin{equation} \label{a_kill10}
    \zeta = \inf\,\{t\ge 0,\, A_t > S\}.
\end{equation}

Set
\begin{equation}    \label{a_kill11}
    \hX_t = \begin{cases}
              X_t,  &\text{if $t<\zeta$},\\
              \gD,  &\text{if $t\ge \zeta$},
            \end{cases}
\end{equation}
where again $\gD$ is a cemetery point.

Finally, for every $t\ge 0$ we define a family $\hcM_t\subset\hcM$ of subsets of $\hgO$
as follows. A set $\hgL\in\hcM$ belongs to $\hcM_t$ if and only if there
exists a set $\gL_t\in\cM_t$ so that
\begin{equation}    \label{a_kill12}
    \hgL \cap \zt = (\gL_t\times\Rbp) \cap \zt.
\end{equation}
If $\hgL=\{\zeta>t\}$, $t\ge 0$, we only need to choose $\gL_t=\gO\in\cM_t$ to see
that $\{\zeta>t\}\in\hcM_t$, i.e., by construction of $(\hcM_t,\,t\ge 0)$, $\zeta$
is a stopping time relative to $(\hcM_t,\,t\ge 0)$.

It is straightforward to show the following.
\begin{lemma}   \label{a_kill_lem4}
$(\hcM_t,\,t\ge 0)$ is a filtration of sub--$\gs$--algebras of $\hcM$, and $\hX$
is adapted to this filtration. Furthermore, if $(\cM_t,\,t\ge 0)$ is right
continuous then so is $(\hcM_t,\,t\ge 0)$.
\end{lemma}

In the sequel we suppose that $T$ is a random time defined on $(\gO,\cM)$, which as
above is extended to $(\hgO,\hcM)$ in the canonical way. Let us define
\begin{equation}\label{a_kill13}
    \bigl\{S\ge A_{T+}\bigr\} = \bigcup_{\gep>0} \bigl\{S\ge A_{T+\gep}\bigr\}.
\end{equation}
The proof of the following lemma is routine and therefore omitted.

\begin{lemma}   \label{a_kill_lem5}
$A_\zeta = S$ on the event $\{\zeta<+\infty\}$. Moreover, the following relations
hold true:
\begin{align}
    \{\zeta\ge T\} &= \{S\ge A_T\} \label{a_kill14},\\
    \{\zeta>T\}    &= \{S \ge A_{T+}\}\cap\{T<+\infty\} \label{a_kill15}.
\end{align}
\end{lemma}

Also the proof of the next lemma is left to the reader.
\begin{lemma}   \label{a_kill_lem6}
Let $T$ be a random time as above. Then on $\{\zeta\ge T\}$
\begin{equation}    \label{a_kill16}
    \zeta = T + \zeta\comp \hth_T.
\end{equation}
\end{lemma}

\begin{lemma}   \label{a_kill_lem7}
Let $T$ be a random time as above, and let $x\in E$. Then
\begin{equation}    \label{a_kill17}
    \hPx\bigl(\zeta > T \cond \hcM^0\bigr) = e^{-A_T}
\end{equation}
holds true $\hPx$--a.s.
\end{lemma}

\begin{proof}
Observe that the right hand side of equation~\eqref{a_kill17} is $\hcM^0$ measurable.
Let $\hgL\in \hcM^0$, i.e., $\hgL = \gL\times \Rbp$ with $\gL\in\cM$. Without loss
of generality we may assume that $\gL\subset \{T<+\infty\}$. Using Fubini's theorem,
lemma~\ref{a_kill_lem5}, and the continuity of $t\mapsto A_t(\go)$, $\go\in\gO$, on
$\R_+$, we can calculate in the following way:
\begin{align*}
    \hPx\bigl(\zt[T]\cap \hgL\bigr)
        &= \hPx\bigl(\{A_{T+}\le S\}\cap\hgL\bigr)\\[2ex]
        &= \lim_{\gep\da 0} \hPx\bigl(\{A_{T+\gep}\le S\}\cap\hgL\bigr)\\[2ex]
        &= \lim_{\gep\da 0} \int_{\gL\times\Rbp} 1_{[A_{T+\gep},+\infty]}\comp S\,
                                    d\bigl(P_x\otimes P_e\bigr)\\[2ex]
        &= \lim_{\gep\da 0} \int_\gL\Bigl(\int_{A_{T+\gep}(\go)}^\infty
                                e^{-s}\,ds\Bigr)\,dP_x(\go)\\[2ex]
        &= \lim_{\gep\da 0} E_x\bigl(e^{-A_{T+\gep}}; \gL\bigr)\\[2ex]
        &= E_x\bigl(e^{-A_T}; \gL\bigr)\\[2ex]
        &= \hEx\bigl(e^{-A_T}; \hgL\bigr).  \qedhere
\end{align*}
\end{proof}

Recall our convention that every function $f$ on $E$ is extended
to $E_\gD$ with $f(\gD)=0$.

\begin{corollary}[Feynman--Kac-formula]   \label{a_kill_cor8}
For every bounded measurable function $f$ on $E$, and all $t\ge 0$, $x\in E$, the
following formula holds true
\begin{equation}    \label{a_kill18}
    \hEx\bigl(f(\hX_t)\bigr) = E_x\bigl(f(X_t)\,e^{-A_t}\bigr).
\end{equation}
\end{corollary}

\begin{proof}
Since $f(X_t)\in\hcM^0/\cB(\R)$, we get from lemma~\ref{a_kill_lem7} (with the choice
$T=t$)
\begin{align*}
    \hEx\bigl(f(\hX_t)\bigr)
        &= \hEx\bigl(f(X_t); \zeta>t\bigr)\\
        &= \hEx\bigl(f(X_t)\, \hPx\bigl(\zeta>t\cond\hcM^0\bigr)\bigr)\\
        &= \hEx\bigl(f(X_t)\,e^{-A_t}\bigr)\\
        &= E_x\bigl(f(X_t)\,e^{-A_t}\bigr). \qedhere
\end{align*}
\end{proof}

\subsection{Preservation of the Markov Property} \label{aappZii}
Now we prove the
\begin{theorem} \label{a_kill_thm9}
$\hX$ is a Markov process with respect to $(\hcM_t,\,t\ge 0)$.
\end{theorem}

\begin{proof}
Let $f$ be a bounded, measurable function on $E$. We remark that~corollary~\ref{a_kill_cor8}
shows that $x\mapsto \hEx(f(\hX_t))$ is $\cE^*/\cB(\R)$--measurable, where
$\cE^*$ denotes the universal augmentation of $\cE$.

Let $x\in E$, $s$, $t\ge 0$. We have to prove that
\begin{equation}    \label{a_kill19}
    \hEx\bigl(f(\hX_{s+t})\cond \hcM_t\bigr)
        = \hat E_{\hX_t}\bigl(f(\hX_s)\bigr),\qquad \text{$\hPx$--a.s.}
\end{equation}
To this end, let $\hgL\in\hcM_t$, and choose $\gL_t\in\cM_t$ to be such that
$\hgL\cap\zt = (\LtR)\cap\zt$. Then we can compute in the following way:
\begin{align*}
    \hEx\bigl(f(\hX_{s+t}); \hgL\bigr)
        &= \hEx\bigl(f(X_{s+t}); \hgL\cap\zt[s+t]\bigr)\\
        &= \hEx\bigl(f(X_{s+t}); (\hgL\cap\zt)\cap\zt[s+t]\bigr)\\
        &= \hEx\bigl(f(X_{s+t}); (\LtR)\cap\zt[s+t]\bigr)\\
        &= \hEx\bigl(f(X_{s+t})\,1_{\LtR}\,\hPx(\zeta>s+t\cond\hcM^0)\bigr)\\
        &= \hEx\bigl(f(X_{s+t})\,1_{\LtR}\,e^{-A_{s+t}}\bigr)\\
        &= E_x\bigl(f(X_{s+t})\,e^{-A_{s+t}};\gL_t\bigr),
\end{align*}
where we made use of lemma~\ref{a_kill_lem7}. By the additivity of $A$ the last
expression is equal to
\begin{align*}
    E_x\bigl((f(X_s)\,e^{-A_s})\comp\theta_t\,e^{-A_t}; \gL_t\bigr)
        &= E_x\Bigl( E_x\bigl((f(X_s)\,e^{-A_s})\comp\theta_t\cond \cM_t\bigr)
                \,e^{-A_t}; \gL_t\Bigr)\\
        &= E_x\Bigl( E_{X_t}\bigl(f(X_s)\,e^{-A_s}\bigr)\,e^{-A_t}; \gL_t\Bigr)\\
        &= E_x\bigl(F(X_t)\,e^{-A_t};\gL_t\bigr),
\end{align*}
where we have used the Markov property of $X$ relative to $(\cM_t,\,t\ge 0)$, and where
we have set
\begin{equation}	\label{a_kill19a}
    F(x) = E_x\bigl(f(X_s)\,e^{-A_s}\bigr).
\end{equation}
Observe that by corollary~\ref{a_kill_cor8}
\begin{equation}    \label{a_kill20}
    F(x) = \hEx\bigl(f(\hX_s\bigr).
\end{equation}
Thus we have
\begin{align*}
    \hEx\bigl(f(\hX_{s+t}); \hgL\bigr)
        &= E_x\bigl(F(X_t)\,e^{-A_t};\gL_t\bigr)\\
        &= \hEx\bigl(F(X_t)\,1_{\LtR}\,e^{-A_t}\bigr)\\
        &= \hEx\Bigl(F(X_t)\,1_{\LtR}\,\hPx\bigl(\zeta>t\cond \hcM^0\bigr)\Bigr)\\
        &= \hEx\bigl(F(X_t); (\LtR)\cap\zt\bigr)\\
        &= \hEx\bigl(F(X_t); \hgL\cap\zt\bigr)\\
        &= \hEx\bigl(F(\hX_t);\hgL\bigr)\\
        &= \hEx\Bigl(\hat E_{\hX_t}\bigl(f(\hX_s)\bigr);\hgL\Bigr),
\end{align*}
and the proof of relation~\eqref{a_kill19} is complete.
\end{proof}

\subsection{Preservation of the Strong Markov Property} \label{aappZiii}
In this subsection we consider the strong Markov property. So we assume in
addition that $X$ is a strong Markov process as in section~I.8  of~\cite{BlGe68}.

The crucial point is the following analogue of Lemma~III.3.9, p.108,
in~\cite{BlGe68}.

\begin{lemma}   \label{a_kill_lem10}
Let $\hT$ be a stopping time relative to $(\hcM_t,\,t\ge 0)$.
\begin{enum_a}
    \item There is a unique stopping time $T$ relative to $(\cM_t,\,t\ge 0)$
          so that%
          \footnote{\footnotesize Recall that we denote the canonical extension
          of $T$ from $\gO$ to $\hgO$ again by $T$.}%
          \begin{equation}  \label{a_kill21}
                \hT \land \zeta = T \land \zeta
          \end{equation}
		 holds.
    \item Let $\hT$ and $T$ be as in~(a), and let $\hgL\in\hMT$. Then there is a set
          $\gL\in\cM_T$ so that
          \begin{equation}  \label{a_kill22}
                \hgL\cap\zt[\hT] = \LR\cap\zt[\hT].
          \end{equation}
\end{enum_a}
\end{lemma}

\begin{proof}
The proof of statement~(a) is very similar to the corresponding one in Lem\-ma~III.3.9
in~\cite{BlGe68}, and is therefore omitted.

We begin the proof of statement~(b) with the following remark which will prove to be
useful. Let $\tau\ge 0$ and $\go\in\gO$ be given. We claim there exists
$\hgo=(\go,s)\in\hgO$, $s\in\Rbp$, so that $\zeta(\hgo)>\tau$. In fact, we note
that by equation~\eqref{a_kill15} this claim is equivalent to saying that
there is an $\gep>0$ so that $A_{\tau+\gep}(\go)\le S(\hgo)$. Choose any $\gep>0$, and
observe that the continuity of the increasing mapping $u\mapsto A_u(\go)$ guarantees
the existence of $s\ge 0$ so that $s\ge A_{\tau+\gep}(\go)$. Set $\hgo=(\go,s)$. Then
the claim follows because $S(\hgo)=s$.

Now let $\hT$ be a stopping time with respect to $(\hcM_t,\,t\ge 0)$, and let $T$ be the
associated $(\cM_t,\,t\ge 0)$--stopping time as in part~(a). Choose
$\hgL\in\hcM_{\hT}$, and for $t\ge 0$ define the set
\begin{equation*}
    \hgL_t = \hgL\cap\{\hT\le t,\, \zeta>t\}.
\end{equation*}
Since $\hgL\cap\{\hT\le t\}\in\hcM_t$, we have that
\begin{equation}    \label{a_kill23}
    \hgL_t = (\LtR)\cap\zt,
\end{equation}
for some set $\gL_t\in\cM_t$. We show that $\gL_t\subset\{T\le t\}$: If
$\go\in\gL_t$ pick $s\in\R_+$ so that $\zeta(\go,s)>t$ (which is always possible due
the remark above). Therefore
\begin{equation*}
    (\go,s)\in\hgL_t\subset \{\hT\le t,\,\zeta>t\}.
\end{equation*}
Equation~\eqref{a_kill21} entails that
\begin{equation}    \label{a_kill24}
    \{\hT\le t,\,\zeta>t\} = \bigl(\{T\le t\}\times\Rbp\bigr)\cap\zt,
\end{equation}
and therefore we get $(\go,s)\in\{T\le t\}$.

Similarly one shows that for $t'$, $t\ge 0$ with $t'\le t$, one has
$\gL_{t'}\subset \gL_t$, where these two sets are associated with $\hgL_t$
and $\hgL_{t'}$ as in~\eqref{a_kill23}.

Now we claim that for $t'$, $t\ge 0$ with $t'\le t$
\begin{equation*}
    \gL_{t'}\cap\{T\le t'\}= \gL_t\cap\{T\le t'\}
\end{equation*}
holds true. Since we have that $\gL_{t'}\subset \gL_t$, we only
need to show the inclusion
\begin{equation*}
    \gL_{t'}\cap\{T\le t'\}\supset \gL_t\cap\{T\le t'\}.
\end{equation*}
To this end, assume that $\go\in\gL_t\cap\{T\le t'\}$. Choose $s\ge 0$ so that
$(\go,s)\in\zt\subset \zt[t']$. Then
\begin{equation*}
    (\go,s) \in (\gL_t\times\Rbp)\cap\{T\le t\}\cap\zt = \hgL_t.
\end{equation*}
Therefore $(\go,s)\in\hgL\cap\{T\le t',\,\zeta> t'\}$, i.e., $(\go,s)\in\hgL_{t'}$.
Using formula~\eqref{a_kill23} for $\hgL_{t'}$, we get that $\go\in\gL_{t'}\cap\{T\le
t'\}$, and our claim is proved.

We conclude that for every $t\ge 0$, $\gL_t\in\cM_T$. Indeed, let $t'\le t$, then
\begin{equation*}
    \gL_t\cap\{T\le t'\} = \gL_{t'}\cap\{T\le t'\} = \gL_{t'}\in\cM_{t'}.
\end{equation*}
If $t'\ge t$, then $\gL_t\subset \{T\le t\}\subset \{T\le t'\}$ and therefore
\begin{equation*}
    \gL_t\cap\{T\le t'\} = \gL_t\in\cM_t\subset \cM_{t'}.
\end{equation*}

To finish the proof of the lemma, set
\begin{equation*}
    \gL = \bigcup_{t\in\Q_+} \gL_t,
\end{equation*}
where $\Q_+ = \Q\cap\R_+$. Then we find that $\gL\in\cM_T$, and moreover by
what has just been proved
\begin{equation*}
    \gL\cap \{T\le t\} = \gL_t,\qquad t\in\Q_+.
\end{equation*}
Now we can calculate as follows
\begin{align*}
    \hgL\cap\zt[\hT]
        &= \hgL\cap \Bigl(\bigcup_{t\in\Q_+} \{\hT\le t,\,\zeta>t\}\Bigr)\\
        &= \bigcup_{t\in\Q_+} \hgL\cap\{\hT\le t,\,\zeta>t\}\\
        &= \bigcup_{t\in\Q_+} \hgL_t\\
        &= \bigcup_{t\in\Q_+} (\gL_t\times\Rbp)\cap\zt\\
        &= \bigcup_{t\in\Q_+} (\gL\times\Rbp)\cap\bigl(\{T\le t\}\times\Rbp\bigr)\cap\zt\\
        &= (\gL\times\Rbp)\cap \Bigl(\bigcup_{t\in\Q_+} \{\hT\le t,\,\zeta>t\}\Bigr)\\
        &= (\gL\times\Rbp)\cap\zt,
\end{align*}
and the proof of the lemma is complete.
\end{proof}

For the statement of the strong Markov property of $\hX$, we need the following
lemma whose proof is straightforward in view of lemma~\ref{a_kill_lem10} and
therefore left to the reader:

\begin{lemma}   \label{a_kill_lem11}
Let $\hT$ be stopping time relative to $(\hcM_t,\,t\ge 0)$. Then
$\hX_{\hT}\in\hcM_{\hT}/\cE$.
\end{lemma}

\begin{theorem} \label{a_kill_thm12}
Suppose that $X$ is a strong Markov process relative to $(\cM_t,\,t\ge 0)$.
Then $\hX$ is a strong Markov process relative to $(\hcM_t,\,t\ge 0)$.
\end{theorem}

\begin{proof}
Let $\hT$ be an $(\hcM_t,\,t\ge 0)$--stopping time, and let $T$ be the associated
stopping with respect to $(\cM_t,\,t\ge 0)$ as in lemma~\ref{a_kill_lem10}. Let $x\in E$,
$\hgL\in\hcM_{\hT}$ and consider a bounded measurable function $f$ on $E$. According
to lemma~\ref{a_kill_lem10}.b there exists a set $\gL\in\cM_T$ so that
relation~\eqref{a_kill22} holds true. With a similar calculation as in the proof of
theorem~\ref{a_kill_thm9} and lemma~\ref{a_kill_lem7} we get
\begin{equation*}
    \hEx\bigl(f(\hX_{s+\hT});\hgL\bigr)
			= E_x\bigl(f(X_{s+T})\,e^{-A_{s+T}}; \gL\bigr),\qquad s\ge 0.
\end{equation*}
With equation~\eqref{a_kill1} and the assumption that $X$ is a strong Markov process
we obtain
\begin{align*}
    \hEx\bigl(f(\hX_{s+\hT});\hgL\bigr)
        &= E_x\Bigl(\bigl(f(X_s)\,e^{-A_s}\bigr)\comp\theta_T\,e^{-A_T};\gL\Bigr)\\
        &= E_x\Bigl(E_{X_T}\bigl(f(X_s)\,e^{-A_s}\bigr)\,e^{-A_T};\gL\Bigr)\\
        &= \hEx\Bigl(E_{X_T}\bigl(f(X_s)\,e^{-A_s}\bigr)\,e^{-A_T};\gL\times\Rbp\Bigr).
\end{align*}
With equation~\eqref{a_kill17} and the notation~\eqref{a_kill19a} we find
\begin{align*}
    \hEx\bigl(f(\hX_{s+\hT});\hgL\bigr)
        &= \hEx\Bigl(F(X_T)\,\hPx\bigl(\zeta>T\cond\hcM^0\bigr);\gL\times\Rbp\Bigr)\\
        &= \hEx\Bigl(F(X_T);\LR\cap\zt[T]\Bigr),
\end{align*}
where we used the fact that the indicator of $\gL\times\Rbp$ and $F(X_T)$ are both
$\hcM^0$--measurable. Because of $\hT=T$, $\hX_{\hT}=X_T$ on $\zt[T]=\zt[\hT]$, and
due to relation~\eqref{a_kill22} we get
\begin{align*}
    \hEx\bigl(f(\hX_{s+\hT});\hgL\bigr)
        &= \hEx\bigl(F(\hX_{\hT}); \hgL\bigr)\\
        &= \hEx\Bigl(\hat E_{\hX_{\hT}}\bigl(f(\hX_s)\bigr); \hgL\Bigr).
\end{align*}
As a consequence, $\hPx$--a.s.\
\begin{equation}
    \hEx\bigl(f(\hX_{s+\hT})\cond\hcM_{\hT}\bigr)
        = \hat E_{\hX_{\hT}}\bigl(f(\hX_s)\bigr),
\end{equation}
which concludes the proof.
\end{proof}

\section{Some Results related to Brownian Local Time}  \label{app_LT}
We consider a standard one-dimensional Brownian motion $B=(B_t,\,t\ge 0)$ on a
family $(\gO,\cA,(P_x,\,x\in\R))$ of probability spaces, endowed with a right
continuous, complete filtration $\cF=(\cF_t,\,t\ge 0)$. For given $x\in\R$ we denote
by $L(x) = (L_t(x),\,t\ge 0)$ its local time at $x$. We assume, as we may, that both
have exclusively continuous paths. (For example, for the local time $L(x)$ this
follows from the Tanaka formula~\cite[Proposition~3.6.8]{KaSh91} for $L(x)$, and
Skorokhod's Lemma~\cite[Lemma~3.6.14]{KaSh91}.) In particular, for given $x\in\R$,
$L(x)$ is a progressively measurable stochastic process. The right continuous
pseudo-inverse of $L$ will be denoted by $K$:
\begin{equation}    \label{a_LT1}
    K_r = \inf\,\{t\ge 0,\, L_t>r\},\qquad r\in\R_+,
\end{equation}
with the convention $\inf\emptyset = +\infty$. Because $L$ is $\cF$--adapted and
$\cF$ is right continuous, for every $r\ge 0$, $K_r$ is an $\cF$--stopping time.

\begin{lemma}   \label{a_LT_lem1}
For $r\ge 0$,
\begin{equation}    \label{a_LT2}
    P_0\bigl(K_r\in dl\bigr) = \frac{r}{\sqrt{\mathstrut2\pi l^3}}\,e^{-r^2/2l}\,dl,\qquad l\ge 0,
\end{equation}
and
\begin{equation}    \label{a_LT3}
    E_0\bigl(e^{-\gl K_r}\bigr) = e^{-\sgl r},\qquad \gl\ge 0.
\end{equation}
\end{lemma}

\begin{proof}
By Tanaka's formula (e.g., \cite[Chapter~3.6]{KaSh91}), for every $\ga>0$ the
stochastic process
\begin{equation*}
    t \mapsto e^{-\ga\bigl(|B_t|-L_t\bigr) - \ga^2 t/2}
\end{equation*}
is a martingale under $P_0$. Since $L_t \ua +\infty$, $P_0$--a.s., we get
$P_0(K_r<+\infty)=1$. The usual argument of truncation of $K_r$, application of
Doob's optional sampling/stopping theorem (e.g., \cite[p.~19]{KaSh91}, or
\cite[p.~65, ff.]{ReYo91}), followed by the removal of the truncation yields with
$\ga = \sgl$,
\begin{equation*}
    E_0\bigl(e^{-\gl K_r}\bigr) = e^{-\sgl\,r},\qquad r\ge 0,\,\gl>0.
\end{equation*}
Formula~\eqref{a_LT2} follows from equation~\eqref{a_LT3} by inversion of the
Laplace transform with the help of the formula~\eqref{st14}, and the uniqueness theorem
for Laplace transforms.
\end{proof}

Let $a>0$, and consider the stopping time $H_a$, defined as the hitting time of $a$
by the reflecting Brownian motion $|B|$, i.e., $H_a$ is the hitting time of the set
$\{-a,a\}$ by $B$. We want to calculate the law of $L_{H_a}(0)$ under $P_0$ with a
method similar to the proof of Theorem~6.4.3 in~\cite{KaSh91} --- actually we shall
solve a variant of Problem~6.4.4 in~\cite{KaSh91}.

Denote by $C^2(\R;0,\pm a)$ the space of bounded, continuous functions $f$ on $\R$
which belong to $C^2(\R\setminus\{-a,0,a\})$, and which are such that at $-a$, $0$,
$a$,  $f'$ and $f''$ have (finite) left and right limits. For $f\in C^2(\R;0,\pm
a)$, $t\ge 0$, the generalized It\^o formula reads
(cf.~\cite[Equation~3.6.53]{KaSh91})
\begin{equation}    \label{a_LT5}
\begin{split}
    df(B_t)
        &= f'(B_t)\,dB_t + \frac{1}{2}\,f''(B_t)\,dt\\
            &\quad +\frac{1}{2}\,\Bigl(\gD f'(0)\,dL_t(0) + \gD f'(-a)\,dL_t(-a)
                + \gD f'(a)\,dL_t(a) \Bigr),
\end{split}
\end{equation}
where we denoted $\gD f'(b) = f'(b+)-f'(b-)$, $b\in\R$. For $\ga$, $\gb>0$,
$\rho\in\R$ set
\begin{equation} \label{a_LT6}
    Z_t = \exp\bigl(-\ga t - \gb L_t(0) + \rho (L_t(-a)+L_t(a))\bigr),
\end{equation}
and note that for all $t\ge 0$, $Z_t$ belongs to all $\cL^p(P_x)$, $p\ge 1$,
$x\in\R$. (For example, this follows easily from an application of H\"older's
inequality and of the strong Markov property of $B$, the translation invariance of
the transition density of $B$, and the explicit law of $L_t(0)$ under $P_0$, e.g.,
Theorem~3.6.17 in\cite{KaSh91}.)

Now let $v\in C^2(\R;0,\pm a)$ be an even function, and define the stochastic process
\begin{equation}    \label{a_LT7}
    X_t = v(B_t)\,Z_t,\qquad t\ge 0.
\end{equation}
We choose $v$ in such a way that $X$ becomes for $P_x$, $x\in\R$, a martingale,
namely, we require that $v$ solves the following boundary value problem:
\begin{subequations} \label{a_LT8}
\begin{align}
    \frac{1}{2}\,v'' &= \ga v,\qquad \text{on $\R\setminus\{-a,0,a\}$,}  \label{a_LT8a}\\
    v'(0+)           &= \gb v(0), \label{a_LT8b}\\
    v(a)             &= 1. \label{a_LT8c}
\end{align}
\end{subequations}
Furthermore, in the definition~\eqref{a_LT6} of $Z$ we choose
\begin{equation}    \label{a_LT9}
    \rho = -\frac{1}{2}\,\gD v'(a).
\end{equation}
Then it is straightforward to check with formula~\eqref{a_LT5} that $X_t$ is indeed
a martingale under $P_x$, $x\in\R$.

We remark that $L_t(\pm a)=0$ for all $t\le H_a$. Now we proceed as in the proof of
lemma~\ref{a_LT_lem1}: We truncate $H_a$ so that it becomes a finite stopping time,
apply Doob's optional sampling theorem, and remove the truncation by an application
of the dominated convergence theorem. The result is the formula
\begin{equation}    \label{a_LT10}
    v(x) = E_x\bigl(e^{-\ga H_a - \gb L_{H_a}(0)}\bigr), \qquad x\in\R,
\end{equation}
and in particular
\begin{equation}    \label{a_LT11}
    E_0\bigl(e^{-\gb L_{H_a}(0)}\bigr) = \lim_{\ga\downarrow 0} v(0).
\end{equation}

The unique bounded even solution of the boundary problem~\eqref{a_LT8} is given by
\begin{equation}    \label{a_LT12}
    v(x) = \begin{cases}
                \displaystyle
                \frac{\sqrt{2\ga}\,\cosh(\sqrt{2\ga}x) +\gb \sinh(\sqrt{2\ga}x)}
                    {\sqrt{2\ga}\,\cosh(\sqrt{2\ga}a) +\gb \sinh(\sqrt{2\ga}a)},
                            & 0\le x \le a,\\[3ex]
                \displaystyle
                e^{-\sqrt{2\ga}(x-a)}, & x>a,
           \end{cases}
\end{equation}
and $v(x) = v(-x)$, $x<0$. Thus we get
\begin{equation}    \label{a_LT13}
    E_0\bigl(e^{-\gb L_{H_a}(0)}\bigr) = \frac{1}{1+\gb a},
\end{equation}
and we have proved the

\begin{lemma}   \label{a_LT_lem2}
Under $P_0$, $L_{H_a}(0)$ is exponentially distributed with mean $a$.
\end{lemma}

\begin{remark}  \label{a_LT_rem3}
For the reader who compares this result with equation~(6.4.16) in~\cite{KaSh91}, we
want point out that here $H_a$ denotes the hitting time of $\{-a,a\}$ while there
$T_b$ is the hitting time of $b$, and furthermore our convention for the local time
differs by the one used in~\cite{KaSh91} by a factor $2$.
\end{remark}

For the next result we continue to denote by $H_a$ the hitting time of the set
$\{-a,a\}$ by the Brownian motion $B$, and use the set-up of
appendix~\ref{app_killing}. Fix $\gb>0$, and denote by $P_\gb$ the exponential law
on $(\R,\cB(\R))$ of rate $\gb$. Construct the family of product spaces $(\hgO,\hat
A, (\hat P_x,\,x\in\R))$ analogously as in appendix~\ref{app_killing} with $P_e$
replaced by $P_\gb$ in equation~\eqref{a_kill6}. For a PCHAF we choose here the local
time $L(0)$ at zero, and consider the stopping time $\zeta_\gb$ defined as
in~\eqref{a_kill10} with $A_t=L_t(0)$, $t\ge 0$.

\begin{lemma}   \label{a_LT_lem4}
For all $a$, $\gb>0$,
\begin{equation}    \label{a_LT14}
    \hat P_0\bigl(H_a < \zeta_\gb \bigr) = \frac{1}{1+\gb a}
\end{equation}
\end{lemma}

\begin{proof}
We use equation~\eqref{a_kill14}, and lemma~\ref{a_LT_lem2} to compute
\begin{align*}
    \hat P_0\bigl(H_a \le \zeta_\gb\bigr)
        &= \hat P_0\bigl(S \ge L_{H_a}(0)\bigr)\\
        &= \gb\int_0^\infty e^{-\gb s}\, P_0\bigl(L_{H_a}(0)\le s\bigr)\,ds\\
        &= \gb\int_0^\infty e^{-\gb s}\, \Bigl(\frac{1}{a}\,\int_0^s e^{-u/a}\,du\Bigr)\,ds\\
        &= \frac{1}{1+\gb a}.
\end{align*}
The proof is concluded by the remark that by definition of $H_a$ and $\zeta$
we have that $H_a\ne \zeta_\gb$ unless $H_a=\zeta_\gb=+\infty$, but the latter
event has $\hat P_0$--measure zero.
\end{proof}

\section{An Inverse Laplace Transform and Heat Kernels} \label{app_LTHK}
The following lemma gives a generalization of a well-known formula for Laplace
transforms (e.g., \cite[eq.~(5.1.2)]{ErMa54a}.

\begin{lemma}   \label{a_LTHK_lem1}
Suppose that $f$ is a bounded, measurable function on $\R_+$, and let $a>0$, $b\ge
0$. Define
\begin{equation}    \label{a_LTHK1}
    K_{a,b}f(t) = \sqrt{\frac{a}{2\pi}}\,\int_0^t \frac{s+b}{(t-s)^{3/2}}\,
                    \exp\Bigl(-a\,\frac{(s+b)^2}{2(t-s)}\Bigr)\,f(s)\,ds.
\end{equation}
Then
\begin{equation}    \label{a_LTHK2}
    \cL\bigl(K_{a,b}f\bigr)(\gl)
        = e^{-\sqrt{2a\gl}b}\,\cL f\bigl(\gl+\sqrt{2a\gl}\bigr),\qquad \gl>0,
\end{equation}
where $\cL$ denotes the Laplace transform.
\end{lemma}

\begin{proof}
Let $\gl>0$. The theorem of Fubini-Tonelli yields
\begin{equation*}
    \cL\bigl(K_{a,b}f\bigr)(\gl)
        = \int_0^\infty f(s)\int_s^\infty e^{-\gl t}\,
                \frac{\sqrt{a}\,(s+b)}{\sqrt{2\pi(t-s)^3}}\,
                    \exp\Bigl(-a\,\frac{(s+b)^2}{2(t-s)}\Bigr)\,dt\,ds.
\end{equation*}
The inner integral is equal to
\begin{equation}    \label{a_LTHK3}
    \int_0^\infty e^{-\gl(t+s)}\,\frac{\sqrt{a}\,(s+b)}{\sqrt{2\pi t^3}}\,
            e^{-a(s+b)^2/2t}\,dt
        = e^{-\gl s}\,e^{-\sqrt{2a\gl}(s+b)},
\end{equation}
where we used the Laplace transform~\eqref{st14}. Inserting the right hand side of
equation~\eqref{a_LTHK3} above, we find formula~\eqref{a_LTHK2}.
\end{proof}

Recall the definition~\eqref{gen6} for $\rho(\gl)$, $\gl>0$, and~\eqref{gen14}
for $g_{\gb,\gg}(t,x)$, $\gb$, $\gg$, $t>0$, $x\ge 0$. Then lemma~\ref{a_LTHK_lem1}
immediately gives with the choice $a=\gg^{-2}$, $b=\gg x$, and $f(s)=\exp(-\gb s/\gg)$
the following result.

\begin{corollary}   \label{a_LTHK_cor2}
For $\gb$, $\gg$, $x\ge 0$, the inverse Laplace transform of
\begin{equation}    \label{a_LTHK6}
    \gl \mapsto \rho(\gl)\,e^{-\sqrt{2\gl}x},\qquad \gl>0,
\end{equation}
is given by $t\mapsto g_{\gb,\gg}(t,x)$, $t > 0$.
\end{corollary}

In order to relate $g_{\gb,\gg}$ to other heat kernels in this article we consider
the limits when the parameters $\gb$, $\gg$ tend to zero. First fix $\gb>0$. Then a
straightforward (though somewhat tedious) calculation shows that for all $t>0$, $x\ge
0$,
\begin{equation}    \label{a_LTHK7}
    \lim_{\gg\downarrow 0} g_{\gb,\gg}(t,x) = g_{\gb,0}(t,x),
\end{equation}
where $g_{\gb,0}(t,x)$ is given by~\eqref{el24}. On the other hand, for fixed
$\gg>0$, one has
\begin{equation}    \label{a_LTHK10}
    \lim_{\gb\downarrow 0} g_{\gb,\gg}(t,x) = g_{0,\gg}(t,x),
\end{equation}
with $g_{0,\gg}(t,x)$ given by~\eqref{st22}. Maybe the easiest way to see this is to
compute the limit of the expression defining $g_{\gb,\gg}$, i.e., formula~\eqref{gen14},
and then to apply lemma~\ref{a_LTHK_lem1}
to show that the Laplace transform of the resulting integral is given by
\begin{equation*}
   \gl \mapsto  \frac{1}{\sgl+\gg\gl}\,e^{-\sqrt{2\gl}x},\qquad \gl>0.
\end{equation*}
Now one can invert this Laplace transform, e.g., with formula~(5.6.16)
in~\cite{ErMa54a}, and this gives $g_{0,\gg}$ as defined above. Finally, with
formulae~\eqref{el24}, \eqref{st22} it is easy to see that
\begin{equation}    \label{a_LTHK12}
     \lim_{\gb\downarrow 0} g_{\gb,0}(t,x)
        = \lim_{\gg\downarrow 0} g_{0,\gg}(t,x) = g(t,x).
\end{equation}

\section{First Passage Time Formula}    \label{app_FPTF}
For the convenience of the reader we provide in this section a proof of a specific
form of the well-known first passage time formula (e.g., \cite[p.~94]{ItMc74},
\cite[p.~158]{Kn81}) which we use in this article.

We assume that $X$ is a normal, strong Markov process relative to a right continuous
filtration $\cF$, with state space $(E,\cE)$ being a locally compact, separable
metric space, such that $X$ has c\`adl\`ag paths and $\cF$ is complete with respect
to the family $P=(P_x,\,x\in\ E)$ of probability measures on the underlying
measurable space $(\gO,\cA)$.

Let $x_0$ be a given point in $E$, and denote by $H_0$ its hitting time by $X$. We
suppose that $P$--a.s.\ $H_0$ is finite. In particular, this implies that $X$ cannot
be killed before reaching $x_0$. Moreover, our assumptions above entail that $H_0$
is an $\cF$--stopping time (cf., e.g., Theorem~III.2.17 in~\cite{ReYo91}). The
resolvent of $X$ will be denoted by $R=(R_\gl,\,\gl>0)$, while
$R^0=(R^0_\gl,\,\gl>0)$ stands for the resolvent of the process $X^0$
which is obtained from $X$ by killing it when reaching the point $x_0$. Recall that
by convention every real valued function $f$ on $E$ is extended to $E\cup\{\gD\}$ by
setting $f(\gD)=0$. Then we have the

\begin{lemma}   \label{a_FPTF_lem1}
For every bounded measurable functions $f$ on $E$, and all $x\in E$, $\gl>0$,
\begin{equation}    \label{a_FPTF1}
    R_\gl f(x) = R^0_\gl f(x) + E_x\bigl(e^{-\gl H_0}\bigr)\,R_\gl f(x_0).
\end{equation}
\end{lemma}

\begin{proof}
Let $f$, $\gl$ and $x$ be as in the hypothesis. Then
\begin{align*}
    R_\gl f(x)
        &= E_x\Bigl(\int_0^\infty e^{-\gl t} f(X_t)\,dt\Bigr)\\
        &= E_x\Bigl(\int_0^{H_0} e^{-\gl t} f(X_t)\,dt\Bigr)
                + E_x\Bigl(\int_{H_0}^\infty e^{-\gl t} f(X_t)\,dt\Bigr)\\
        &= E_x\Bigl(\int_0^\infty e^{-\gl t} f(X^0_t)\,dt\Bigr)
                + E_x\Bigl(e^{-\gl H_0}\int_0^\infty e^{-\gl t} f(X_{t+H_0})\,dt\Bigr)\\
        &= R^0_\gl f(x)
                + E_x\Bigl(e^{-\gl H_0}\int_0^\infty e^{-\gl t}
                        E_x\bigl(f(X_{t+H_0}\bigr)\cond \cF_{H_0})\,dt\Bigr)\\
        &= R^0_\gl f(x)
                + E_x\Bigl(e^{-\gl H_0}\int_0^\infty e^{-\gl t}
                        E_{x_0}\bigl(f(X_t)\bigr)\,dt\Bigr)\\
        &= R^0_\gl f(x)
                + E_x\bigl(e^{-\gl H_0}\bigr)\,R_\gl f(x_0),
\end{align*}
where we used the strong Markov property of $X$, and its right continuity to
conclude that $X_{H_0}=x_0$.
\end{proof}

\section{Proofs of Lemmas~\ref{int_lemiv}--\ref{int_lemviii}}    \label{app_int}
In this appendix we prove lemmas~\ref{int_lemiv}--\ref{int_lemviii}.

\begin{proof}[Proof of lemma~\ref{int_lemiv}]
Fix $k\in\N$, $x\in [0,1]$. By standard arguments it is enough to show
equality~\eqref{int_eqvii} for $V$ of the form
\begin{equation*}
    V = \prod_{i=1}^n f_i(X_{t_i}),
\end{equation*}
with $n\in\N$, $0<t_1<t_2<\dotsb<t_n$, and bounded measurable functions $f_1$,
$f_2$,\dots, $f_n$ on $[0,1]$. Thus we have to show that
\begin{equation}    \label{aint_eqi}
    E_x\Bigl(\prod_{i=1}^n f_i(X_{t_i+T_k});\,\gL\Bigr)
        =    E_x\Bigl(E_{X_{T_k}}\Bigl(\prod_{i=1}^n f_i(X_{t_i})\Bigr);\,\gL\Bigr)
\end{equation}
for every $\gL\in\cF^X_{T_k}\cap\{X_{T_k}\ne\gD\}$. It is sufficient to show~\eqref{aint_eqi}
for every $\gL$ in a $\cap$--stable generator of $\cF^X_{T_k}\cap\{X_{T_k}\ne\gD\}$, and since
by Galmarino's theorem, lemma~\ref{int_lemii}, $\cF^X_{T_k}=\gs(X_{\,.\,\land T_k})$, we can
choose $\gL$ to be of the form
\begin{equation*}
    \gL = \bigcap_{j=1}^m \{X_{s_j\land T_k}\in C_j\}\cap\{X_{T_k}\ne\gD\},
\end{equation*}
with $m\in\N$, $s_1$, $s_2$, \dots, $s_m\in\R_+$, $C_1$, $C_2$, \dots,
$C_m\in\cB([0,1]\cup\{\gD\})$. Remark that by construction $\{X_{T_k}\ne \gD\}=
\{X_{T_k}\in\{0,1\}\}$. Hence we may decompose $\gL = \gL_0\uplus\gL_1$ in
$\cF^X_{T_k}\cap\{X_{T_k}\ne \gD\}$ with $\gL_c = \gL\cap\{X_{T_k}=c\}$, $c=0$, $1$.
Consider the right hand side of equation~\eqref{aint_eqi} with $\gL$ replaced by $\gL_0$.
By lemma~\ref{int_lemi} we find
\begin{equation}	\label{pfeq1}
	 E_x\Bigl(\prod_{i=1}^n f_i(X_{t_i+T_k}); \gL_0\Bigr)
		= E_x\Bigl(\prod_{i=1}^n f_i(Y_{t_i+S_k}); \gL^Y_0\Bigr)
\end{equation}
with
\begin{equation*}
	\gL^Y_0 = \bigcap_{j=1}^m \{Y_{s_j\land S_k}\in C_j\}\cap \{Y_{S_k}=0\}.
\end{equation*}
In order to compute the expectation value on the right hand side
of equation~\eqref{pfeq1}, we factorize the probability space
$(\Xi,\cC,Q_x)$ as follows:
\begin{align*}
	\Xil &= \BigCart_{k=0}^{n-1} \Xi^k,  &
				\Xiu &= \BigCart_{k=n}^{\infty} \Xi^k,\\
	\cCl &= \bigotimes_{k=0}^{n-1} \cC^k, &
				\cCu &= \bigotimes_{k=n}^{\infty} \cC^k,\\
	\Qxl &= Q_x^0\otimes\Bigl(\bigotimes_{k=1}^{n-1} Q^k\Bigr), &
				\Qu &= \bigotimes_{k=n}^{\infty} Q^k.
\end{align*}
We show that $\gL^Y_0$ belongs to the $\gs$--algebra
$\cCl\times\Xiu$, i.e., that the indicator of this set only is a function
of the variables in $\Xil$. Firstly, from the construction of $Y$ and the
sequence $(S_k,\,k\in\N)$ it follows that
\begin{equation*}
	\{Y_{S_k}=0\}
		= \begin{cases}
			\{\rho^0_1<\rho^0_0\}\cap\{S_k<+\infty\}, &\text{if $k$ is even},\\
			\{\rho^0_0<\rho^0_1\}\cap\{S_k<+\infty\}, &\text{if $k$ is odd}.
		  \end{cases}	
\end{equation*}
Since in both cases the sets on the right hand side belong to $\cCl\times\Xiu$,
we find that $\{Y_{S_k}=0\}\in\cCl\times\Xiu$. Next, for $j\in\{1,\dotsc,m\}$
we get
\begin{equation*}
\begin{split}
	\{Y_{s_j\land S_k}&\in C_j\}\cap\{Y_{S_k}=0\}\\
		&= \{Y_{s_j}\in C_j,\,s_j< S_k,\,Y_{S_k}=0\}\uplus
			\{Y_{S_k}\in C_j,\,s_j\ge S_k,\,Y_{S_k}=0\}.
\end{split}
\end{equation*}
The first term belongs to $\cCl\times\Xiu$, and for the second
we note that it is empty unless $0\in C_j$, and in this case it is equal
to $\{s_j\ge S_k\}\cap\{Y_{S_k}=0\}$. Therefore this event belongs to
$\cCl\times\Xiu$, too. Hence we can write
\begin{align}	\label{pfeqA}
	E_x\Bigl(\prod_{i=1}^n f_i&(Y_{t_i+S_k}); \gL^Y_0\Bigr)\nonumber\\
		&= \int_{\Xil} 1_{\gL^Y_0}\, \Bigl(\int_{\Xiu}
				\prod_{i=1}^n f_i(Y_{t_i+S_k})\,d\Qu\Bigr)\,d\Qxl\nonumber\\
		&= \sum_{k\le l_1\le \dotsb \le l_n}
				\int_{\Xil} 1_{\gL^Y_0}\, R_{l_1,\dotsc,l_n}\,d\Qxl,
\end{align}
with
\begin{equation}	\label{pfeq2}
	R_{l_1,\dotsc,l_n}
		= \int_{\Xiu} \prod_{i=1}^n f_i(Y_{t_i+S_k})
			\,1_{\{S_{l_i}\le t_i+S_k < S_{l_i+1}\}}\,d\Qu,
\end{equation}
and we made use of the fact that $Q_x$--a.s.\ the sequence $(S_n,\,n\in\N)$
is strictly increasing to infinity.

We claim that on $\{Y_{S_k}=0\}$
\begin{equation}	\label{pfeq3}
	R_{l_1,\dotsc,l_n}
		= E_0\Bigl(\prod_{i=1}^n f_i(Y_{t_i})\,
			1_{\{S_{l_i-k+1}\le t_i < S_{l_i-k+2}\}}\Bigr)
\end{equation}
holds true. Consider the integrand on the right hand side of equation~\eqref{pfeq2},
choose $i\in\{1,\dotsc,n\}$, and assume first that $l_i=k$. Then (on $\{Y_{S_k}=0\}$)
the corresponding term in the product is equal to
\begin{equation*}
	f_i(Y_{t_i+S_k})\,1_{\{0\le t_i <\gs^k\}}
		= f_i(A^k_{t_i})\,1_{\{0\le t_i <\gs^k\}}.
\end{equation*}
Note that the term on the right hand side is constant on $\{Y_{S_k}=0\}\subset\Xil$,
and that under $\Qu$ it has the same law as
\begin{equation*}
	f_i(A^1_{t_i})\,1_{\{0\le t_i <\gs^1\}}
\end{equation*}
under $Q_0$. Next consider the case $l_i\ge k+1$. By construction
the relations
\begin{equation}	\label{pfeq4}
\begin{split}
	S_{l_i}-S_k
		&= \gs^k + \tau^{k+1} + \dotsb + \eta^{l_i-1},\\
	S_{l_i+1}-S_k
		&= \gs^k + \tau^{k+1} + \dotsb + \eta^{l_i},
\end{split}
\end{equation}
hold true on $\{Y_{S_k}=0\}$, where for $r\in\N$, $r\ge k$,
\begin{equation}	\label{pfeq5}
	\eta^r = \begin{cases}
				\gs^r,	&\text{if $r-k$ is even},\\
				\tau^r,	&\text{if $r-k$ is odd}.
			\end{cases}
\end{equation}
Moreover, on $\{Y_{S_k}=0\}\cap\{S_{l_i}\le t_i+S_k<S_{l_i+1}\}$,
\begin{equation}	\label{pfeq6}
	Y_{t_i+S_k} = \begin{cases}
					A^{l_i}_{t_i-(S_{l_i}-S_k)},  &\text{if $l_i-k$ is even},\\[2ex]
					B^{l_i}_{t_i-(S_{l_i}-S_k)},  &\text{if $l_i-k$ is odd}.
			     \end{cases}		
\end{equation}
Remark that the right hand sides of the relations~\eqref{pfeq4}, \eqref{pfeq6} are
constant on $\Xil$. On the other hand, under $Q_0$, i.e., with start of $Y$ at $0$, we
get $\rho^0_0=0$,
and hence
\begin{equation}	\label{pfeq7}
\begin{split}
	S_{l_i-k+1}
		&= \gs^1 + \tau^2 + \dotsb + \xi^{l_i-k},\\
	S_{l_i-k+2}
		&= \gs^1 + \tau^2 + \dotsb + \xi^{l_i-k+1},
\end{split}
\end{equation}
with ($r\in\N$, $r\ge 1$)
\begin{equation}	\label{pfeq8}
	\xi^r = \begin{cases}
				\gs^r,	&\text{if $r$ is odd},\\
				\tau^r,  &\text{if $r$ is even.}
		    \end{cases}		
\end{equation}
Furthermore, if $t_i\in[S_{l_i-k+1},S_{l_i-k+2})$ then
\begin{equation}	\label{pfeq9}
	Y_{t_i} = \begin{cases}
				A^{l_i-k+1}_{t_i-S_{l_i-k+1}},	&\text{if $l_i-k$ is even},\\[2ex]
				B^{l_i-k+1}_{t_i-S_{l_i-k+1}},	&\text{if $l_i-k$ is odd}.
			 \end{cases}
\end{equation}
A comparison of the right hand sides of equations~\eqref{pfeq4}--\eqref{pfeq6}
with~\eqref{pfeq7}--\eqref{pfeq9} shows that on $\{Y_{S_k}=0\}$ formula~\eqref{pfeq3}
holds true as claimed.

We insert  formula~\eqref{pfeq3} into equation~\eqref{pfeqA} and get
\begin{align*}
	E_x\Bigl(\prod_{i=1}^n f_i&(Y_{t_i+S_k}); \gL^Y_0\Bigr)\\
		&= \sum_{k\le l_1\le \dotsb \le l_n} Q_x\bigl(\gL^Y_0\bigr)\,
			E_0\Bigl(\prod_{i=1}^n f_i(Y_{t_i})\,
				1_{\{S_{l_i-k+1}\le t_i < S_{l_i-k+2}\}}\Bigr)\\
		&= Q_x\bigl(\gL^Y_0\bigr)\,
			E_0\Bigl(\prod_{i=1}^n f_i(Y_{t_i})\,1_{\{S_1\le t_i \}}\Bigr)\\
		&= Q_x\bigl(\gL^Y_0\bigr)\,
			E_0\Bigl(\prod_{i=1}^n f_i(Y_{t_i})\Bigr)\\
		&= E_x\Bigl(E_0\Bigl(\prod_{i=1}^n f_i(Y_{t_i})\Bigr); \gL^Y_0\Bigr)\\
		&= E_x\Bigl(E_{Y_{S_k}}\Bigl(\prod_{i=1}^n f_i(Y_{t_i})\Bigr); \gL^Y_0\Bigr)
\end{align*}
where for the third equality we used that under $Q_0$, $S_1=0$. With lemma~\ref{int_lemi}
we find
\begin{equation*}
    E_x\Bigl(\prod_{i=1}^n f_i(X_{t_i+T_k}); \gL_0\Bigr)
		= E_x\Bigl(E_{X_{T_k}}\Bigl(\prod_{i=1}^n f_i(X_{t_i})\Bigr); \gL_0\Bigr).
\end{equation*}
We add this equation to the corresponding one where $\gL_0$ is interchanged
with $\gL_1$, and obtain
\begin{equation*}
    E_x\Bigl(\prod_{i=1}^n f_i(X_{t_i+T_k}); \gL\Bigr)
		= E_x\Bigl(E_{X_{T_k}}\Bigl(\prod_{i=1}^n f_i(X_{t_i})\Bigr); \gL\Bigr).
\end{equation*}
Thus lemma~\ref{int_lemiv} is proved.
\end{proof}

For later use we bring the strong Markov property of the last lemma into an
efficient form. To this end, let us denote the family of cylinder functions of $X$
by $\cG$, i.e., $G\in\cG$ is a random variable of the form
\begin{equation}    \label{aint_eqii}
    G = g\bigl(X_{s_1},\dotsc, X_{s_m}\bigr),
\end{equation}
where $m\in\N$, $g$ is a bounded, measurable function on $[0,1]^m$ (extended to
$([0,1]\cup\{\gD\})^m$ in the usual way), $s_1$, \dots, $s_m\in\R_+$, with $s_1\le
s_2\le \dotsb \le s_m$. For $G\in\cG$ of this form set
\begin{equation}    \label{aint_eqiii}
    G_r = g\bigl(X_{s_1-r},\dotsc, X_{s_m-r}\bigr),\qquad r\le s_1.
\end{equation}

\begin{lemma}   \label{aint_lemi}
Assume that $G\in\cG$ is of the form~\eqref{aint_eqii}. Suppose furthermore that
$k\in\N$, and that $\vp$ is a bounded, measurable function on $\overline \R_+$.
Then for all $x\in[0,1]$,
\begin{equation}    \label{aint_eqiv}
\begin{split}
    E_x\bigl(\vp(T_{k+1})\,&G;\,T_k\le s_1 \cond \cF^X_{T_k}\bigr)\\[.5ex]
        &= E_{X_{T_k}}\bigl(\vp(r+H_c)\,G_r\bigr)\,1_{[0,s_1]}(r)\eval_{r=T_k,\,c=(X_{T_k})^*}
\end{split}
\end{equation}
holds true $P_x$--a.s.\ on $\{X_{T_k}\ne \gD\}$.
\end{lemma}

\begin{proof}
Let $\gL\in\cF^X_{T_k}\cap\{X_{T_k}\ne\gD\}$, and decompose it in $\cF^X_{T_k}$ as
$\gL=\gL_0\uplus \gL_1$with $\gL_c = \gL\cap\{X_{T_k}=c\}$, $c=0$, $1$.

We choose first $\vp(x) = \exp(i\ga x)$, $\ga\in\R$, $x\in\R$, $\vp(+\infty)=1$. Set
$u_j = s_j-s_1\ge 0$, $j=1$, \dots, $m$, and define for $s\ge 0$
\begin{equation*}
    K(s) = g\bigl(X_s, X_{u_2+s},\dotsc, X_{u_m+s}\bigr),
\end{equation*}
so that $G=K(s_1)$. We compute the following Laplace transform at $\gl>0$:
\begin{align*}
    \int_0^\infty e^{-\gl s} E_x\bigl(\vp(T_{k+1}) &K(s);\,\gL_0
            \cap \{T_k\le s\}\bigr)\,ds\\
        &= E_x\Bigl(e^{i\ga T_{k+1}} \int_{T_k}^\infty e^{-\gl s} K(s)\,ds;\,\gL_0\Bigr)\\
        &= E_x\Bigl(e^{i\ga T_{k+1}-\gl T_k}\Bigl(\int_0^\infty e^{-\gl s}
                K(s)\,ds\Bigr)\comp\theta_{T_k};\,\gL_0\Bigr).
\end{align*}
By construction $T_{k+1}=T_k + H_1\comp\theta_{T_k}$ when $X_{T_k}=0$. Therefore the last
expression equals
\begin{align*}
    E_x\Bigl(&e^{(i\ga-\gl)T_k}\,E_x\Bigl(\Bigl(e^{i\ga H_1}\int_0^\infty e^{-\gl s}
            K(s)\,ds\Bigr)\comp\theta_{T_k}\cond \cF_{T_k}\Bigr);\,\gL_0\Bigr)\\
        &= E_x\bigl(e^{(i\ga-\gl)T_k};\,\gL_0\bigr)\,
                    E_0\Bigl(e^{i\ga H_1} \int_0^\infty e^{-\gl s} K(s)\,ds\Bigr)\\
        &= \int_0^\infty e^{-\gl s}\,e^{i\ga s}\,P_x\bigl(\{T_k\in ds\}\cap\gL_0\bigr)
                \int_0^\infty e^{-\gl s}\,E_0\bigl(e^{i\ga H_1}\,K(s)\bigr)\,ds,
\end{align*}
where we used the strong Markov property of $X$ relative to $T_k$,
cf.~lemma~\ref{int_lemiv}. Inverting the Laplace transform we get
\begin{align*}
    E_x\bigl(e^{i\ga T_{k+1}}\,&G;\,\gL_0\cap\{T_k\le s_1\}\bigr)\\
        &= \int_0^{s_1} P_x\bigl(\gL_0\cap\{T_k\in dr\}\bigr)\,
                E_0\bigl(e^{i\ga(r+H_1)}\,G_r\bigr)\\
        &= \int_0^\infty P_x\bigl(\gL_0\cap\{T_k\in dr,\,T_k\le s_1\}\bigr)\,
                E_0\bigl(e^{i\ga(r+H_1)}\,G_r\bigr)\\
        &= E_x\Bigl(E_0\bigl(e^{i\ga(r+H_1)}\,G_r\bigr)\,
                1_{[0,s_1]}(r)\eval_{r=T_k};\,\gL_0\Bigr)\\
        &= E_x\Bigl(E_{X_{T_k}}\bigl(e^{i\ga(r+H_c)}\,G_r\bigr)\,
                1_{[0,s_1]}(r)\eval_{r=T_k,\, c = (X_{T_k})^*};\,\gL_0\Bigr).
\end{align*}
Repeating the same calculation with $\gL_0$ replaced by $\gL_1$, gives an
analogous expression, and adding both we obtain the following formula
\begin{equation*}
\begin{split}
    E_x\bigl(e^{i\ga T_{k+1}}\,&G;\,\gL\cap\{T_k\le s_1\}\bigr)\\
        &=E_x\Bigl(E_{X_{T_k}}\bigl(e^{i\ga(r+H_c)}\,G_r\bigr)\,
                1_{[0,s_1]}(r)\eval_{r=T_k,\, c = (X_{T_k})^*};\,\gL\Bigr)
\end{split}
\end{equation*}
for arbitrary $\gL\in\cF^X_{T_k}\cap\{X_{T_k}\ne\gD\}$. Finally, a routine approximation
argument combined with Fourier analysis yields equation~\eqref{aint_eqiv}.
\end{proof}

Now we are ready for the

\begin{proof}[Proof of lemma~\ref{int_lemvi}]
We work on the event $\{0\le t < T_1,\,X_t\ne \gD\}$, $t\in\R_+$. For $s\ge 0$, write
\begin{equation}    \label{aint_eqAi}
\begin{split}
    E_x\bigl(f(X_{s+t})\cond \cF^X_t\bigr)
        = E_x\bigl(&f(X_{s+t});\,s+t < T_1\cond \cF^X_t\bigr)\\
            &+ E_x\bigl(f(X_{s+t});\, s+t \ge T_1\cond \cF^X_t\bigr).
\end{split}
\end{equation}
The first term on the right hand side only involves the process $X$ with killing at
the boundary of the interval $[0,1]$, and this process is Markovian relative to
$\cF^X$ (cf.\ the discussion at the end of subsection~\ref{ssect_constr}). Therefore
this term is equal to
\begin{equation}    \label{aint_eqAii}
    E_{X_t}\bigl(f(X_s);\,s<T_1\bigr).
\end{equation}
For the second term on the right hand side of equation~\eqref{aint_eqAi} we remark
that
\begin{equation*}
    \cF^X_t\cap\{t<T_1\}\subset \cF^X_{T_1}\cap\{t<T_1\},
\end{equation*}
and therefore it is equal to
\begin{align*}
    E_x\Bigl(E_x\bigl(&f(X_{s+t})\cond \cF^X_{T_1}\bigr);\,
                        T_1\le s+t,\,X_{T_1}=0\cond \cF^X_t\Bigr)\\
        &\hspace{4em} + E_x\Bigl(E_x\bigl(f(X_{s+t})\cond \cF^X_{T_1}\bigr);\,
                        T_1\le s+t,\,X_{T_1}=1\cond \cF^X_t\Bigr)\\
        &= E_x\bigl(\vp_0(T_1);\, T_1\le s+t,\,X_{T_1}=0\cond \cF^X_t\bigr)\\
        &\hspace{4em} + E_x\bigl(\vp_1(T_1);\, T_1\le s+t,\,X_{T_1}=1\cond \cF^X_t\bigr),
\end{align*}
where
\begin{equation*}
    \vp_c(r) = E_c\bigl(f(X_{s+t-r})\bigr),\qquad c=0,1,\,0\le r\le s+t,
\end{equation*}
and we used the strong Markov property of $X$ relative to $T_1$,
lemma~\ref{aint_lemi}. The last two expectation values only involve the process
$\hat X$, that is, $X$ with absorption at the endpoints of the interval $[0,1]$, and
its absorption time~$T_1$. This is a Markov process with respect to $\cF^X$ (cf.\
our discussion at the end of subsection~\ref{ssect_constr}). Moreover, $t<T_1$
entails the relations $\vp_c(T_1)=\vp_c(t+T_1)\comp \theta_t$, $\{T_1\le s+t\}=
\{T_1\comp\theta_t \le s\}$, and $\{X_{T_1}=c\}=\{X_{T_1}\comp\theta_t=c\}$. Thus
we get
\begin{align*}
    E_x\bigl(f(X_{s+t});\,&T_1\le s+t\cond \cF^X_t\bigr)\\
        &= E_{X_t}\bigl(\vp_0(t+T_1);\,T_1\le s,\,X_{T_1}=0\bigr)\\
        &\hspace{4em} + E_{X_t}\bigl(\vp_1(t+T_1);\,T_1\le s,\,X_{T_1}=1\bigr)\\
        &= E_{X_t}\Bigl(E_0\bigl(f(X_{s-r})\bigr)\eval_{r=T_1};\,
                    T_1\le s,\,X_{T_1}=0\Bigr)\\
        &\hspace{4em} + E_{X_t}\Bigl(E_1\bigl(f(X_{s-r})\bigr)\eval_{r=T_1};\,
                    T_1\le s,\,X_{T_1}=1\Bigr)\\
        &= E_{X_t}\Bigl(E_{X_{T_1}}\bigl(f(X_{s-r})\bigr)\eval_{r=T_1};\,T_1\le s\Bigr).
\end{align*}
We apply lemma~\ref{aint_lemi} to the last expression, and the result is
\begin{equation*}
    E_{X_t}\Bigl(E_{X_t}\bigl(f(X_s)\cond \cF^X_{T_1}\bigr);\,T_1\le s\Bigr)
        = E_{X_t}\bigl(f(X_s);\,T_1\le s\bigr).
\end{equation*}
Combining this with~\eqref{aint_eqAii} for the first term on the right hand side
of equation~\eqref{aint_eqAi}, we get on $\{t<T_1,\,X_t\ne \gD\}$
\begin{equation*}
    E_x\bigl(f(X_{s+t})\cond \cF^X_t\bigr) = E_{X_t}\bigl(f(X_s)\bigr),\qquad \text{$P_x$--a.s.},
\end{equation*}
and lemma~\ref{int_lemvi} is proved.
\end{proof}

Recall that we write $\tilde X^c$ for the process $X$ with killing at $c$, $c\in\{0,1\}$,
and $\tilde X$ for the process $X$ with killing at either of the points $0$, $1$.
Furthermore, $\tilde A$ and $\tilde B$ denote the processes $A$, $B$, with killing at $1$,
$0$ respectively. We shall make use of the semigroups $U^1$, $U^0$,
generated by $\tilde A$, $\tilde B$ respectively:
\begin{align*}
    U^1_t f(x) &= E_x\bigl(f(\tilde A_t)\bigr) = E_x\bigl(f(A_t);\,t<H^A_1\bigr),\\
    U^0_t f(x) &= E_x\bigl(f(\tilde B_t)\bigr) = E_x\bigl(f(B_t);\,t<H^B_0\bigr).
\end{align*}

With the help of the strong Markov property of the processes $A$, $B$, it is not hard
to prove the following result:

\begin{lemma}   \label{aint_lemii}
If $s\le t$, then a.s.\ for all $y\in [0,1]$,
\begin{equation}    \label{aint_eqv}
\begin{split}
     E_y\bigl(f(A_t);\,t<H^A_1\cond \cF^A_s\bigr)
        &= U^1_{t-s}f(A_s)\,1_{\{s<H^A_1\}},\\
     E_y\bigl(f(B_t);\,t<H^B_0\cond \cF^B_s\bigr)
        &= U^0_{t-s}f(B_s)\,1_{\{s<H^B_0\}} .
\end{split}
\end{equation}
\end{lemma}

\begin{proof}[Proof of lemma~\ref{int_lemv}]
We only show the equivalence of $\tilde X^1$ and $\tilde A$, the equivalence of
$X^0$ and $\tilde B$ follows from an analogous argument. Also in this proof, for the
sake of simplicity of notation we denote by $\cF$ the filtration generated by $A$, and its
hitting times of $c=0$, $1$, by $H_c$. Fix $t>0$, let $G$ be a bounded
$\cF_s$--measurable random variable, $0\le s < t$, and suppose that $F$ is a bounded
measurable function on $[0,1]$. Since $\{H_0<H_1\}\in\cF_{H_0}$ (see, e.g.,
\cite[Lemma~1.2.16]{KaSh91}), we get with the strong Markov property of $A$,
\begin{align*}
    E_x\bigl(G\, &F(A_t);\,s<H_0\le t <H_1\bigr)\\
        &= E_x\bigl(G\, 1_{\{s<H_0<H_1\}}\,F(A_t)\,1_{\{t<H_1\}}\bigr)\\
        &= E_x\bigl(G\, 1_{\{s<H_0<H_1\}}\,E_x\bigl(F(A_t)\,1_{\{t<H_1\}}
                                            \cond \cF_{H_0}\bigr)\bigr)\\
        &= E_x\bigl(G\, 1_{\{s<H_0<H_1\}}\,E_0\bigl(F(A_{t-r});
                                \,t-r<H_1\bigr)\,1_{[0,t]}(r)\eval_{r=H_0}\bigr),
\end{align*}
where we can, for example, use the Laplace transform as in the proof of
lemma~\ref{aint_lemi}, together with fact that on $\{H_0<H_1\}$, we have
$H_1=H_0+H_1\comp\theta_{H_0}$, to prove the last equation.
Therefore we obtain the following formula
\begin{equation}    \label{aint_eqvi}
\begin{split}
    E_x\bigl(G\, F(A_t);\,&s<H_0\le t <H_1\bigr)\\
        &= \int_s^t E_0\bigl(F(A_{t-r});\,t-r<H_1\bigr)\,m_{x,s,G}(dr),
\end{split}
\end{equation}
with the finite, signed measure $m_{x,s,G}$ on $(\overline\R_+,\cB(\overline\R_+))$
being defined by
\begin{equation}    \label{aint_eqvii}
    m_{x,s,G}(C) = E_x\Bigl(G;\,H_0\in C,\,s<H_0<H_1\bigr),
                                        \qquad C\in \cB(\overline\R_+).
\end{equation}

Now suppose that $k\in\N$, $s<t=t_1<t_2<\dots<t_k$, and that $h_1$, $h_2$, \dots,
$h_k$ are bounded measurable functions on $[0,1]$. Put
\begin{equation*}
    F(x) = h_1(x)\,\bigl(U^b_{t_2-t_1}h_2\,U^b_{t_3-t_2} h_3
                        \dotsb U^b_{t_k-t_{k-1}}h_k\bigr)(x),\qquad x\in [0,1].
\end{equation*}
Then equations~\eqref{aint_eqvi} and~\eqref{aint_eqvii} entail
\begin{equation}    \label{aint_eqviii}
\begin{split}
   E_x\bigl(G\, &h_1(A_{t_1}) \dotsb h_k(A_{t_k});
		\,s<H_0\le t_1<t_2<\dotsb<t_k <H_1\bigr)\\&=
		\int_s^t E_0\bigl(h_1(A_{t_1-r})\dotsb
         h_k(A_{{t_k}-r};\,t_k-r<H_1\bigr)\,m_{x,s,G}(dr).
\end{split}
\end{equation}

Next we make the following choice
\begin{equation*}
    G = g_1(A_{s_1})\, g_2(A_{s_2})\dotsb g_l(A_{s_l}),
\end{equation*}
with $l\in\N$, $0<s_1<\dotsb<s_l=s$, $g_1$, \dots, $g_l$ being bounded, measurable
functions on $[0,1]$. Consider the measure $m_{x,s,G}$ in this case, and let
$C\in\cB(\overline\R_+)$. Then
\begin{equation*}
    m_{x,s,G}(C)
        = E_x\bigl(g_1(A_{s_1})\dotsb g_l(A_{s_l});\,H_0\in C,\,s_l<H_0<H_1\bigr).
\end{equation*}
Recall that we assume, as we may, that $A$ is constructed pathwise from the standard
Brownian motion $O$, which exclusively has continuous paths. In particular,
the paths of $A$ and $O$ coincide up to the hitting time $H^O_0$ (inclusive),
so that $H^A_0=H^O_0$, where $H^A_c$ and $H^O_c$ denote the hitting times of $c=0$,
$1$, by $A$ and $O$ respectively. We claim that $\{H_0^A<H^A_1\} = \{H^O_0<H^O_1\}$.
Indeed, this equality is equivalent to $\{H^A_1\le H^A_0\}=\{H^O_1\le H^O_0\}$,
but since $A$ and $O$ coincide on the random interval $[0,H^A_0]=
[0,H^O_0]$ this is obvious. Consequently the formula
\begin{equation*}
    m_{x,s,G}(C)
        = E_x\bigl(g_1(O_{s_1})\dotsb g_l(O_{s_l});\,H^O_0\in C,\,s<H^O_0<H^O_1\bigr)
\end{equation*}
holds true. Moreover, starting in $0$, $A$ is equivalent to the process $A^1$. Hence
formula~\eqref{aint_eqviii} reads for the above choice of $G$
\begin{equation*}
\begin{split}
    E_x\bigl(G &h_1(A_{t_1})\dotsb h_k(A_{t_k});\,s<H_0\le t_1<t_2<\dotsb<t_k<H_1\bigr)\\
        &= \int_s^t E_x\bigl(g_1(O^0_{s_1})\dotsb g_l(O^0_{s_l});\,H^{O^0}_0\in
                dr,\,s<H^{O^0}_0<H^{O^0}_1\bigr)\\
        &\hspace{8em} \times E_0\bigl(h_1(A^1_{t_1-r})\dotsb
                h_k(A^1_{t_k-r});\,t_k-r<\gs^1\bigr),
\end{split}
\end{equation*}
and we recall that $\gs^1$ denotes the hitting time of $1$ by $A^1$.
From the construction of $Y$ we obtain
\begin{align}   \label{aint_eqix}
    E_x\bigl(&g_1(A_{s_1})\dotsb g_l(A_{s_l})\nonumber\\
        &\hspace{4em}\times h_1(A_{t_1})\dotsb h_k(A_{t_k});
            \,s_l<H^A_0\le t_1<t_2<\dotsb<t_k<H^A_1 \bigr)\nonumber\\
    &= E_x\bigl(g_1(Y_{s_1})\dotsb g_l(Y_{s_l})\nonumber\\
        &\hspace{4em}\times h_1(Y_{t_1})\dotsb h_k(Y_{t_k});
            \,s_l<H^Y_0\le t_1<t_2<\dotsb<t_k<H^Y_1 \bigr)\nonumber\\
    &= E_x\bigl(g_1(X_{s_1})\dotsb g_l(X_{s_l})\nonumber\\
        &\hspace{4em}\times h_1(X_{t_1})\dotsb h_k(X_{t_k});
            \,s_l<H^X_0\le t_1<t_2<\dotsb< t_k<H^X_1 \bigr)
\end{align}
and in the last step we used lemma~\ref{int_lemi}.

Finally, let $f_1$, \dots, $f_n$, $n\in\N$, be bounded, measurable functions on
$[0,1]$, and assume that $0< v_1< v_2<\dotsb<v_n$. Then we write
\begin{equation*}
\begin{split}
    E_x\bigl(f_1&(A_{v_1})\dotsb f_n(A_{v_n});\,v_n<H^A_0\bigr)\\[.5ex]
        &= E_x\bigl(f_1(A_{v_1})\dotsb f_n(A_{v_n});\,v_n<H^A_0\land H^A_1\bigr)\\
        &\hspace{3em} +\sum_{j=1}^n E_x\bigl(f_1(A_{v_1})\dotsb f_n(A_{v_n});\,
                        v_{j-1}<H^A_0\le v_j,\,v_n<H^A_1\bigr).
\end{split}
\end{equation*}
Since up to the hitting time of $0$, $A$ and $X$ are both equivalent to a standard
Brownian motion, the first term on the right hand side is equal to
\begin{equation*}
    E_x\bigl(f_1(X_{v_1})\dotsb f_n(X_{v_k});\,v_n<H^X_0\land H^X_1\bigr).
\end{equation*}
For each term under the sum we use formula~\eqref{aint_eqix} with appropriate choices of
$l$, $k$, $s_i$, $g_i$, $i=1$, \dots, $l$, $t_j$, $h_j$, $j=1$, \dots, $k$.
Therefore we finally obtain
\begin{equation*}
    E_x\bigl(f_1(A_{v_1})\dotsb f_n(A_{v_n});\,v_n<H^A_1\bigr)
        = E_x\bigl(f_1(X_{v_1})\dotsb f_n(X_{v_n});\,v_n<H^X_1\bigr),
\end{equation*}
and the proof of the first statement of lemma~\ref{int_lemv} is finished. The second
statement follows immediately, because $A$ is a Markov process, and hence so is
$\tilde A$ (e.g., \cite[Chapter~III.3]{BlGe68}, \cite[Chapter~X]{Dy65a}), and the
simple Markov property actually is a property of the finite dimensional
distributions of a process. The last statement of lemma~\ref{int_lemv} is then a
consequence of the fact that $A$ and $B$ are equivalent to a standard Brownian
motion up to their hitting time of $0$, $1$ respectively.
\end{proof}

For $t\in\R_+$, and $k\in\N$, we consider the family $\cG_{k,t}$ of events $\gL$ of
the form
\begin{equation}    \label{aint_eqx}
    \gL = \gL_1 \cap \gL_2 \cap \{T_k\le s_1\},
\end{equation}
with $\gL_1\in\cF^X_{T_k}$,
\begin{equation*}
    \gL_2 = \bigcap_{j=1}^n \{X_{s_j}\in C_j\},
\end{equation*}
and where $n\in\N$, $0<s_1<$ \dots $<s_n\le t$, $C_1$, $C_2$, \dots
$C_n\in\cB([0,1]\cup\{\gD\})$. Then it is a routine exercise to show the following
result:

\begin{lemma}   \label{aint_lemiii}
$\cG_{k,t}$ is a $\cap$--stable generator of the $\gs$--algebra $\cF^X_t\cap\{T_k\le
t\}$.
\end{lemma}

\begin{proof}[Proof of lemma~\ref{int_lemvii}]
Take $\gL$ of the form~\eqref{aint_eqx}. Observe that $T_k\le s_1$ entails that
$T_k$ is finite, and thus $X_{T_k}\in\{0,1\}$, i.e., $X_{T_k}\ne \gD$. We have
to compute
\begin{equation*}
    E_x\bigl(f(X_{s+t});\,\gL\cap\{s+t<T_{k+1}\}\bigr).
\end{equation*}
(Note that $\gL\cap\{s+t<T_{k+1}\}$ includes the event $\{T_k\le t <T_{k+1}\}$.)
We want to apply lemma~\ref{aint_lemi}, and choose $G$ in equation~\eqref{aint_eqii}
as follows: $m=n+1$, $s_m=s_{n+1}=s+t$, and
\begin{equation*}
    g(X_{s_1},\dotsc, X_{s_m}) = \prod_{j=1}^n 1_{C_j}(X_{s_j})\,f(X_{s+t}).
\end{equation*}
Moreover, we choose the $\vp$ in lemma~\ref{aint_lemi} as the indicator of
$(s+t,+\infty]$. Then we can compute with lemma~\ref{aint_lemi} as follows:
\begin{align*}
    E_x\bigl(f(X_{s+t});\,&\gL\cap\{s+t<T_{k+1}\}\bigr)\\
        &= E_x\bigl(\vp(T_{k+1})\,G;\,\{T_k\le s_1\}\cap\gL_1\bigr)\\
        &= E_x\Bigl(E_x\bigl(\vp(T_{k+1})\,G;\,T_k\le s_1\cond
                \cF^X_{T_k}\bigr);\,\gL_1\Bigr)\\
        &= E_x\Bigl(E_{X_{T_k}}\bigl(\vp(r+H_c)\,G_r\bigr)\,
                1_{[0,s_1]}(r)\eval_{r=T_k,\,c=(X_{t_k})^*};\,\gL_1\Bigr).
\end{align*}
As before, we decompose $\gL_1=\gL_{1,0}\uplus\gL_{1,1}$ in $\cF^X_{T_k}$ with
$\gL_{1,c}=\gL_1\cap\{X_{T_k}=c\}$, $c=0$, $1$. The term involving $\gL_{1,0}$
is equal to
\begin{equation}    \label{aint_eqxi}
    E_x\Bigl(E_0\bigl(\vp(r+H_1)\,G_r\bigr)\,
                1_{[0,s_1]}(r)\eval_{r=T_k};\,\gL_{1,0}\Bigr).
\end{equation}
We consider the inner expectation for $r\in[0,s_1]$:
\begin{equation*}
    E_0\bigl(\vp(r+H_1)\,G_r\bigr)
        = E_0\Bigl(f(X_{s+t-r})\,\prod_{j=1}^n 1_{C_j}(X_{s_j-r});\,s+t-r<H_1\Bigr),
\end{equation*}
and remark that this expectation only concerns the process $X$ killed when
reaching~$1$. Moreover, the event $\cap_j \{X_{s_j-r}\in C_j\}$ belongs to
$\cF^X_{t-r}$. We use the Markov property (cf.\ lemma~\ref{int_lemv}) of the
killed process relative to this $\gs$--algebra, as well as
\begin{equation*}
    \{s+t-r<H_1\} = \{s<H_1\comp \theta_{t-r},\,t-r<H_1\}
\end{equation*}
to compute
\begin{align*}
    E_0\bigl(\vp(r&+H_1)\,G_r\bigr)\\
        &= E_0\Bigl(E_0\big(f(X_s)\comp\theta_{t-r};\,s<H_1\comp\theta_{t-r}\cond
                \cF^X_{t-r}\bigr);\\
        &\hspace{8em} \{t-r<H_1\}\cap \bigl(\cap_{j=1}^n \{X_{s_j-r}\in C_j\}\bigr)\Bigr)\\
        &= E_0\Bigl( E_{X_{t-r}}\bigl(f(X_s);\,s<H_1\bigr);\\
        &\hspace{8em} \{t-r<H_1\}\cap \bigl(\cap_{j=1}^n \{X_{s_j-r}\in C_j\}\bigr)\Bigr)\\
        &= E_0\bigl(J_r;\,t-r<H_1\bigr),
\end{align*}
with
\begin{equation*}
    J_r = E_{X_{t-r}}\bigl(f(X_s);\,s<H_1\bigr)\,\prod_{j=1}^n 1_{C_j}(X_{s_j-r}).
\end{equation*}
Therefore
\begin{align*}
    E_0\bigl(\vp(r+H_1)\,G_r\bigr)\,1_{[0,s_1]}(r)\eval_{r=T_k}
        &= E_0\bigl(J_r;\, t<r+H_1\bigr)\,1_{[0,s_1]}(r)\eval_{r=T_k}\\[1ex]
        &= E_x\bigl(J_0;\,T_k\le s_1 ,\,t<T_{k+1}\cond \cF^X_{T_k}\bigr),
\end{align*}
with another application of lemma~\ref{aint_lemi}. Set $J=J_0$ and insert this
into~\eqref{aint_eqxi}, then
\begin{align*}
    E_x\bigl(f(X_{s+t});\,&\gL\cap\{X_{T_k}=0\}\cap\{s+t<T_{k+1}\}\bigr)\\
        &= E_x\Bigl(E_x\bigl(J;\,T_k\le s_1,\,t<T_{k+1}\cond
                \cF^X_{T_k}\bigr);\,\gL_{1,0}\Bigr)\\
        &= E_x\bigl(J;\,\{T_k\le s_1,\,t<T_{k+1}\}\cap\gL_{1,0}\bigr)\\
        &= E_x\Bigl(E_{X_t}\bigl(f(X_s);\,s<H_1\bigr);\,\{t<T_{k+1}\}\cap\gL\cap\{X_{T_k=0}\}\Bigr).
\end{align*}
We add this expression to the corresponding one, where $\{X_{T_k}=0\}$ is replaced
by $\{X_{T_k}=1\}$, and find
\begin{align*}
    E_x\bigl(f(X_{s+t});\,&\gL\cap\{s+t<T_{k+1}\}\bigr)\\
        &= E_x\Bigl(E_{X_t}\bigl(f(X_s);\,s<H_c\bigr)
                \eval_{c=(X_{T_k})^*};\,\{t<T_{k+1}\}\cap\gL\Bigr).
\end{align*}
Therefore we have shown ($P_x$--a.s.)
\begin{align*}
    E_x\bigl(f(&X_{s+t});\,T_k\le t,\,s+t<T_{k+1}\cond\cF^X_t \bigr)\\
        &= E_{X_t}\bigl(f(X_s);\,s<H_c\bigr)\eval_{c=(X_{T_k})^*}\,1_{\{T_k\le t<T_{k+1}\}}\\
        &= E_{X_t}\bigl(f(X_s);\,H_{c^*}\le s<H_c\bigr)\eval_{c=(X_{T_k})^*}\,
                                1_{\{T_k\le t< T_{k+1}\}}\\
        &\hspace{4em} + E_{X_t}\bigl(f(X_s);\,s<H_c\land H_{c^*}\bigr)\eval_{c=(X_{T_k})^*}\,
                                1_{\{T_k\le t< T_{k+1}\}}\\
        &= E_{X_t}\bigl(f(X_s);\,T_1\le s<T_2,\,X_{T_1}=c\bigr)
                                \eval_{c=X_{T_k}}\,1_{\{T_k\le t< T_{k+1}\}}\\
        &\hspace{4em} + E_{X_t}\bigl(f(X_s);\,s<T_1\bigr)\,
                                1_{\{T_k\le t<T_{k+1}\}},
\end{align*}
which is the statement of lemma~\ref{int_lemvii}.
\end{proof}

\begin{proof}[Proof of lemma~\ref{int_lemviii}]
We work on the event $\{X_{T_k}=0,\,T_k\le t<T_{k+1},\,X_t\ne \gD\}$, $k\in\N_0$.
Then
\begin{align}	\label{aint_eqxii}
    E_x\bigl(f(X_{s+t});\,T_k&\le t<T_{k+1}\le s+t\cond\cF^X_t\bigr)\nonumber\\
        &= E_x\Bigl(E_x\bigl(f(X_{s+t})\cond\cF^X_{T_{k+1}}\bigr);\,
                                    T_k\le t<T_{k+1}\le s+t\cond\cF^X_t\Bigr)\nonumber\\
        &= E_x\bigl(\vp(T_{k+1});\,T_k\le t\cond \cF^X_t\bigr),
\end{align}
with
\begin{equation*}
    \vp(r) = E_1\bigl(f(X_{s+t-r})\bigr)\,1_{(t,s+t]}(r),\qquad r\in\R_+,
\end{equation*}
and we have applied lemma~\ref{aint_lemi}. Let $\gL\in\cG_{k,t}$ be of the
form~\eqref{aint_eqx}, and as before we write $\gL_{1,0} = \gL_1\cap\{X_{T_k}=0\}$.
We set $G=1_{\gL_2}$, then with another application of lemma~\ref{aint_lemi}
we get
\begin{align}   \label{aint_eqxiii}
    E_x\bigl(\vp(T_{k+1});\,&\gL\cap\{X_{T_k}=0\}\bigr)\nonumber\\
        &= E_x\bigl(\vp(T_{k+1});\,\gL_{1,0}\cap\gL_2\cap\{T_k\le s_1\}\bigr)\nonumber\\
        &= E_x\Bigl(E_x\bigl(\vp(T_{k+1})\,G\cond \cF^X_{T_k}\bigr);\,
                \gL_{1,0}\cap\{T_k\le s_1\}\Bigr)\nonumber\\
        &= E_x\Bigl(E_0\bigl(\vp(u+H_1)\,G_u\bigr)\,1_{[0,s_1]}(u)\eval_{u=T_k};\,
                \gL_{1,0}\Bigr).
\end{align}
Consider the inner expectation value
\begin{equation*}
    E_0\bigl(\vp(u+H_1)\,G_u\bigr)\,1_{[0,s_1]}(u)
        = E_0\Bigl(\vp(u+H_1)\,\prod_{j=1}^n
                1_{C_j}(X_{s_j-u})\Bigr)\,1_{[0,s_1]}(u).
\end{equation*}
The definition of $\vp$ implies that under this expectation we have $t-u<H_1$, and
therefore it only concerns the process $X$ killed at $1$, together with its lifetime
$H_1$ (which we recall is a stopping time relative to the natural filtration of
$X$). Therefore by lemma~\ref{int_lemv} we can use the Markov property of the killed
process to compute this expectation. To this end note that $s_1-u\le$ \dots $\le
s_n-u\le t-u$, so that $G_u$ is $\cF^X_{t-u}$ measurable. Thus
\begin{align} \label{aint_eqxiiia}
    E_0\bigl(&\vp(u+H_1)\,G_u\bigr)\,1_{[0,s_1]}(u)\nonumber\\
        &= E_0\Bigl(E_0\bigl(\vp(u+H_1);\,t-u<H_1\le s+t-u\cond \cF^X_{t-u}\bigr)\,
                G_u\Bigr)\,1_{[0,s_1]}(u)\nonumber\\
        &= E_0\Bigl(E_0\bigl(\vp(t+H_1)\comp\theta_{t-u};\,0<H_1\comp\theta_{t-u}\le s
                \cond \cF^X_{t-u}\bigr)\nonumber\\
        &\hspace{16em}\times G_u;\,t-u<H_1\Bigr)\,1_{[0,s_1]}(u)\nonumber\\
        &= E_0\Bigl(E_{X_{t-u}}\bigl(\vp(t+H_1);\,0<H_1\le s\bigr)\,
                G_u;\,t-u<H_1\Bigr)\,1_{[0,s_1]}(u),
\end{align}
and we made use of the relation
\begin{equation*}
    \{t-u<H_1\le s+t-u\}
        = \{H_1 = t-u+H_1\comp\theta_{t-u},\,0<H_1\comp\theta_{t-u}\le s,\,
                   t-u<H_1\}.
\end{equation*}
We evaluate equality~\eqref{aint_eqxiiia} at $u=T_k$, and apply lemma~\ref{aint_lemi}.
This gives
\begin{equation*}
\begin{split}
    E_0\bigl(&\vp(u+H_1)\,G_u\bigr)\,1_{[0,s_1]}(u)\eval_{u=T_k}\\[1ex]
        &= E_0\Bigl(E_{X_{t-u}}\bigl(\vp(t+H_1);\,0<H_1\le s\bigr)\,
                        G_u;\,t-u<H_1\Bigr)\,1_{[0,s_1]}(u)\eval_{u=T_k}\\[1ex]
        &= E_x\Bigl(E_{X_t}\bigl(\vp(t+H_1);\,0<H_1\le s\bigr)\,
                        G;\,T_k\le s_1,\,t<T_{k+1}\cond\cF^X_{T_k}\Bigr).
\end{split}
\end{equation*}
The last expression is inserted into equation~\eqref{aint_eqxiii}. This results
in
\begin{align*}
    E_x\bigl(&\vp(T_{k+1});\,\gL\cap\{X_{T_k}=0\}\bigr)\\
        &= E_x\Bigl(E_x\Bigl(E_{X_t}\bigl(\vp(t+H_1);\,0<H_1\le s\bigr)\\
        &\hspace{12em}\times G;\,T_k\le s_1,\,t<T_{k+1}\cond \cF^X_{T_k}\Bigr);\gL_{1,0}\Bigr)\\
        &= E_x\Bigl(E_{X_t}\bigl(\vp(t+H_1);\,0<H_1\le s\bigr);\,
                                \{X_{T_k}=0,\,T_k\le t<T_{k+1}\}\cap\gL\Bigr).
\end{align*}
With equation~\eqref{aint_eqxii}, this implies that on
$\{X_{T_k}=0,\,T_k\le t<T_{k+1},\,X_t\ne\gD\}$ we have the relation
\begin{equation}    \label{aint_eqxiv}
\begin{split}
    E_x\bigl(f(X_{s+t});\,T_k\le t&<T_{k+1}\le s+t\cond \cF^X_t\bigr)\\
        &= E_{X_t}\bigl(\vp(t+H_1);\,0<H_1\le s\bigr),
\end{split}
\end{equation}
because $\gL_{1,0}\in\cF^X_{T_k}$. Using lemma~\ref{aint_lemi} we compute the right hand
side of this formula:
\begin{align*}
    E_{X_t}\bigl(\vp(t&+H_1);\,0<H_1\le s\bigr)\\[1ex]
        &= E_{X_t}\bigl(\vp(t+H_1);\,0\le H_0< H_1\le s\bigr)\\
        &\hspace{6em} + E_{X_t}\bigl(\vp(t+H_1);\,0 < H_1\le s\land H_0\bigr)\\[1ex]
        &= E_{X_t}\bigl(\vp(t+T_2);\,T_2\le s,\,X_{T_1}=0\bigr)\\
        &\hspace{6em} + E_{X_t}\bigl(\vp(t+T_1);\,T_1\le s,\,X_{T_1}=1\bigr)\\[1ex]
	    &= E_{X_t}\bigl(E_1\bigl(f(X_{s-r})\bigr)\eval_{r=T_2};\,T_2\le s,\,X_{T_1}=0\bigr)\\
	    &\hspace{6em} + E_{X_t}\bigl(E_1\bigl(f(X_{s-r})\bigr)\eval_{r=T_1};\,
             T_1\le s,\,X_{T_1}=1\bigr)\\[1ex]
        &= E_{X_t}\bigl(E_{X_t}\bigl(f(X_s)\cond \cF^X_{T_2}\bigr);\,
             T_2\le s,\,X_{T_1}=0\bigr)\\
        &\hspace{6em} + E_{X_t}\bigl(E_{X_t}\bigl(f(X_s)\cond \cF^X_{T_1}\bigr);\,
            T_1\le s,\,X_{T_1}=1\bigr)\\[1ex]
        &= E_{X_t}\bigl(f(X_s);\,T_2\le s,\,X_{T_1}=0\bigr)\\
        &\hspace{6em} + E_{X_t}\bigl(f(X_s);\,T_1\le s,\,X_{T_1}=1\bigr).
\end{align*}
We have proved that on $\{X_{T_k}=0,\,T_k\le t <T_{k+1},\,X_t\ne\gD\}$ the following
formula holds true $P_x$--a.s.:
\begin{equation*}
\begin{split}
    E_x\bigl(f(&X_{s+t});\, T_{k+1}\le s+t\cond \cF^X_t\bigr)\\
        &= E_{X_t}\bigl(f(X_s);\,T_2\le s,\,X_{T_1}=0\bigr)
            + E_{X_t}\bigl(f(X_s);\,T_1\le s,\,X_{T_1}=1\bigr).
\end{split}
\end{equation*}
The case $\{X_{T_k}=1\}$ is dealt with in the same way, and lemma~\ref{int_lemviii} is proved.
\end{proof}

\end{appendix}

\providecommand{\bysame}{\leavevmode\hbox to3em{\hrulefill}\thinspace}
\providecommand{\MR}{\relax\ifhmode\unskip\space\fi MR }
\providecommand{\MRhref}[2]{%
  \href{http://www.ams.org/mathscinet-getitem?mr=#1}{#2}
}
\providecommand{\href}[2]{#2}


\begin{thebibliography}{10}

\bibitem{Ba91}
H.~Bauer, \emph{{W}ahrscheinlichkeitstheorie}, fourth ed., de~Gruyter, Berlin,
  New York, 2002.

\bibitem{BaCh84}
J.~R. Baxter and R.~V. Chacon, \emph{{T}he equivalence of diffusions on
  networks to {B}rownian motion}, Contemporary Math. \textbf{26} (1984),
  33--48.

\bibitem{BlGe68}
R.~M. Blumenthal and R.~K. Getoor, \emph{{M}arkov {P}rocesses and {P}otential
  {T}heory}, Academic Press, New York and London, 1968.

\bibitem{BrEx09}
B.M. Brown, P.~Exner, P.~Kuchment, and T.~Sunada (eds.), \emph{{A}nalysis on
  graphs and {I}ts {A}pplications}, Proc. Symp. Pure Math., vol.~77, 2009.

\bibitem{DeJa93}
D.~S. Dean and K.~M. Jansons, \emph{{B}rownian excursions on combs}, J. Stat.
  Phys. \textbf{70} (1993), 1313--1332.

\bibitem{Dy65a}
{E.B.} Dynkin, \emph{{M}arkov {P}rocesses}, vol.~1, Springer-Verlag, Berlin,
  Heidelberg, New York, 1965.

\bibitem{DyJu69}
{E.B.} Dynkin and {A.A.} Juschkewitsch, \emph{{S}{\"a}tze und {A}ufgaben
  {\"u}ber {M}arkoffsche {P}rozesse}, Springer-Verlag, Berlin, Heidelberg, New
  York, 1969.

\bibitem{ErMa54a}
A.~Erd{\'e}lyi, W.~Magnus, F.~Oberhettinger, and F.~G. Tricomi, \emph{{T}ables
  of {I}ntegral {T}ransforms}, vol.~I, McGraw-Hill, New York, Toronto, London,
  1954.

\bibitem{Fe52}
W.~Feller, \emph{{T}he parabolic differential equations and the associated
  semi-groups of transformations}, Ann. Math. \textbf{3} (1952), 468--519.

\bibitem{Fe54}
\bysame, \emph{{D}iffusion processes in one dimension}, Trans. American Math.
  Soc. \textbf{77} (1954), 1--31.

\bibitem{Fe57a}
\bysame, \emph{{G}eneralized second order differential operators and their
  lateral conditions}, Illinois J. Math. \textbf{1} (1957), 459--504.

\bibitem{Fe58}
\bysame, \emph{{N}otes to my paper: ``{O}n boundaries and lateral conditions
  for the {K}olmogorov differential equartions''}, Ann. Math. \textbf{68}
  (1958), 735--736.

\bibitem{FrWe93}
A.~I. Freidlin and A.~D. Wentzell, \emph{{D}iffusion processes on graphs and
  the averaging principle}, Ann. Probab. \textbf{21} (1993), 2215--2245.

\bibitem{FrSh00}
M.~Freidlin and S.-J. Sheu, \emph{{D}iffusion processes on graphs: stochastic
  differential equations, large deviation principle}, Probab. Theory Rel.
  Fields \textbf{116} (2000), 181--220.

\bibitem{Fr94}
A.~Friedman and C.~Huang, \emph{{D}iffusion in network}, J. Math. Anal. Apl.
  \textbf{183} (1994), 352--384.

\bibitem{Gr99}
N.~Grunewald, \emph{{M}artingales on graphs}, Ph.D. thesis, Mathematics
  Institute, University of Warwick, 1999.

\bibitem{It71a}
K.~It\^o, \emph{{P}oisson point processes attached to {M}arkov processes},
  Proc. Sixth Berkeley Symp. Math. Statist. Probab., vol.~3, University of
  California Press, 1972.

\bibitem{ItMc63}
K.~It\^o and H.~P. McKean~Jr., \emph{{B}rownian motions on a half line},
  Illinois J. Math. \textbf{7} (1963), 181--231.

\bibitem{ItMc74}
\bysame, \emph{{D}iffusion {P}rocesses and their {S}ample {P}aths}, 2nd ed.,
  Springer--Verlag, Berlin, Heidelberg, New York, 1974.

\bibitem{KaSh91}
I.~Karatzas and S.~E. Shreve, \emph{{B}rownian {M}otion and {S}tochastic
  {C}alculus}, 2nd ed., Springer-Verlag, Berlin, Heidelberg, New York, 1991.

\bibitem{Kn81}
{F.B.} Knight, \emph{{E}ssentials of {B}rownian {M}otion and {D}iffusion},
  Mathematical Surveys and Monographs, vol.~18, American Mathematical Society,
  Providence, Rhode Island, 1981.

\bibitem{KoPo07d}
V.~Kostrykin, J.~Potthoff, and R.~Schrader, \emph{{H}eat kernels on metric
  graphs and a trace formula}, {A}dventures in {M}athematical {P}hysics
  (F.~Germinet and P.~Hislop, eds.), Contemporary Mathematics, no. 447,
  American Math. Soc., Providence, Rhode Island, 2007.

\bibitem{KoPo09c}
\bysame, \emph{{C}ontraction {S}emigroups on {M}etric {G}raphs}, Analysis on
  Graphs and Its Applications (B.M. Brown, P.~Exner, P.~Kuchment, and
  T.~Sunada, eds.), Proc. Symp. Pure Math., vol.~77, 2009, pp.~423--458.

\bibitem{BMMG1}
\bysame, \emph{{B}rownian motions on metric graphs {I} - {D}efinition, {F}eller
  {P}roperty, and {G}enerators}, arXiv: 1012.0733, December 2010.

\bibitem{BMMG2}
\bysame, \emph{{B}rownian motions on metric graphs {II} - {C}onstruction of
  {B}rownian motions on single vertex graphs}, arXiv: 1012.0737, December 2010.

\bibitem{BMMG3}
\bysame, \emph{{B}rownian motions on metric graphs {III} - {C}onstruction of
  {B}rownian motions on general metric graphs}, arXiv: 1012.0739, December
  2010.

\bibitem{KoSc99}
V.~Kostrykin and R.~Schrader, \emph{{K}irchhoff's rule for quantum wires}, J.
  Phys. A: Math. Gen. \textbf{32} (1999), 595--630.

\bibitem{KoSc00}
\bysame, \emph{{K}irchhoff's rule for quantum wires {II}: {T}he inverse problem
  with possible applications to quantum computers}, Fortschr. Phys. \textbf{48}
  (2000), 703--716.

\bibitem{KoSc06a}
\bysame, \emph{{T}he inverse scattering problem for metric graphs and the
  traveling salesman problem}, preprint 0603010, arXiv:math-ph, 2006.

\bibitem{Kr95}
W.~B. Krebs, \emph{{B}rownian motion on a continuum tree}, Probab. Theory Rel.
  Fields \textbf{101} (1995), 421--433.

\bibitem{Ku04}
P.~Kuchment, \emph{{Q}uantum graphs {I}: {S}ome basic structures}, Waves in
  Random and Complex Media \textbf{14} (2004), S107--S128.

\bibitem{Mc69}
H.~P. McKean~Jr., \emph{{S}tochastic integrals}, Academic Press, New York and
  London, 1969.

\bibitem{ObBa73}
F.~Oberhettinger and L.~Badii, \emph{{T}ables of {L}aplace {T}ransforms},
  Springer-Verlag, Berlin, Heidelberg, New York, 1973.

\bibitem{ReYo91}
D.~Revuz and M.~Yor, \emph{{C}ontinuous {M}artingales and {B}rownian {M}otion},
  Springer--Verlag, Berlin, Heidelberg, New York, 1999.

\bibitem{Ro83}
L.~C.~G. Rogers, \emph{{I}t{\^o} excursion theory via resolvents}, Probab. Th.
  Rel. Fields \textbf{63} (1983), 237--255.

\bibitem{Sa86b}
Th.~S. Salisbury, \emph{{C}onstruction of right processes from excursions},
  Probab. Th. Rel. Fields \textbf{73} (1986), 351--367.

\bibitem{Sa86a}
\bysame, \emph{{O}n the {I}t{\^o} excursion process}, Probab. Th. Rel. Fields
  \textbf{73} (1986), 319--350.

\bibitem{We56}
A.~D. Wentzell, \emph{{S}emi-groups of operators corresponding to a generalized
  differential operator of second order}, Dokl. Akad. Nauk. SSSR (N.S.)
  \textbf{111} (1956), 269--272, (in Russian).

\bibitem{Wi79}
D.~Williams, \emph{{D}iffusions, {M}arkov {P}rocesses, and {M}artingales}, John
  Wiley {\&} Sons, Chichester, New York, Brisbane, Toronto, 1979.

\end{thebibliography}
\end{document}